\documentclass[12pt]{article}
\usepackage{amssymb,amsfonts,amsmath,amsthm,dsfont,xcolor,comment,stmaryrd,authblk,tikz,pgfplots,graphicx,float}
\usepackage[colorlinks, citecolor=purple,linkcolor=blue,urlcolor=teal]{hyperref}
\usepackage[lite,alphabetic]{amsrefs}
\usepackage{cleveref}
\usepackage{needspace}
\usepackage[a4paper, portrait, margin=25mm]{geometry}
\pgfplotsset{compat=1.18}
\newcommand{\1}{\mathds{1}}
\newcommand{\0}{\mathds{O}}

\newcommand{\Q}{\mathbb{Q}}
\newcommand{\R}{\mathbb{R}}

\newcommand{\N}{\mathbb{N}}

\newcommand{\Ba}{\mathrm{B}}
\newcommand{\Id}{\mathrm{Id}}

\newcommand{\So}{\mathrm{S}}

\newcommand{\8}{\infty}

\DeclareMathOperator{\spa}{span}

\DeclareMathOperator{\conv}{conv}
\DeclareMathOperator{\ext}{ext}

\DeclareMathOperator{\supp}{supp}
\DeclareMathOperator{\Aff}{Aff}
\newcommand{\Co}{\mathcal{C}}

\newcommand{\todo}{\xrightarrow[]{\mathrm{do}}}
\newcommand{\too}{\xrightarrow[]{\mathrm{o}}}

\newcommand{\toso}{\xrightarrow[]{\sigma\mathrm{o}}}

\newcommand{\toru}{\xrightarrow[]{\mathrm{ru}}}
\newcommand{\tow}{\xrightarrow[]{\mathrm{w}}}
\newcommand{\toww}{\xrightarrow[]{{\mathrm{w}^{*}}}}

\newcommand{\icecream}{%
  \mathord{%
    \tikz[baseline=-1.2ex,scale=0.13]{
      \draw (-1,0) -- (1,0);
      \draw (-1,0) arc (180:0:1);
      \draw (-1,0) -- (0,-2);
      \draw (1,0) -- (0,-2);
    }%
  }%
}

\newcommand{\smallicecream}{%
  \mathord{%
    \tikz[baseline=-0.7ex,scale=0.08]{
      \draw (-1,0) -- (1,0);
      \draw (-1,0) arc (180:0:1);
      \draw (-1,0) -- (0,-2);
      \draw (1,0) -- (0,-2);
    }%
  }%
}

\newtheoremstyle{mytheorem}
  {12pt}   % space above
  {4pt}   % space below
  {\itshape} % body font
  {}       % indentation
  {\bfseries} % theorem head font
  {.}      % punctuation after theorem head
  {.5em}   % space after theorem head
  {}       % theorem head specification

\newtheoremstyle{mydefinition}
  {12pt}        % space above
  {8pt}       % space below
  {\normalfont} % body font
  {}           % indentation
  {\bfseries}   % head font
  {.}           % punctuation
  {.5em}        % space after head
  {}

\newcounter{dummy} \numberwithin{dummy}{section}
\theoremstyle{mytheorem}
\newtheorem{theorem}[dummy]{Theorem}
\newtheorem{lemma}[dummy]{Lemma}

\newtheorem{proposition}[dummy]{Proposition}
\newtheorem{corollary}[dummy]{Corollary}
\newtheorem{question}[dummy]{Question}
\theoremstyle{mydefinition}
\newtheorem{remark}[dummy]{Remark}
\newtheorem{example}[dummy]{Example}
\usepackage{scalerel}

\title{Order convergence in ordered vector spaces: compactness, atoms, pervasiveness, and ice cream spaces}

\author[1,2]{Eugene Bilokopytov\footnote{bilokopy@ualberta.ca}}
\author[3]{Mark Roelands\footnote{m.roelands@math.leidenuniv.nl}}
\author[3]{Samuel Tiersma\footnote{s.j.tiersma@math.leidenuniv.nl}}
\author[4,5]{Jan Harm van der Walt\footnote{janharm.vanderwalt@up.ac.za}}
\author[3]{Onno van Gaans\footnote{vangaans@math.leidenuniv.nl}}
\author[4]{Marten Wortel\footnote{marten.wortel@up.ac.za}}

\affil[1]{Instituto de Ciencias Matem\'aticas, Nicol\'as Cabrera, 13-15 Campus de Cantoblanco, UAM, 28049, Madrid, Spain}
\affil[2]{Department of Mathematical and Statistical Sciences, University of Alberta, Edmonton, AB, T6G\,2G1, Canada}
\affil[3]{Mathematical Institute, Leiden University, 2300 RA Leiden, The Netherlands}
\affil[4]{Department of Mathematics and Applied Mathematics, University of Pretoria, Cor\-ner of Lynnwood Road and Roper Street, Hatfield 0083, Pretoria, South Africa}
\affil[5]{DSTI-NRF Centre of Excellence in Mathematical and Statistical Sciences (CoE-MaSS), South Africa}

\begin{document}
\maketitle

\begin{abstract}
We study order convergence in ordered vector spaces that are not necessarily lattices. In this setting, we consider two generalizations of order convergence on vector lattices, a directed and an undirected variant. We show that these notions agree for the class of pervasive spaces and extend several results concerning atomicity and order convergence in vector lattices to this class. In particular, we prove that a pervasive JB-algebra is associative, which strengthens a result of Sherman for C*-algebras. We also establish a number of new characterizations of pervasive spaces.

%We prove that a C*-algebra is commutative if and only if its self-adjoint elements form a pervasive ordered vector space, which strengthens a result of Sherman.
Additionally, we systematically study a class of order unit spaces associated to Banach spaces, called ice cream spaces. We provide descriptions of both directed and undirected order convergences in these spaces, yielding examples where the two notions differ. We also consider other order-theoretic aspects of ice cream spaces, such as atoms and lattice operations, relating them to the geometry of the underlying Banach spaces. We demonstrate that ice cream spaces exhibit behavior very different from vector lattices and can serve as a rich source of counterexamples in the theory of ordered vector spaces.\medskip



\end{abstract}

{\footnotesize\tableofcontents}

\section{Introduction}
The term ``order convergence'' refers to several similar but non-identical concepts which allow one to endow a partially ordered set with a topology-like structure. In analysis, order convergence has proved useful in the study of vector lattices and Boolean algebras. In particular, a striking fact is that in Banach function spaces order convergence of sequences agrees with dominated convergence almost everywhere, which shows that the latter convergence only depends on the order structure of the space.

While order convergence of nets is rarely induced by a topology, among the competing notions of order convergence in vector lattices, only one gives rise to a so-called convergence space. Many concepts and results about topological spaces remain valid for convergence spaces. Hence, the appropriate notion of order convergence enables the transfer of tools from the theory of topological vector spaces to the theory of vector lattices. Due to these advantages, in recent times this definition of order convergence has seen a nearly universal adoption by the community.

A net in an ordered vector space is declared to be order convergent to zero if its tails are dominated by a set with infimum zero. In vector lattices, this set can always be assumed to be downward directed. In the absence of a lattice ordering, this need no longer hold. Demanding that the set be downward directed leads to a new notion, called directed order convergence. The variant in which the controlling set is not required to be directed will simply be called order convergence, on the grounds that it is the restriction of order convergence on the Dedekind completion.

The starting point for this paper is to analyze the extent to which the two notions of order convergence are different. To this end, we study two classes of ordered vector spaces: a general class of so-called pervasive ordered vector spaces, which are close to vector lattices, and a more concrete class of ice cream spaces, whose order structure is very far from being a lattice. We show that in pervasive spaces the two modes of convergence actually agree and give a concrete description of both modes in certain ice cream spaces. In particular, this yields an example of an ordered vector space where directed order convergence is strictly stronger than order convergence.

The first collection of results in this paper concerns Archimedean pervasive ordered vector spaces. These are the ordered vector spaces which embed order densely (in the sense of vector lattice theory) into Archimedean vector lattices, see \cite{kv} and \cite{km}. We add a number of characterizations of Archimedean pervasive ordered vector spaces (Theorem~\ref{pervasive}) and show that order convergence and directed order convergence coincide on these spaces (Proposition~\ref{o=do}). We also prove in Theorem~\ref{ordprec} that the results of \cite{abt} extend to the setting of pervasive Archimedean ordered vector spaces. For such a space $E$, atomicity is equivalent to order pre-compactness of every order interval (i.e.\ to every order bounded set containing an order Cauchy subnet) and to the fact that order convergence is topological on order intervals. These conditions are also equivalent to the existence of a set $X$ such that $E$ embeds as a subspace of $\R^{X}$ containing $c_{00}(X)$. On order bounded subsets, order convergence then agrees with pointwise convergence.

For Archimedean ordered vector spaces, pervasiveness is closely related to being a vector lattice, and for this reason we intend to study the connection between pervasiveness and the lattice structure explicitly in the context of JB-algebras. These constitute a vast and rich family of ordered vector spaces, including the self-adjoint parts of C*-algebras. Each associative JB-algebra is isomorphic to $\Co_{0}\left(X\right)$, for a locally compact Hausdorff space $X$, see \cite[Theorem~3.2.2]{hs}. Therefore, associativity implies the structure of a vector lattice. The converse was proved in \cite[Theorem~1]{sherman} for the case of the self-adjoint part of a C*-algebra, and the more general result for JB-algebras can be found in \cite[Theorem~15]{ilr}. We add to the various characterizations of associativity the fact that a JB-algebra is associative if and only if it is pervasive (Theorem~\ref{jblat}).

The second collection of results in this paper revolves around the study of the ice cream space over a normed space $F$, given by
\[\icecream\left(F\right):=F\oplus\mathbb{R},
    \qquad
    \icecream\left(F\right)_+:=\{(f,r)\in F\oplus\mathbb{R}:\|f\|\le r\}.\]
The ice cream space $\icecream\left(F\right)$ can be identified with the space of $\mathrm{w}^{*}$-continuous affine functions on the dual unit ball $\Ba_{F^{*}}$ of $F$, which makes geometric properties of the Banach space $F$ correspond to order-theoretic properties of $\icecream\left(F\right)$.

We prove that directed order convergence, relative uniform convergence, and norm convergence coincide on $\icecream\left(F\right)$ (Proposition~\ref{icedo}). Furthermore, under suitable assumptions on $F$, order convergence agrees with bounded weak and bounded pointwise convergence, order intervals are order metrizable and order pre-compact, and they are order compact exactly when $F$ is reflexive (Theorem~\ref{iceord}). If in addition $F$ is infinite-dimensional, order convergence and directed order convergence are therefore distinct. Atomicity in ice cream spaces also admits a geometric description: $\icecream\left(F\right)$ is atomic if and only if every nonempty face of $\Ba_{F}$ contains an extreme point (Proposition~\ref{atm}).

The paper is organized as follows. Section~\ref{s2} fixes the terminology and records the general facts needed for the sequel. For partially ordered sets, order convergence is outlined in Subsection~\ref{ss21}; it is characterized by eventual upper and lower bounds and is related to the MacNeille completion. For ordered vector spaces, Subsections~\ref{ss22}--\ref{ss24} cover the terminology and properties of order Cauchy nets, relative uniform convergence, various notions of completeness, the Dedekind completion, and the conditions under which order convergence and directed order convergence are linear convergence structures.

The results concerning order pre-compactness in pervasive ordered vector spaces and the study of pervasive JB-algebras are collected in Section~\ref{s3}. We start with various characterizations of pervasiveness and the study of the Dedekind completion and completeness of pervasive spaces in Subsection~\ref{ss31}. Next, in Subsection~\ref{sec_pervatom}, atoms are introduced and we characterize atomicity through order pre-compactness and topological properties of order convergence. Subsection~\ref{ss33} covers the corresponding compactness questions for $\sigma$-order convergence and relative uniform convergence. In Subsection~\ref{ss34}, it is shown that pervasive JB-algebras are associative and lattice-ordered.

For dessert, the ice cream spaces are served in Section~\ref{s4}. The necessary preliminary facts about faces and order unit spaces are collected in Subsection~\ref{icepre}. Subsection~\ref{icebsc} gives the affine-function representations and basic convergence and completeness results for ice cream spaces. Subsection~\ref{iceoc} identifies order convergence in ice cream spaces under additional assumptions on the underlying normed space $F$. The extent to which lattice operations exist in ice cream spaces is examined in Subsection~\ref{icelat}. Subsection~\ref{iceatm} characterizes atomicity and strong atomicity (every positive vector is a supremum of atoms) of ice cream spaces and gives examples distinguishing these properties from order (pre-)compactness of order intervals.

We conclude the introduction by sharing the \LaTeX{} command for the $\icecream$ symbol (a separate command is used for the ice cream in subscripts or superscripts, with baseline \texttt{-0.7ex} and scale \texttt{0.08}):

\begin{verbatim}
\newcommand{\icecream}{\mathord{\tikz[baseline=-1.2ex,scale=0.13]{
\draw (-1,0) -- (1,0); \draw (-1,0) arc (180:0:1); 
\draw (-1,0) -- (0,-2); \draw (1,0) -- (0,-2);}}}
\end{verbatim}

\section{Order convergence}\label{s2}

\subsection{Order convergence in a poset}\label{ss21}

Let $P$ be a poset. Let us first introduce some notation. If $Q\subseteq P$, the sets of all upper and lower bounds of $Q$ are denoted by $Q^{\uparrow}$ and $Q^{\downarrow}$, respectively. If $Q$ is such that $\left[p,q\right]\subseteq Q$ whenever $p,q\in Q$ and $p\le q$, then we call $Q$ \emph{full}; we call it a \emph{lower} set if $p\le q\in Q$ implies $p\in Q$. \emph{Upper} sets are defined analogously. If $p\in P$, let $\left[p\right):=\left\{q\in P: q\ge p\right\}$, and $\left[p\right)_{Q}:=\left[p\right)\cap Q$; $\left(p\right]$ and $\left(p\right]_{Q}$ are defined analogously; $\left[p\right)$ and $\left(p\right]$ are called \emph{principal upper and lower sets}, respectively. If $Q,R\subseteq P$, we write $Q\le R$ if $q\le r$, for every $q\in Q$ and $r\in R$. For $p\le q$ we also define $\left[p,q\right]_{Q}:=\left[p,q\right]\cap Q$, $\left(p,q\right]:=\left[p,q\right]\setminus\left\{p\right\}$ and $\left[p,q\right):=\left[p,q\right]\setminus\left\{q\right\}$. The material of this section is not novel (see e.g.\ \cite{acw}), but we present it here in a coherent way for the reader's convenience. In particular, we fix definitions for the convergences that we will be studying. Note that throughout the literature the terms for these convergences are used in an inconsistent way.

A net $\left(p_{\alpha}\right)_{\alpha\in A}$ in $P$ \emph{directed order converges} to $p\in P$ (denoted $p_{\alpha}\todo p$) if there exist an increasing net $\left(q_{\beta}\right)_{\beta\in B}$ and a decreasing net $\left(r_{\beta}\right)_{\beta\in B}$ such that $\bigvee_{\beta\in B}q_{\beta}=p=\bigwedge_{\beta\in B}r_{\beta}$ and for every $\beta\in B$ there is $\alpha_{0}\in A$ such that $q_{\beta}\le p_{\alpha}\le r_{\beta}$, as soon as $\alpha\ge\alpha_{0}$. The term ``directed order convergence'' is justified by the following result.

\begin{proposition}\label{do}
$p_{\alpha}\todo p$ if and only if there are a directed upward $\varnothing\ne Q\subseteq P$ and a directed downward $\varnothing\ne R\subseteq P$ such that $\bigvee Q=p=\bigwedge R$ and for every $q\in Q$ and $r\in R$ there is $\alpha_{0}$ such that $p_{\alpha}\in\left[q,r\right]$, as soon as $\alpha\ge\alpha_{0}$.
\end{proposition}
\begin{proof}
For necessity it is enough to take $Q:=\left\{q_{\beta}\right\}_{\beta\in B}$ and $R:=\left\{r_{\beta}\right\}_{\beta\in B}$. For sufficiency, let $B:=Q\times R^{\mathrm{op}}$, and for $\beta=\left(q,r\right)\in B$ define $q_{\beta}:=q$, and $r_{\beta}:=r$.
\end{proof}

Motivated by this result we say that a net $\left(p_{\alpha}\right)_{\alpha\in A}\subseteq P$ converges \emph{in order} to $p\in P$ (denoted $p_{\alpha}\too p$) if there are $\varnothing\ne Q,R\subseteq P$ such that $\bigvee Q=p=\bigwedge R$ and for every $q\in Q$ and $r\in R$ there is $\alpha_{0}$ such that $p_{\alpha}\in\left[q,r\right]$, as soon as $\alpha\ge\alpha_{0}$. We will say that $Q,R$ in the definition \emph{witness} $p_{\alpha}\too p$. Similarly, in the setting of Proposition~\ref{do} we will say that $Q,R$ \emph{witness} $p_{\alpha}\todo p$.

Clearly, an order convergent net is eventually order bounded. Directed order convergence is \emph{stronger} than order convergence (denoted $\mathrm{do}\ge \mathrm{o}$ or $\mathrm{o}\le \mathrm{do}$), which means that directed order convergence of a net implies order convergence of this net to the same limit. It is easy to see that the two notions of convergence coincide if $P$ is a lattice. Indeed, in this case $Q,R$ in the definition of order convergence may be replaced with $Q^{\vee}:=\left\{q_{1}\vee\dots\vee q_{n}: q_{1},\dots, q_{n}\in Q,~n\in\N\right\}$ and $R^{\wedge}$ defined analogously. However, in general the two convergences may differ (see e.g.\ \cite[Example~3]{acw} and part~(iii) of Theorem~\ref{iceord}).

We call a map $\varphi\colon P\to Q$ between partially ordered sets \emph{(directed) order continuous} if it preserves (directed) order convergence of nets. We will say that $\varphi$ \emph{preserves all suprema} if $r=\bigvee R$ implies $\varphi\left(r\right)=\bigvee\varphi\left(R\right)$, for $r\in P$ and $\varnothing\ne R\subseteq P$ (and the same for infima). The next result follows directly from the definitions.

\begin{proposition}\label{oco}
Every order preserving map that preserves all suprema and infima (resp. extrema of monotone nets) is order (resp. directed order) continuous. The identity map is order- and directed order continuous from $\left(P,\le\right)$ into $\left(P,\ge\right)$.
\end{proposition}

We will say that $q$ is an \emph{eventual upper bound} for $\left(p_{\alpha}\right)_{\alpha\in A}$ if $\left(q\right]$ contains a tail of $\left(p_{\alpha}\right)_{\alpha\in A}$, that is, there is $\alpha_{0}$ such that $p_{\alpha}\le q$, for every $\alpha\ge\alpha_{0}$. Let $\left(p_{\alpha}\right)_{\alpha\in A}^{\nearrow}$ be the set of all eventual upper bounds for $\left(p_{\alpha}\right)_{\alpha\in A}$. Eventual lower bounds are defined analogously; their set is denoted $\left(p_{\alpha}\right)_{\alpha\in A}^{\searrow}$. Note that for any net we have $\left(p_{\alpha}\right)_{\alpha\in A}^{\searrow}\le \left(p_{\alpha}\right)_{\alpha\in A}^{\nearrow}$, but these sets may be empty. The eventual order bounds allow for the following characterization of order convergence.

\begin{proposition}\label{oc}
A net $\left(p_{\alpha}\right)_{\alpha\in A}\subseteq P$ order converges to $p\in P$ if and only if
\[\bigvee \left(p_{\alpha}\right)_{\alpha\in A}^{\searrow}=p=\bigwedge \left(p_{\alpha}\right)_{\alpha\in A}^{\nearrow}.\]
In particular, order limits are unique. Order intervals and principal upper and lower sets are closed with respect to order convergence.
\end{proposition}
\begin{proof}
Sufficiency: Take $Q:=\left(p_{\alpha}\right)_{\alpha\in A}^{\searrow}$ and $R:=\left(p_{\alpha}\right)_{\alpha\in A}^{\nearrow}$. We only need to show that they are nonempty. Indeed, if $Q=\varnothing$, then $p=\bigvee \varnothing$ implies that $p$ is the least element of $P$, and so in particular, $p\in Q$, contradiction. Similarly, $R\ne\varnothing$.

Necessity: Let $Q,R$ be as in the definition of order convergence. Since we have that $Q\subseteq \left(p_{\alpha}\right)_{\alpha\in A}^{\searrow}\le \left(p_{\alpha}\right)_{\alpha\in A}^{\nearrow}$, it follows that $Q\le \left(p_{\alpha}\right)_{\alpha\in A}^{\nearrow}$, and since $\bigvee Q=p$ we conclude that $p\le \left(p_{\alpha}\right)_{\alpha\in A}^{\nearrow}$. This together with $R\subseteq \left(p_{\alpha}\right)_{\alpha\in A}^{\nearrow}$ and $\bigwedge R=p$ yields $p= \bigwedge\left(p_{\alpha}\right)_{\alpha\in A}^{\nearrow}$. The other equality is proved similarly.

For the last claim assume that $\left(p_{\alpha}\right)_{\alpha\in A}\subseteq\left[q\right)$, for some $q\in P$. Then, $q\in \left(p_{\alpha}\right)_{\alpha\in A}^{\searrow}$, from where $p=\bigvee \left(p_{\alpha}\right)_{\alpha\in A}^{\searrow}\ge q$. Hence, $\left[q\right)$ is order closed, and similarly $\left(q\right]$ is also order closed. For $p\le q$ we have that $\left[p,q\right]=\left[p\right)\cap \left(q\right]$ is order closed as an intersection of order closed sets.
\end{proof}

\begin{corollary}
If $\left(p_{\alpha}\right)_{\alpha\in A}\subseteq P$ is increasing, then $p=\bigvee p_{\alpha}$ if and only if $p_{\alpha}\todo p$ if and only if $p_{\alpha}\too p$ (and the same for decreasing nets and infimum).
\end{corollary}
\begin{proof}
The direct implications are straightforward. Assume that $p_{\alpha}\too p$. For every $\alpha\in A$ we have that $p_{\alpha}\le p_{\beta}$, for each $\beta\ge\alpha$, and as the tail $\left(p_{\beta}\right)_{\beta\ge\alpha}\subseteq\left[p_{\alpha}\right)$ order converges to $p$, by Proposition~\ref{oc} we have $p\in \left[p_{\alpha}\right)$. Hence, $p\ge p_{\alpha}$, for every $\alpha\in A$. Assume that $q\ge p_{\alpha}$, for every $\alpha\in A$. Then, by Proposition~\ref{oc}, $\left(q\right]\ni p_{\alpha}\too p$ yields $p\le q$. Hence, $p=\bigvee p_{\alpha}$.
\end{proof}

Since directed order convergence is stronger than order convergence, Proposition~\ref{oc} has an immediate consequence.

\begin{proposition}
A directed order limit is unique. Order intervals and principal upper and lower sets are directed order closed.
\end{proposition}

If $R\subseteq Q\subseteq P$, then $\bigwedge_{Q}R$ and $\bigvee_{Q} R$ are the infimum and supremum of $R$ taken in the partially ordered set $Q$. In general, they may differ from $\bigwedge_{P}R$ and $\bigvee_{P} R$, respectively; in particular they exist independently from one another. Note, however, that if $q\in Q$ is such that $q=\bigwedge_{P}R$, then it also follows that $q=\bigwedge_{Q}R$. Similarly, it is not always true that order convergence in $Q$ is the restriction of order convergence in $P$. Let us consider the case when this in fact happens. We say that $Q$ is \emph{supremum-dense} in $P$ if $p=\bigvee\left(p\right]_{Q}$ for every $p\in P$, or equivalently, if for every $p\in P$ there is $R\subseteq Q$ with $p=\bigvee R$. \emph{Infimum-denseness} is defined analogously.

\begin{proposition}\label{ocp}
If $Q\subseteq P$ is supremum- and infimum-dense, then the embedding from $Q$ into $P$ preserves all suprema and infima and is directed order continuous. Moreover, order convergence in $Q$ coincides with order convergence in $P$ restricted to $Q$.
\end{proposition}
\begin{proof}
First, let us prove that the inclusion map from $Q$ into $P$ preserves all suprema and infima. Assume that $r\in Q$ and $R\subseteq Q$ are such that $r=\bigvee_{Q} R$. If $p\in P$ is such that $p\ge R$, then for every $q\in \left[p\right)_{Q}$ we have $R\le p\le q$, therefore $r\le q$; hence $\left[p\right)_{Q}\ge r$, from where $p=\bigwedge\left[p\right)_{Q}\ge r$. As $p$ was chosen arbitrarily, we conclude that $r=\bigvee_{P} R$. Similarly, if $r=\bigwedge_{Q} R$, then $r=\bigwedge_{P} R$. By virtue of Proposition~\ref{oco} we conclude that the inclusion map is order- and directed order continuous.

On the other hand, assume that $\left(q_{\alpha}\right)_{\alpha\in A}\subseteq Q$ order converges in $P$ to $q\in Q$, that is, $\bigvee \left(q_{\alpha}\right)_{\alpha\in A}^{\searrow_{P}}=q=\bigwedge \left(q_{\alpha}\right)_{\alpha\in A}^{\nearrow_{P}}$. For every $p\in \left(q_{\alpha}\right)_{\alpha\in A}^{\searrow_{P}}$ we have that $\left(p\right]_{Q}\subseteq \left(q_{\alpha}\right)_{\alpha\in A}^{\searrow_{Q}}$ and $p=\bigvee\left(p\right]_{Q}$. It follows that
\[\left(q_{\alpha}\right)_{\alpha\in A}^{\searrow_{Q}}=\bigcup_{p\in \left(q_{\alpha}\right)_{\alpha\in A}^{\searrow_{P}}}\left(p\right]_{Q},\quad\text{and so}\quad \bigvee \left(q_{\alpha}\right)_{\alpha\in A}^{\searrow_{Q}}=\bigvee_{p\in \left(q_{\alpha}\right)_{\alpha\in A}^{\searrow_{P}}}\left(p\right]_{Q}=q\]
(supremum in $P$, hence in $Q$), and similarly $q=\bigwedge \left(q_{\alpha}\right)_{\alpha\in A}^{\nearrow_{Q}}$. Thus, $q_{\alpha}\too q$ in $Q$.
\end{proof}

We say that $P$ is a \emph{complete lattice} if every subset of $P$ has a supremum (equivalently, infimum). In particular, $\bigvee P$ is the greatest element of $P$, while $\bigvee\varnothing$ is the least element. A weaker condition is \emph{Dedekind completeness}: every nonempty subset of $P$ which is bounded from above has a supremum (equivalently, every nonempty subset of $P$ which is bounded from below has an infimum). Note that an upward and downward directed Dedekind complete poset is always a lattice. We also call $P$ \emph{monotone complete} if every increasing net which is bounded from above has a supremum, and the same for infimum.

If $P$ is a Dedekind complete lattice and $\left(p_{\alpha}\right)_{\alpha\in A}\subseteq P$ is bounded from above, it is easy to see that
\[\left(p_{\alpha}\right)_{\alpha\in A}^{\nearrow}=\bigcup_{\alpha}\left[\bigvee_{\beta\ge\alpha}p_{\beta}\right)\]
(and dually for the eventual lower bounds), which allows us to recover the following well-known fact from Proposition~\ref{oc}.

\begin{corollary}\label{dcoc}
If $P$ is a Dedekind complete lattice, then a net $\left(p_{\alpha}\right)_{\alpha\in A}$ order converges to $p\in P$ if and only if there is $\alpha_{0}$ such that $\left\{p_{\alpha}\right\}_{\alpha\ge\alpha_{0}}$ is order bounded and
\[\bigwedge_{\alpha\ge\alpha_{0}}\bigvee_{\beta\ge\alpha}p_{\beta}=p=\bigvee_{\alpha\ge\alpha_{0}}\bigwedge_{\beta\ge\alpha}p_{\beta}.\]
\end{corollary}

Recall that the \emph{MacNeille completion}\footnote{The other terms used in order theory are ``Dedekind completion'' and ``Dedekind--MacNeille completion'', and the usual notation would be $P^{\delta}$. However, we will reserve this notation and the term ``Dedekind completion'' for the case of a directed Archimedean ordered vector space. Note that the two completions are distinct aside from the trivial case: $P^{\partial}$ has a least and a greatest element; in the case of a vector space we get a Dedekind complete vector lattice, which has neither a greatest nor a least element.} $P^{\partial}$ of $P$ is the collection of all \emph{(lower) Dedekind cuts} on $P$, i.e.\ subsets $Q\subseteq P$ such that $Q=Q^{\uparrow\downarrow}$. It is the essentially unique complete lattice containing $P$ as a supremum- and infimum-dense subset (see \cite[Theorem~7.41]{dp}). The next result therefore is an immediate consequence of this definition and Proposition~\ref{ocp}.

\begin{corollary}\label{macneille}
Order convergence in $P$ is the restriction of order convergence in $P^{\partial}$.
\end{corollary}

Note that $P$ embeds into $P^{\partial}$ via $p\mapsto\left(p\right]$. The intersection of a collection of elements in $P^{\partial}$ (each of which is a subset of $P$) is their infimum. If $\mathcal{Q}$ is a collection of elements of $P^{\partial}$, then its supremum is $\left(\bigcup \mathcal{Q}\right)^{\uparrow\downarrow}$.

Also note that $\left(P^{\partial}\right)^{\mathrm{op}}$ can be identified with the collection of all \emph{upper Dedekind cuts} on $P$, i.e.\ subsets $Q\subseteq P$ such that $Q=Q^{\downarrow\uparrow}$. In particular, $Q\mapsto Q^{\uparrow}$ is an order reversing map from $P^{\partial}$ into $\left(P^{\partial}\right)^{\mathrm{op}}$, whose inverse is $R\mapsto R^{\downarrow}$. Most of the aforementioned results rely on the fact that if $R\subseteq P$, then $R^{\uparrow\downarrow\uparrow}=R^{\uparrow}$ and $R^{\downarrow\uparrow\downarrow}=R^{\downarrow}$.

We proceed with a characterization of order convergence in $P^{\partial}$. Condition (iii) in the following result may be viewed as the definition of an order Cauchy net in $P$ (cf.\ Proposition~\ref{ocauchy}). Recall that elements of $P^{\partial}$ are subsets of $P$, but we also view $P\subseteq P^{\partial}$.

\begin{proposition}
If $\left(p_{\alpha}\right)_{\alpha\in A}\subseteq P$, then:
\item[(i)] $\bigwedge_{\alpha}\bigvee_{\beta\ge\alpha}p_{\beta}=\left(p_{\alpha}\right)_{\alpha\in A}^{\nearrow\downarrow}$ in $P^{\partial}$.
\item[(ii)] $\bigwedge_{\alpha}\bigvee_{\beta\ge\alpha}p_{\beta}=\left(p_{\alpha}\right)_{\alpha\in A}^{\searrow\uparrow}$ in $\left(P^{\partial}\right)^{\mathrm{op}}$.
\item[(iii)] There is an order limit of $\left(p_{\alpha}\right)_{\alpha\in A}$ in $P^{\partial}$ if and only if there is one in $\left(P^{\partial}\right)^{\mathrm{op}}$ if and only if $\left(p_{\alpha}\right)_{\alpha\in A}^{\nearrow\downarrow}=\left(p_{\alpha}\right)_{\alpha\in A}^{\searrow\uparrow\downarrow}$ if and only if $\left(p_{\alpha}\right)_{\alpha\in A}^{\nearrow\downarrow\uparrow}=\left(p_{\alpha}\right)_{\alpha\in A}^{\searrow\uparrow}$.
\end{proposition}
\begin{proof}
(i): Let $J$ be the embedding from $P$ into $P^{\partial}$ via $p\mapsto\left(p\right]$, and let $I\colon P\to \left(P^{\partial}\right)^{\mathrm{op}}$ be defined via $p\mapsto\left[p\right)$. Note that $I$ is order reversing. For every $\alpha$ we have that $\bigwedge_{\beta\ge\alpha}Jp_{\beta}=\bigcap_{\beta\ge\alpha}\left(p_{\beta}\right]$. Hence, in $P^{\partial}$ we have
\[\bigvee_{\alpha}\bigwedge_{\beta\ge\alpha}Jp_{\beta}=\left(\bigcup_{\alpha}\bigcap_{\beta\ge\alpha}\left(p_{\beta}\right]\right)^{\uparrow\downarrow}=\left(p_{\alpha}\right)_{\alpha\in A}^{\searrow\uparrow\downarrow}.\]
Similarly, $\bigvee_{\alpha}\bigwedge_{\beta\ge\alpha}Ip_{\beta}=\left(p_{\alpha}\right)_{\alpha\in A}^{\nearrow\downarrow\uparrow}$.
Now
\[\bigwedge_{\alpha}\bigvee_{\beta\ge\alpha}Jp_{\beta}=\left(\bigvee_{\alpha}\bigwedge_{\beta\ge\alpha}Ip_{\beta}\right)^{\downarrow}=\left(p_{\alpha}\right)_{\alpha\in A}^{\nearrow\downarrow\uparrow\downarrow}=\left(p_{\alpha}\right)_{\alpha\in A}^{\nearrow\downarrow}.\]

(ii) is proved similarly.

(iii): It follows from Corollary~\ref{dcoc} that there is an order limit of $\left(p_{\alpha}\right)_{\alpha\in A}$ in $P^{\partial}$ if and only if
\[\bigwedge_{\alpha}\bigvee_{\beta\ge\alpha}p_{\beta}=\bigvee_{\alpha}\bigwedge_{\beta\ge\alpha}p_{\beta},\]
which is equivalent to $\left(p_{\alpha}\right)_{\alpha\in A}^{\nearrow\downarrow}=\left(p_{\alpha}\right)_{\alpha\in A}^{\searrow\uparrow\downarrow}$, and also equivalent to $\left(p_{\alpha}\right)_{\alpha\in A}^{\nearrow\downarrow\uparrow}=\left(p_{\alpha}\right)_{\alpha\in A}^{\searrow\uparrow}$.
\end{proof}

We conclude this section with an example that demonstrates that order and directed order convergence can behave quite pathologically in general posets (cf.\ Proposition~\ref{ocl} and the discussion above it).

\begin{example}\label{pat}
Let $P=\left\{\left(0,t\right): 0\le t\le 1\right\}\cup \left\{\left(t,0\right): 0\le t\le 1\right\}\cup\left\{\left(1,1\right)\right\}\subseteq\R^{2}$. For $n\in\N$ define $p_{n}:=\left(0,\frac{1}{n}\right)$ and $q_{n}:=\left(\frac{1}{n},0\right)$. We have $p_{n},q_{n}\downarrow \left(0,0\right)$, hence $p_{n},q_{n}\too \left(0,0\right)$ and $p_{n},q_{n}\todo \left(0,0\right)$, but the only eventual upper bound for $p_{1},q_{1},p_{2},q_{2},\dots$ is $\left(1,1\right)$, which means that the sequence does not order converge (and hence does not directed order converge) to $\left(0,0\right)$.
\qed\end{example}

\subsection{Order convergence in an ordered vector space}\label{ss22}

In this section, $E$ is an ordered vector space with positive cone $E_{+}$. Since $E$ is a poset, it carries both order and directed order convergences. These convergences have no a priori relation to the linear structure of $E$. In order to describe these relations we need to consider some preliminary facts.

\begin{lemma}\label{sum}
If $P,Q\subseteq E$ with $g=\bigwedge P$ and $h=\bigwedge Q$, then $g+h=\bigwedge\left(P+Q\right)$.
\end{lemma}
\begin{proof}
It is easy to see that $g+h\le P+Q$. If $e\in E$ is such that $e\le P+Q$, then, for every $p\in P$ and $q\in Q$ we have that $e-p\le q$. Hence, $e-p\le Q$, from where $e-p\le h$, or $e-h\le P$. Thus, $e-h\le g$, which completes the proof.
\end{proof}

We also note that if $P$ and $Q$ are directed downwards, then so is $P+Q$.

\needspace{6\baselineskip}
\begin{proposition}\label{ocv}
\item[(i)] A net $\left(e_{\alpha}\right)_{\alpha\in A}\subseteq E$ converges in order to $e\in E$ if and only if there is $P\subseteq E$ with $\bigwedge P=0_{E}$ which dominates the tails of $\left(e_{\alpha}-e\right)_{\alpha\in A}$, i.e.\ for each $p\in P$ there is $\alpha_{0}$ such that $-p\le e_{\alpha}-e\le p$, as soon as $\alpha\ge \alpha_{0}$.
\item[(ii)] Addition and taking the additive inverse are order continuous operations on $E$.
\end{proposition}
\begin{proof}
(i): $e_{\alpha}\too e$ if and only if there are $Q,R\subseteq E$ with $\bigvee Q=e=\bigwedge R$, and such that for every $q\in Q$ and $r\in R$ there is $\alpha_{0}$ such that $\left\{e_{\alpha}\right\}_{\alpha\ge\alpha_{0}}\subseteq\left[q,r\right]$. Note that if $f\in\left[q,r\right]$, then $q-r\le f-e\le r-q$. Hence, $P:=R-Q$ dominates the tails of $\left(e_{\alpha}-e\right)_{\alpha\in A}$, and satisfies $\bigwedge P=0_{E}$ by Lemma~\ref{sum}.

Conversely, if $P$ dominates the tails of $\left(e_{\alpha}-e\right)_{\alpha\in A}$, let $Q:=e-P$ and $R:=e+P$. Clearly, $\bigvee Q=e=\bigwedge R$. If $q\in Q$ and $r\in R$, there are $g,h\in P$ such that $q=e-g$ and $r=e+h$. Find $\alpha_{1},\alpha_{2}$ such that $-g\le e_{\alpha}-e\le g$, for every $\alpha\ge \alpha_{1}$, and $-h\le e_{\alpha}-e\le h$, for every $\alpha\ge \alpha_{2}$. Take $\alpha_{0}\ge\alpha_{1},\alpha_{2}$. For every $\alpha\ge\alpha_{0}$ we have $q=e-g\le e_{\alpha}\le e+h=r$.

(ii): Assume that $e_{\alpha}\too e$ and $f_{\alpha}\too f$, which is witnessed by $P$ and $Q$, respectively. Using Lemma~\ref{sum}, it is easy to show that $P+Q$ witnesses $e_{\alpha}+f_{\alpha}\too e+f$. Order continuity of the map $e\mapsto -e$ is straightforward.
\end{proof}

A net $\left(e_{\alpha}\right)_{\alpha\in A}\subseteq E$ is \emph{(directed) order Cauchy} if the doubly indexed net $\left(e_{\alpha}-e_{\beta}\right)_{\alpha,\beta\in A}$ is (directed) order-null. It is clear that any order Cauchy net is eventually order bounded. Order Cauchy nets can be characterized as follows.

\begin{proposition}\label{ocauchy}
For a net $\left(e_{\alpha}\right)_{\alpha\in A}\subseteq E$ the following conditions are equivalent:
\item[(i)] $\left(e_{\alpha}\right)_{\alpha\in A}$ is order Cauchy;
\item[(ii)] There are $Q,R\subseteq E$ with $\bigwedge \left(R-Q\right)=0_{E}$ and such that for all $q\in Q$ and $r\in R$ there is $\alpha_{0}$ such that $e_{\alpha}\in\left[q,r\right]$, as soon as $\alpha\ge \alpha_{0}$;
\item[(iii)] $\bigwedge\left(\left(e_{\alpha}\right)_{\alpha\in A}^{\nearrow}-\left(e_{\alpha}\right)_{\alpha\in A}^{\searrow}\right)=0_{E}$.
\end{proposition}
\begin{proof}
(i)$\Rightarrow$(ii): Let $P\subseteq E$ with $\bigwedge P=0_{E}$ dominate the tails of $\left(e_{\alpha}-e_{\beta}\right)_{\alpha,\beta\in A}$, so that for every $p\in P$ we can choose $\alpha_{p}\in A$ such that $-p\le e_{\alpha}-e_{\beta}\le p$, as soon as $\alpha,\beta\ge\alpha_{p}$. For all $p,q\in P$ we can find $\alpha\ge \alpha_{p},\alpha_{q}$; then $e_{\alpha_{p}}-e_{\alpha_{q}}=e_{\alpha_{p}}-e_{\alpha}+e_{\alpha}-e_{\alpha_{q}}\le p+q$, from where $e_{\alpha_{p}}-p\le e_{\alpha_{q}}+q$. Then, $Q:=\left\{e_{\alpha_{p}}-p:p\in P\right\}$ and $R:=\left\{e_{\alpha_{p}}+p:p\in P\right\}$ satisfy $\bigwedge \left(R-Q\right)=0_{E}$. Moreover, for $p,q\in P$ and $\alpha\ge\alpha_{p},\alpha_{q}$ we have $e_{\alpha_{p}}-p\le e_{\alpha}\le e_{\alpha_{q}}+q$, hence $Q, R$ are as required.

(ii)$\Rightarrow$(i) is proved similarly to Proposition~\ref{ocv}. (ii)$\Leftrightarrow$(iii) is proved similarly to Proposition~\ref{oc}.
\end{proof}

Analogously to Propositions~\ref{ocv} and~\ref{ocauchy}, one can prove the following.

\begin{proposition}\label{todovs}
\item[(i)] A net $\left(e_{\alpha}\right)_{\alpha\in A}\subseteq E$ directed order converges to $e\in E$ if and only if there is a downward directed $P\subseteq E$ with $\bigwedge P=0_{E}$ which dominates the tails of $\left(e_{\alpha}-e\right)_{\alpha\in A}$.
\item[(ii)] If for a net $\left(e_{\alpha}\right)_{\alpha\in A}\subseteq E$ there are an upward directed $Q\subseteq E$ and a downward directed $R\subseteq E$ with $\bigwedge \left(R-Q\right)=0_{E}$ such that for all $q\in Q$ and $r\in R$ there is $\alpha_{0}$ such that $e_{\alpha}\in\left[q,r\right]$, as soon as $\alpha\ge \alpha_{0}$, then this net is directed order Cauchy.
\item[(iii)] Addition and taking the additive inverse are directed order continuous on $E$.
\end{proposition}

The following result is a special case of Proposition~\ref{oco}.

\begin{proposition}\label{docont}
If $E$ and $F$ are ordered vector spaces, then a positive linear map $T\colon E\to F$ is directed order continuous if and only if for any $P\subseteq E$ with $P\downarrow 0_{E}$ it follows that $TP\downarrow 0_{F}$.
\end{proposition}

\subsection{Order convergence in an Archimedean ordered vector space}\label{ss23}

Recall that $E$ is called \emph{Archimedean} if $ne\le f$, for every $n\in\N$, implies $e\le 0_{E}$. By replacing $f$ with $f-e$ if needed, only positive $f$ need to be checked, and so $E$ is Archimedean if and only if $\bigwedge_{n\in\N}\frac{1}{n}f=0_{E}$ for every $f\in E_{+}$. It follows that any linear subspace of an Archimedean ordered vector space is Archimedean. In some cases, the Archimedean property passes up, as we show next. For the proof of the following result we note that a linear subspace $F\subseteq E$ is supremum-dense in $E$ if and only if it is infimum-dense.

\begin{proposition}\label{arch0}
If an ordered vector space $F$ contains an Archimedean ordered vector space $E$ as a supremum-dense subspace, then $F$ is Archimedean.
\end{proposition}
\begin{proof}
Let $f\in F_{+}$. There is $e\in E$ such that $e\ge f$. Assume that $h\in F$ is such that $h\le \frac{1}{n}f$, for every $n\in\N$. Then, for every $g\in E$ such that $g\le h$ and $n\in\N$ we have $g\le h\le \frac{1}{n}f\le \frac{1}{n}e$. Since $E$ is Archimedean, it follows that $g\le 0_{E}$. Then, $h=\bigvee\left(h\right]_{E}\le 0_{F}$, which shows that $\bigwedge_{n\in\N}\frac{1}{n}f=0_{F}$.
\end{proof}

\begin{lemma}\label{arch}Assume that $E$ is Archimedean and $\varnothing\ne Q\subseteq E$ is bounded from below. Then:
\item[(i)] If $f\in E$ is such that $Q-f\subseteq Q$, then $f\le0_{E}$.
\item[(ii)] $\bigwedge \left(Q-Q^{\downarrow}\right)=0_{E}$.
\end{lemma}
\begin{proof}
First, by shifting $Q$ if needed we may assume that $Q\subseteq E_{+}$.

(i): Since $Q$ is nonempty, there is $g\in Q$. Then, $g-f\in Q$. Arguing by induction, it follows that $g-nf\in Q\subseteq E_{+}$, hence
$g\ge nf$, for every $n\in\N$. As $E$ is Archimedean, we conclude that $f\le 0_{E}$.

(ii): First, $0_{E}\in Q^{\downarrow}$, and so $Q^{\downarrow}\ne\varnothing$; moreover $Q\ge Q^{\downarrow}$, hence $Q- Q^{\downarrow}\subseteq E_{+}$. Assume that $e\in E$ is such that $Q-Q^{\downarrow}\ge e$. Then, for every $g\in Q$ and $h\in Q^{\downarrow}$ it follows that $g-h\ge e$, from where $g\ge h+e$, and so $h+e\in Q^{\downarrow}$. Hence, $Q^{\downarrow}+e\subseteq Q^{\downarrow}$. Since $Q^{\downarrow}$ is bounded from above, it is easy to deduce from (i) that $e\le 0_{E}$, and so the claim follows.
\end{proof}

Recall that an ordered vector space $E$ is directed if and only if $E=E_{+}-E_{+}$ (see \cite[Proposition~1.1.7]{kv}). If in this case $F\subseteq E$ is a \emph{majorizing} subspace (i.e.\ such that for every $e\in E$ there is $f\in F$ with $f\ge e$), then $F$ is also directed. Any directed Archimedean vector space has a \emph{Dedekind completion} $E^{\delta}$, which is the essentially unique Dedekind complete vector lattice which contains $E$ as a supremum- and infimum-dense subspace (see \cite[Section~2.1]{kv}). If $E\ne \left\{0_{E}\right\}$, it is constructed by removing the top and the bottom elements from $E^{\partial}$ and introducing a linear structure on the remainder.

A net $\left(e_{\alpha}\right)_{\alpha\in A}\subseteq E$ converges \emph{relatively uniformly} to $e\in E$ (we denote this by $e_{\alpha}\xrightarrow[]{\mathrm{ru}}e$) if there is $h\in E_{+}$ such that for every $\varepsilon>0$ there is $\alpha_{0}$ such that $-\varepsilon h\le e-e_{\alpha}\le\varepsilon h$, as soon as $\alpha\ge\alpha_{0}$. In this case $h$ is called a \emph{regulator} of convergence.

\begin{proposition}\label{rucauchy}
A net $\left(e_{\alpha}\right)_{\alpha\in A}\subseteq E$ is relatively uniformly Cauchy if and only if there are $h\in E_{+}$, and sequences $\left(f_{n}\right)_{n\in\N},\left(g_{n}\right)_{n\in\N}\subseteq E$ and $\left(\alpha_{n}\right)_{n\in\N}\subseteq A$ such that $f_{n}\le f_{n+1}$, $g_{n+1}\le g_{n}$, $g_{n}-f_{n}\le\frac{1}{n}h$, and $f_{n}\le e_{\alpha}\le g_{n}$, for every $n\in\N$ and $\alpha\ge\alpha_{n}$.
\end{proposition}
\begin{proof}
It is easy to see that if the conditions in the proposition are satisfied, then the net $\left(e_{\alpha}-e_{\beta}\right)_{\alpha,\beta\in A}$ is relatively uniformly null with regulator $h$. Hence, we only need to prove the converse. For every $n\in\N$ there is $\alpha_{n}\in A$ such that $-\frac{1}{2^{n}}h\le e_{\alpha}-e_{\beta}\le \frac{1}{2^{n}}h$, for every $\alpha,\beta\ge\alpha_{n}$. Furthermore, we may assume that $\left(\alpha_{n}\right)_{n\in\N}$ is increasing.

Fix $n\in\N$ and define $f_{n}:=e_{\alpha_{n}}-\frac{1}{2^{n-1}}h$ and $g_{n}:=e_{\alpha_{n}}+\frac{1}{2^{n-1}}h$. For $\alpha\ge\alpha_{n}$ we have $-\frac{1}{2^{n}}h\le e_{\alpha}-e_{\alpha_{n}}\le \frac{1}{2^{n}}h$, hence $f_{n}\le e_{\alpha_{n}}-\frac{1}{2^{n}}h\le e_{\alpha}\le e_{\alpha_{n}}+\frac{1}{2^{n}}h\le g_{n}$. In particular, $f_{n}+\frac{1}{2^{n}}h=e_{\alpha_{n}}-\frac{1}{2^{n}}h\le e_{\alpha_{n+1}}$, hence $f_{n}\le e_{\alpha_{n+1}}-\frac{1}{2^{n}}h=f_{n+1}$ and similarly $g_{n+1}\le g_{n}$. Finally, $g_{n}-f_{n}=\frac{1}{2^{n-2}}h$, and so with some relabelling it is easy to achieve the desired condition.
\end{proof}

Condition (ii) in the following result for the case of vector lattices appeared in \cite[Lemma~2.2]{aeg}.

\begin{proposition}\label{oca}
The following conditions are equivalent:
\item[(i)] $E$ is Archimedean;
\item[(ii)] Directed order convergence is weaker than relative uniform convergence;
\item[(iii)] Order convergence is weaker than relative uniform convergence;
\item[(iv)] Any increasing bounded from above net is order Cauchy.\medskip

If, additionally, $E$ is directed, then order convergence in $E$ is the restriction of order convergence in $E^{\delta}$.
\end{proposition}
\begin{proof}
(i)$\Rightarrow$(ii) follows from the fact that if $h$ is a regulator of convergence for a net, then $\left\{\frac{1}{n}h\right\}_{n\in\N}$ witnesses directed order convergence of this net. (ii)$\Rightarrow$(iii) follows from $\mathrm{o}\le\mathrm{do}$. (iii)$\Rightarrow$(i) follows from the facts that $\left(\frac{1}{n}h\right)_{n\in\N}$ is relatively uniformly null, for every $h\in E_{+}$, and that the order limit of a decreasing net is its infimum.

(i)$\Rightarrow$(iv): Let $\left(e_{\alpha}\right)_{\alpha\in A}$ be an order bounded increasing net in $E$, let $Q:=\left\{e_{\alpha}\right\}_{\alpha\in A}$ and let $R:=Q^{\uparrow}$. It can be deduced from part~(ii) of Lemma~\ref{arch} that $\bigwedge\left(R-Q\right)=0_{E}$. Hence, according to Proposition~\ref{ocauchy}, the net $\left(e_{\alpha}\right)_{\alpha\in A}$ is order Cauchy.

(iv)$\Rightarrow$(i): Fix $h\in E_{+}$ and consider the order bounded decreasing sequence $\left(\frac{1}{n}h\right)_{n\in\N}$, which is order Cauchy by (iv). It follows that there is $P\subseteq E$ such that $\bigwedge P=0_{E}$ and for every $p\in P$ there is $n_{p}$ so that $\left|\frac{1}{n}h-\frac{1}{m}h\right|\le p$, for all $m,n\ge n_{p}$. In particular, taking $n:=n_{p}$ and $m:=2n_{p}$ yields $\frac{1}{2n_{p}}h\le p$. We now conclude from $\bigwedge P=0_{E}$ that $\bigwedge_{n\in\N}\frac{1}{n}h=0_{E}$. Thus, $E$ is Archimedean.

The last claim follows immediately from Proposition~\ref{ocp}.
\end{proof}

Using the last claim of Proposition~\ref{oca} and the fact that Dedekind complete vector lattices are complete with respect to order convergence (see \cite[Proposition~9.10]{btw})\footnote{In fact, a vector lattice is Dedekind complete if and only if it is Archimedean and complete with respect to order convergence. For this reason some authors use ``order complete'' and ``Dedekind complete'' interchangeably. However, the mentioned equivalence does not hold for directed Archimedean ordered vector spaces (see Corollary~\ref{iceord1}), and so we will not conflate these terms.}, we draw the following conclusion.

\begin{corollary}
A directed Archimedean ordered vector space $E$ is complete with respect to order convergence if and only if it is closed with respect to order convergence in $E^{\delta}$.
\end{corollary}

We now consider the relationships between monotone completeness and other properties.

\begin{corollary}\label{monc}
\item[(i)] Every monotone complete ordered vector space is Archimedean and relatively uniformly complete. In such spaces any net which satisfies the conditions of part~(ii) of Proposition~\ref{todovs} is directed order convergent.
\item[(ii)] Every Archimedean ordered vector space which is complete with respect to order convergence is monotone complete and directed order complete.
\end{corollary}
\begin{proof}
(i): Every monotone order bounded net has an extremum, hence is order convergent, and so order Cauchy. This implies the Archimedean property by Proposition~\ref{oca}.

Assume that $\left(e_{\alpha}\right)_{\alpha\in A}\subseteq E$ is relatively uniformly Cauchy and let $h\in E_{+}$, $\left(f_{n}\right)_{n\in\N}$,\ $\left(g_{n}\right)_{n\in\N}\subseteq E$ and $\left(\alpha_{n}\right)_{n\in\N}\subseteq A$ be as in Proposition~\ref{rucauchy}. Since $\left(f_{n}\right)_{n\in\N},\left(g_{n}\right)_{n\in\N}$ are monotone and order bounded, there are $f,g\in E$ such that $f_{n}\uparrow f$ and $g_{n}\downarrow g$. It follows from the Archimedean property that $f=g=:e$, and for every $n\in\N$ and $\alpha\ge\alpha_{n}$ we have $e_{\alpha}-e\le g_{n}-f_{n}\le \frac{1}{n}h$, and similarly $e_{\alpha}-e\ge -\frac{1}{n}h$, thus justifying $e_{\alpha}\xrightarrow[]{\mathrm{ru}}e$.

Assume that $\left(e_{\alpha}\right)_{\alpha\in A}\subseteq E$ satisfies the conditions of part~(ii) of Proposition~\ref{todovs}. Let $Q,R$ be the corresponding controlling sets. Since $Q$ is upward directed and bounded from above, we may set $q_{0}:=\bigvee Q$. Similarly, we may set $r_{0}:=\bigwedge R$. Then, we have $r_{0}-q_{0}=\bigwedge\left(R-Q\right)=0_{E}$, and so $q_{0}=r_{0}=:e$. It then follows from Proposition~\ref{do} that $e_{\alpha}\todo e$.

(ii): By part~(iv) of Proposition~\ref{oca} every monotone order bounded net is order Cauchy, hence order convergent, thus has an extremum.

Assume that $\left(e_{\alpha}\right)_{\alpha\in A}\subseteq E$ is directed order Cauchy, so that $e_{\alpha}-e_{\beta}\todo 0_{E}$. Since $\mathrm{do}\ge\mathrm{o}$, we have $e_{\alpha}-e_{\beta}\too 0_{E}$, and as $E$ is complete with respect to order convergence, there is $e\in E$ such that $e_{\alpha}\too e$. There is $P\downarrow 0_{E}$ which witnesses $e_{\alpha}-e_{\beta}\todo 0_{E}$. That is, for every $p\in P$ there is $\alpha_{0}$ such that $-p\le e_{\alpha}-e_{\beta}\le p$, for every $\alpha,\beta\ge\alpha_{0}$. Taking the order limit in the second variable yields $-p\le e_{\alpha}-e\le p$, for every $\alpha\ge\alpha_{0}$. Thus, $e_{\alpha}\todo e$, and so $E$ is directed order complete.
\end{proof}

Note that the converse to part~(ii) of Corollary~\ref{monc} is false (see Corollary~\ref{iceord1}). We will also see in Example~\ref{jochen} an ordered vector space which is complete with respect to directed order convergence, but not monotone complete. An example of a monotone complete ordered vector space which fails to be directed order complete is as follows.

\begin{example}\label{monodo}
Recall that $\ell_{1}^{*}=\ell_{\8}$, and so we may view $\ell_{1}$ isometrically embedded into $\ell_{\8}\left(X\right)$, where $X:=\Ba_{\ell_{\8}}$. Let $E:=\left\{\left(e,f\right)\in \ell_{\8}\left(X\right)^{2}: e-f\in\ell_{1}\right\}$. We claim that $E$ is monotone complete, but not complete with respect to directed order convergence. To show the former let $P\subseteq E_{+}$ be directed and such that $P\le \left(e_{0},f_{0}\right)\in E$. For every $p\in P$ we have $p=\left(e_{p},f_{p}\right)$, where $\0\le e_{p}\le e_{0}$, $\0\le f_{p}\le f_{0}$ and $g_{p}:=e_{p}-f_{p}\in\ell_{1}$.

Let us show that $\left(g_{p}\right)_{p\in P}$ is norm Cauchy in $\ell_{1}$. If this is not the case, there are $\varepsilon>0$ and an increasing sequence $\left(p_{n}\right)_{n\in\N}\subseteq P$ such that $\left\|g_{p_{n+1}}-g_{p_{n}}\right\|\ge\varepsilon$. However, for every $x\in X$ we have that $\left(e_{p}\left(x\right)\right)_{p\in P}$ and $\left(f_{p}\left(x\right)\right)_{p\in P}$ are increasing nets bounded in $\R$. Hence, $\left(e_{p_{n}}\left(x\right)\right)_{n\in \N}$ and $\left(f_{p_{n}}\left(x\right)\right)_{n\in \N}$ are increasing bounded sequences in $\R$, and so they converge. It follows that $\left(g_{p_{n}}\left(x\right)\right)_{n\in \N}$ also converges, and so $g_{p_{n+1}}\left(x\right)-g_{p_{n}}\left(x\right)\to 0$, for every $x\in X$. Since $X=\Ba_{\ell_{\8}}$, we conclude that $g_{p_{n+1}}-g_{p_{n}}\tow \0$ in $\ell_1$, which by Schur's theorem implies $g_{p_{n+1}}-g_{p_{n}}\to \0$, contradicting $\left\|g_{p_{n+1}}-g_{p_{n}}\right\|\ge\varepsilon$.

It follows that there is $g\in \ell_{1}$ such that $g_{p}\to g$ in norm, hence $g_{p}\left(x\right)\to g\left(x\right)$, for every $x\in X$. We also have that for every $x\in X$ there are $e\left(x\right)\le e_{0}\left(x\right)$ and $f\left(x\right)\le f_{0}\left(x\right)$ such that $e_{p}\left(x\right)\uparrow e\left(x\right)$ and $f_{p}\left(x\right)\uparrow f\left(x\right)$. We can now conclude that $e,f\in \ell_{\8}\left(X\right)$, $\left(e,f\right)=\bigvee P$ in $\ell_{\8}\left(X\right)^{2}$, and $e-f=g\in\ell_{1}$. Hence, $\left(e,f\right)\in E$ and $\left(e,f\right)=\bigvee P$ in $E$.

Let us now prove that $E$ is not directed order complete. Since $\ell_{1}$ is not reflexive, there are $g\in \ell_{1}^{**}\setminus \ell_{1}$ and a net $\left(g_{\alpha}\right)_{\alpha\in A}\subseteq\Ba_{\ell_{1}}$ which $\mathrm{w}^{*}$-converges to $g$ in $\ell_{1}^{**}$. Viewing $\ell_{1}^{**}$ as isometrically embedded into $\ell_{\8}\left(X\right)$, we have that $g_{\alpha}\left(x\right)\to g\left(x\right)$, for every $x\in X$.

For every $\alpha$ and $x\in X$ put $h_{\alpha}\left(x\right):=\bigvee_{\beta,\gamma\ge\alpha}\left|g_{\beta}\left(x\right)-g_{\gamma}\left(x\right)\right|\le 2$. Then, $h_{\alpha}\left(x\right)\downarrow\0$, for every $x\in X$, and so $f_{\alpha}:=\left(h_{\alpha},h_{\alpha}\right)\downarrow\left(\0,\0\right)$ in $E$. For every $\alpha$ let $e_{\alpha}:=\left(g_{\alpha},\0\right)$; for every $\beta,\gamma\ge\alpha$ we have $-f_{\alpha}\le e_{\beta}-e_{\gamma}\le f_{\alpha}$. We conclude that $\left(e_{\alpha}\right)_{\alpha\in A}$ is directed order Cauchy.

Assume that there is $e\in E$ with $e=\left(u,v\right)$ such that $e_{\alpha}\todo e$. For every $\beta\ge \alpha$ taking the limit in $-f_{\alpha}\le e_{\beta}-e_{\gamma}\le f_{\alpha}$ with respect to $\gamma$ yields $-f_{\alpha}\le e_{\beta}-e\le f_{\alpha}$. In particular, for every $\alpha$ and $x\in X$ we have $\left|v\left(x\right)\right|\le h_{\alpha}\left(x\right)\downarrow 0$ and $\left|g_{\alpha}\left(x\right)-u\left(x\right)\right|\le h_{\alpha}\left(x\right)\downarrow 0$, thus $v=\0$ and $g_{\alpha}\left(x\right)\to u\left(x\right)$, for every $x\in X$. We conclude that $g=u=u-v\in\ell_{1}$, contradiction.
\qed\end{example}

\subsection{Order convergence as a convergence structure}\label{ss24}

Despite their names, Example~\ref{pat} shows that order and directed order convergences on a general poset may fail to be a convergence structure in the sense of \cite{bb} and \cite{bctv}. Nevertheless, in the case of an ordered vector space all the requirements to be a convergence structure are fulfilled. Part~(ii) in the following result for the case of vector lattices appeared in \cite[Proposition~2.3]{aeg}. For the notion of a bounded set in a convergence vector space see \cite[Section~6]{bctv}.

\begin{proposition}\label{ocl}Let $E$ be an ordered vector space.
\item[(i)] Order convergence and directed order convergence in $E$ are convergence structures, which are Hausdorff, meaning that convergent nets have unique limits.
\item[(ii)] Order convergence is linear if and only if directed order convergence is linear if and only if $E$ is directed and Archimedean.
\item[(iii)] If $E$ is directed and Archimedean, relative uniform convergence is a linear Hausdorff convergence structure. Moreover, boundednesses with respect to order, directed order and relative uniform convergences are equivalent to order boundedness.
\end{proposition}
\begin{proof}
We only treat order convergence. The case of directed order convergence is similar.

(i): It is clear that a constant net order converges to its value, and that a quasi-subnet of an order convergent net order converges to the same limit. Assume that $e_{\alpha}\too e$, $f_{\alpha}\too e$, and $g_{\alpha}\in\left\{e_{\alpha},f_{\alpha}\right\}$, for every $\alpha$. It follows from part~(ii) of Proposition~\ref{ocv} that $e_{\alpha}-f_{\alpha}\too 0_{E}$. Let $P$ be as in part~(i) of Proposition~\ref{ocv} for $\left(e_{\alpha}-f_{\alpha}\right)_{\alpha\in A}$. For every $\alpha$ we have that $g_{\alpha}-f_{\alpha}\in\left\{e_{\alpha}-f_{\alpha}, 0_{E}\right\}$, and so it is easy to see that $P$ witnesses $g_{\alpha}-f_{\alpha}\too 0_{E}$. Using part~(ii) of Proposition~\ref{ocv} again we conclude that $g_{\alpha}=g_{\alpha}-f_{\alpha}+f_{\alpha}\too 0_{E}+e=e$. Hausdorffness follows immediately from Proposition~\ref{oc}.

(ii): Necessity: The fact that $E$ is Archimedean is proved similarly to the implication (iii)$\Rightarrow$(i) in Proposition~\ref{oca}. Let $e\in E$; it follows from linearity of order convergence that $\frac{1}{n}e\too 0_{E}$, witnessed by $P\subseteq E_{+}$ with $\bigwedge P=0_{E}$. It follows that there are $0_{E}\le p\in P$ and $n\in\N$ such that $\frac{1}{n}e\le p$. Then, $np\ge e,0_{E}$, showing that $E$ is directed.

Sufficiency: In the light of part~(ii) of Proposition~\ref{ocv}, we only need to prove that if $e_{\alpha}\too e$ and $\R\ni r_{\alpha}\to r$, then $r_{\alpha}e_{\alpha}\too re$. It is easy to see that $re_{\alpha}\too re$, and so without loss of generality we may assume that $r=0$. By passing to a tail, we may also assume that $\left(e_{\alpha}\right)_{\alpha\in A}\subseteq\left[-u,u\right]$, for $u\in E_{+}$. Then, $r_{\alpha}e_{\alpha}\toru 0_{E}$ with regulator $u$, which implies $r_{\alpha}e_{\alpha}\too 0_{E}$, according to Proposition~\ref{oca}.

The case of directed order convergence is analogous.

(iii): The fact that $\mathrm{ru}$ is a linear convergence structure is proven in \cite[Section~10]{btw} (note that it is immaterial for that result whether $E$ is a lattice).

Assume that $B\subseteq\left[-u,u\right]$, for $u\in E_{+}$. Let $\left(e_{\alpha}\right)_{\alpha\in A}\subseteq B$ and let $\left(r_{\alpha}\right)_{\alpha\in A}\subseteq \R$ converge to $0$. Arguing as in the previous step, we see that $r_{\alpha}e_{\alpha}\too 0_{E}$. We conclude that $E$ is bounded with respect to order convergence. The converse follows from the fact that order convergence is locally order bounded and \cite[Corollary~8.3]{bctv} (note that it is immaterial for that result whether $E$ is a lattice).
\end{proof}

While order convergence may be viewed as more natural due to the last claim in Proposition~\ref{oca}, directed order convergence also has a claim to ``naturalness''. Indeed, analogously to \cite[Theorem~4.1]{erz} one can show that directed order convergence is the strongest locally full convergence (i.e., satisfies the Sandwich Theorem) on $E$ in which every net which decreases to zero converges to zero.

Also, note that $\mathrm{ru}$ is the strongest locally full linear convergence on $E$, and $\mathrm{ru}$-completeness follows from completeness with respect to any locally full convergence for which $E_{+}$ is closed (the proofs of \cite[Proposition~5.1 and Theorem~5.2]{bctv} can be easily adapted).

\begin{proposition}\label{do=ru}
Let $E$ be an ordered Banach space such that the norm induces a locally full convergence\footnote{This condition is usually described as ``$E$ has a normal cone'', see \cite[Theorem~2.23]{at}.} such that every net which decreases to $0_{E}$ converges in norm (to $0_{E}$). Then, $\mathrm{ru}=\mathrm{do}\ge\|\cdot\|$ on $E$.
\end{proposition}
\begin{proof}
First, as an ordered Banach space, $E$ is Archimedean (see \cite[Lemma~2.3]{at}), and so $\mathrm{ru}\ge \mathrm{do}$, by Proposition~\ref{oca}. Since the norm convergence is locally full and every net which decreases to $0_{E}$ converges in norm to $0_{E}$, it follows from the comment above that $\mathrm{do}\ge\|\cdot\|$ on $E$.

Assume that $e_{\alpha}\todo 0_{E}$, witnessed by $P\downarrow 0_{E}$. Then, $P\to 0_{E}$ in norm, and so there are $h\in E_{+}$ and $\left(p_{n}\right)_{n\in\N}\subseteq P$ such that $p_{n}\le\frac{1}{n}h$ (the proof of \cite[Lemma~1.1]{tt} goes through, since the elements of the net are positive). For every $n\in\N$ there is $\alpha_{n}$ such that $-\frac{1}{n}h\le -p_{n}\le e_{\alpha} \le p_{n}\le \frac{1}{n}h$, whenever $\alpha\ge\alpha_{n}$. We conclude that $e_{\alpha}\toru 0_{E}$ with regulator $h$.
\end{proof}

We now present an example of a $\mathrm{do}$-complete directed Archimedean ordered vector space which is not monotone complete. We moreover show that an order bounded increasing sequence need not be directed order Cauchy, in contrast to Proposition~\ref{oca}.

\begin{example}\label{jochen}
Consider the ordered Banach space $F$ presented in \cite[Example~2.6]{glu}, which has an order unit (hence directed), and satisfies the condition of Proposition~\ref{do=ru}, but contains an increasing order bounded sequence $\left(f_{n}\right)_{n\in\N}\subseteq F$ which is not norm-Cauchy. Then, this sequence is not $\mathrm{do}$-Cauchy, since otherwise it would follow that $f_{n}-f_{m}\todo 0_{F}$, and hence $f_{n}-f_{m}\to 0_{F}$, and it does not have a supremum, because $f_{n}\uparrow f$ would imply $f_{n}\todo f$, and hence $f_{n}\to f$.
\qed\end{example}

Since there is only one linear Hausdorff convergence structure on a finite-dimensional vector space (see \cite[Theorem~3.3.19]{bb}), it follows from Proposition~\ref{ocl} that if $E$ is directed and Archimedean with $\dim E<\8$, then order and directed order convergences coincide. A large class of directed Archimedean ordered vector spaces (which includes vector lattices) on which order and directed order convergences coincide will be considered in the next section. See part~(iii) of Theorem~\ref{iceord} for a class of spaces for which order and directed order convergences differ. That example also demonstrates the failure of the last claim in Proposition~\ref{oca} for directed order convergence, since the two coincide on the Dedekind completion. It also follows that if $E$ is directed and Archimedean with $\dim E<\8$, then it is complete with respect to both order and directed order convergence. Note that in this case, unless $E$ is order isomorphic to $\R^{n}$ (with the standard ordering) for some $n\in\N$, it is not a lattice (see \cite[Theorem~1.7.8]{kv}).

\section{Pervasive ordered vector spaces}\label{s3}

\subsection{Characterizations and basic properties of pervasive OVS's}\label{ss31}

In this section we study the class of non-lattice ordered vector spaces which are closest to vector lattices in many regards. First, we investigate the behavior of ``large enough'' subspaces of vector lattices. Recall that a subspace $E$ of an ordered vector space $F$ is called \emph{order dense}\footnote{We use this term according to the convention from vector lattice theory. In \cite{kv} this term is used for what we call ``supremum-dense'', whereas what we call ``order dense'' is called ``pervasive'' there.} if for every $f>0_{F}$ we have $\left(0_{F},f\right]\cap E\ne\varnothing$.

\begin{proposition}\label{pervdense}
If $F$ is an Archimedean vector lattice and $E\subseteq F$ is an order dense subspace, then $f=\bigvee\left[0_{F},f\right]_{E}$, for every $f\in F_{+}$.
\end{proposition}
\begin{proof}
As $F$ is a lattice, it is enough to show that $f$ is a minimal upper bound for $\left[0_{F},f\right]_{E}$. Assume that there is $f>g\ge\left[0_{F},f\right]_{E}$. Then, there is $h\in \left(0_{F},f-g\right]\cap E$, but then $\left[0_{F},f\right]_{E}+h\subseteq\left[0_{F},f\right]_{E}$ and part~(i) of Lemma~\ref{arch} yields that $h=0_{F}$, contradiction.
\end{proof}

\begin{corollary}\label{odm}
If $F$ is an Archimedean vector lattice and $E\subseteq F$ is an order dense majorizing subspace, then $E$ is supremum-dense in $F$. In particular, the embedding of $E$ into $F$ preserves arbitrary infima and suprema.
\end{corollary}
\begin{proof}
For $f\in F$ find $e\in E$ with $e\ge f$. Then, $e-f=\bigvee \left[0_{F},e-f\right]_{E}$, hence
\[f=e-\bigvee\left[0_{F},e-f\right]_{E}=\bigwedge\left[f,e\right]_{E},\]
which implies $f=\bigwedge \left[f\right)_{E}$. Thus, $E$ is infimum-dense in $F$, which is equivalent to being supremum-dense. The last claim follows from Proposition~\ref{ocp}.
\end{proof}

The following result extends \cite[Lemma~3.7.11]{kv}.

\begin{proposition}\label{pervdir}
Let $F$ be an Archimedean vector lattice and let $E\subseteq F$ be an order dense subspace. Let $P\subseteq F_{+}$ be upward directed and let $g\in F$ be such that $g=\bigvee P$. Then there is an upward directed $Q\subseteq \bigcup_{p\in P}\left[0_{F},p\right]_{E}$ such that $g=\bigvee Q$.
\end{proposition}
\begin{proof}
The family of all nonempty directed subsets of $R:=\bigcup_{p\in P}\left[0_{F},p\right]_{E}$ is nonempty (as it contains $\left\{0_{F}\right\}$) and every chain of such subsets has an upper bound (its union). By Zorn's lemma there is a maximal member $Q$ of this family. We will show that $g=\bigvee Q$. Assume that this is false, and so there is $f\in F$ with $Q\le f< g$. Since $g=\bigvee P$, we can find $P\ni g_{1}\not\le f$. As $E$ is order dense in $F$, there is $e\in E\cap \left(0_{F},\left(g_{1}-f\right)^{+}\right]$.

For a fixed $h\in Q$ let us show that $h+e\in R$. There is $g_{2}\in P$ such that $h\in \left[0_{F},g_{2}\right]$. Since $P$ is directed, it contains $g_{3}\ge g_{1},g_{2}$. Then, $h\le g_{2}\le g_{3}$ and $h\in Q\le f$ yield $h\le g_{3}\wedge f$. Hence,
\[h+e\le g_{3}\wedge f+\left(g_{1}-f\right)^{+}\le g_{3}\wedge f+\left(g_{3}-f\right)^{+}=g_{3},\]
and so $h+e\in \left[0_{F},g_{3}\right]_{E}\subseteq R$.

It follows that $Q\cup \left(Q+e\right)\subseteq R$. Furthermore, $Q\cup \left(Q+e\right)$ is directed: if $h_{1}\in Q$ and $h_{2}\in Q+e$, find $Q\ni h_{3}\ge h_{1},h_{2}-e$; then $h_{3}+e\ge h_{1},h_{2}$ (other cases are obvious). Due to maximality of $Q$, we have $Q+e\subseteq Q$. As $Q$ is bounded from above and $F$ is Archimedean, $e\le 0_{F}$ by virtue of Lemma~\ref{arch}, contradicting the assumption on $e$.
\end{proof}

\begin{corollary}\label{pervdird}
Let $F$ be an Archimedean vector lattice and let $E\subseteq F$ be an order dense subspace. Let $P\subseteq F_{+}$ be directed downward with $\bigwedge P=0_{F}$ such that there are $p_{0}\in P$ and $e_{0}\in E$ with $p_{0}\le e_{0}$. Then, there is $Q\subseteq \left[0_{F},e_{0}\right]_{E}\cap \left(P+F_{+}\right)$ which is directed downward and such that $\bigwedge Q=0_{F}$.
\end{corollary}
\begin{proof}
By replacing $P$ with $\left[0_{F},p_{0}\right]\cap P$ if needed, we may assume that $P\le p_{0}\le e_{0}$. Let $R:=e_{0}-P$, which is upward directed and satisfies $e_{0}=\bigvee R$. According to Proposition~\ref{pervdir}, there is an upward directed $S\subseteq \bigcup_{r\in R}\left[0_{F},r\right]_{E}$ such that $e_{0}=\bigvee S$. Finally, $Q:=e_{0}-S$ fulfills the requirements.
\end{proof}

A directed ordered vector space is called \emph{pervasive} if it embeds as an order dense subspace of a vector lattice. Equivalence of the conditions (i), (vii) and (ix) in the following result is \cite[Theorem~7]{km}.

\begin{theorem}\label{pervasive}
For a directed Archimedean ordered vector space $E$ the following conditions are equivalent:
\item[(i)] $E$ is pervasive;
\item[(ii)] The image of $E$ under the natural embedding into $E^{\delta}$ is order dense;
\item[(iii)] For every $e\in E^{\delta}_{+}$ there is a directed $Q\subseteq E_{+}$ such that $e=\bigvee Q$;
\item[(iv)] For every $e<f$ we have that $\left[e,f\right]^{\partial}\simeq \left[e,f\right]_{E^{\delta}}$;
\item[(v)] $\left(E_{+}\right)^{\partial}\simeq\left(E^{\delta}\right)_{+}\cup \left\{+\8\right\}$, where if $E=\left\{0_{E}\right\}$ we set $+\8:=0_{E}$;
\item[(vi)] The embedding of $E_{+}$ into $E$ preserves arbitrary infima;
\item[(vii)] For every $P\subseteq E_{+}$ such that $\bigwedge_{E_{+}}P=0_{E}$ we have $\bigwedge P=0_{E}$;
\item[(viii)] For every $e\not\le 0_{E}$ there is $g>0_{E}$ such that $\left\{e,0_{E}\right\}^{\uparrow}\subseteq \left\{g\right\}^{\uparrow}$;
\item[(ix)] For every $e,f\in E_{+}$ with $e\not\le f$ there is $g>f$ such that $\left\{e,f\right\}^{\uparrow}\subseteq \left\{g\right\}^{\uparrow}$;
\item[(x)] For every $e,f$ with $\left\{e,f\right\}^{\uparrow}\subseteq E_{+}\setminus\left\{0_{E}\right\}$ there is $g>0_{E}$ such that $\left\{e,f\right\}^{\uparrow}\subseteq \left\{g\right\}^{\uparrow}$.
\end{theorem}
\begin{proof}
(iii)$\Rightarrow$(ii)$\Rightarrow$(i) is trivial. (ii)$\Rightarrow$(iii) follows from applying Proposition~\ref{pervdir} to $\left\{e\right\}$. (v)$\Rightarrow$(ii) follows from the fact that $E_{+}$ is supremum-dense in $\left(E_{+}\right)^{\partial}$.

(ii)$\Rightarrow$(iv): First, note that $\left[e,f\right]_{E^{\delta}}$ is a complete lattice, which contains $\left[e,f\right]$. Hence, in order to prove that $\left[e,f\right]^{\partial}\simeq \left[e,f\right]_{E^{\delta}}$ it is enough to show that $\left[e,f\right]$ is supremum- and infimum-dense in $\left[e,f\right]_{E^{\delta}}$. Since translations on ordered vector spaces are order isomorphisms, without loss of generality we may assume that $e=0_{E}$. Now the supremum density follows directly from Proposition~\ref{pervdense}, and the fact that being a supremum of a set in a poset implies being a supremum in its order interval. Infimum density can be reduced to the preceding argument by taking $f=0_{E}$ and using the fact that additive inversion is an order anti-isomorphism.

(iv)$\Rightarrow$(v): By the argument above we only need to prove that $E_{+}$ is supremum- and infimum-dense in $\left(E^{\delta}\right)_{+}\cup \left\{+\8\right\}$. For every $0_{E^{\delta}}<f\in E^{\delta}$ there is $e\in E$ such that $f<e$. By assumption, we have
\[f=\bigvee_{\left[0_{E^{\delta}},e\right]_{E^{\delta}}}\left(\left[0_{E},e\right]\cap\left[0_{E^{\delta}},f\right]_{E^{\delta}}\right)=\bigvee_{\left[0_{E^{\delta}},e\right]_{E^{\delta}}}\left(E\cap\left[0_{E^{\delta}},f\right]_{E^{\delta}}\right).\] It is easy to see that then $f=\bigvee_{E^{\delta}}\left(E\cap\left[0_{E^{\delta}},f\right]_{E^{\delta}}\right)$. This argument shows supremum density; infimum density follows from the fact that $E$ is infimum-dense in $E^{\delta}$.

(v)$\Rightarrow$(vi): Assume that $Q\subseteq E_{+}$ and $e\in E_{+}$ are such that $e=\bigwedge_{E_{+}}Q$. Then, by Proposition~\ref{ocp}, it follows that $e=\bigwedge_{\left(E_{+}\right)^{\partial}}Q$, and it follows from the assumption that $e=\bigwedge_{\left(E^{\delta}\right)_{+}}Q$. As $E^{\delta}$ is a vector lattice, this implies that $e=\bigwedge_{E^{\delta}}Q$, and as all involved elements are in $E$, we conclude that $e=\bigwedge_{E}Q$.

(vi)$\Rightarrow$(vii) is trivial.

(vii)$\Rightarrow$(viii): If $e\not\le 0_{E}$, then $0_{E}$ is not the infimum of $\left\{e,0_{E}\right\}^{\uparrow}$ in $E$. By assumption it follows that $0_{E}$ is not the infimum of $\left\{e,0_{E}\right\}^{\uparrow}$ in $E_{+}$, and so there is a nonzero lower bound $g$ for this set in $E_{+}$. It is easy to see that this $g$ satisfies the requirements.

(viii)$\Rightarrow$(ii): For $0_{E^{\delta}}<h$, the set $Q:=\left(h-E^{\delta}_{+}\right)\cap E$ is a cut in $E$ which strictly contains $-E_{+}$, and so there is $e\in Q$ such that $e\not\le 0_{E^{\delta}}$. Take $0_{E^{\delta}}<g\in E$ such that $\left\{0_{E},e\right\}^{\uparrow}\subseteq g+E_{+}$, from where $g\in\left\{0_{E},e\right\}^{\uparrow\downarrow}\subseteq Q\subseteq h-E^{\delta}_{+}$, hence $g\in \left(0_{E^{\delta}},h\right]$.

(i)$\Rightarrow$(viii): Assume that $F$ is a vector lattice, and that $E\subseteq F$ is an order dense subspace. Take $E\ni e\not\le 0_{E}$. Then, $e^{+}>0_{F}$, and so we can find $g\in E\cap \left(0_{F},e^{+}\right]$. Now if $f\ge e,0_{E}$, then $f\ge e^{+}$, and so $f\ge g$, as required.

(viii)$\Rightarrow$(ix): Applying (viii) to $e-f$ we get $h>0_{E}$ such that $\left\{e-f,0_{E}\right\}^{\uparrow}\subseteq \left\{h\right\}^{\uparrow}$. Then, $g:=h+f$ satisfies the requirements.

(ix)$\Rightarrow$(viii): Since $E$ is directed, we have that $e=u-v$, where $u,v\in E_{+}$ with $u\not\le v$. Hence, by (ix) there is $w>v$ such that $\left\{u,v\right\}^{\uparrow}\subseteq \left\{w\right\}^{\uparrow}$. Then, $g:=w-v$ satisfies the requirements.

(x)$\Rightarrow$(viii) is obtained by taking $f=0_{E}$.

(ii)$\Rightarrow$(x): As $\left\{e,f\right\}^{\uparrow}\subsetneq E_{+}$ it follows that $\left\{e,f\right\}^{\uparrow\downarrow}\supsetneq -E_{+}=0_{E^{\delta}}$, which means that $\left\{e,f\right\}^{\uparrow\downarrow}=e\vee_{E^{\delta}}f$ is a strictly positive element of $E^{\delta}$. By (ii) there is $g\in E$ such that $0_{E^{\delta}}<g\le e\vee_{E^{\delta}}f$. The latter means $g\in\left\{e,f\right\}^{\uparrow\downarrow}$, which translates into $\left\{e,f\right\}^{\uparrow}\subseteq \left\{g\right\}^{\uparrow}$.
\end{proof}

\begin{remark}\label{pervint}
Analogously to (v)$\Rightarrow$(vi) one can show that in a pervasive Archimedean ordered vector space embeddings of order intervals preserve arbitrary suprema and infima.
\qed\end{remark}

It follows from Theorem~\ref{pervasive} that if $E$ is pervasive and Archimedean, then it is dense in $E^{\delta}$ with respect to order convergence in the sense that every element of $E^{\delta}$ is the order limit of a net in $E$. We now consider other properties of pervasive spaces.

\begin{proposition}\label{arch1}
If an ordered vector space $F$ contains a pervasive Archimedean ordered vector space $E$ as a supremum-dense subspace, then $F$ is pervasive and $E$ is order dense and majorizing in $F$.
\end{proposition}
\begin{proof}
First, according to Proposition~\ref{arch0}, $F$ is Archimedean, and so $F^{\delta}$ exists. Moreover, $E$ is supremum-dense in $F^{\delta}$, therefore $F^{\delta}=E^{\delta}$. By Theorem~\ref{pervasive}, $E$ is order dense in $E^{\delta}$, and since we have $E\subseteq F\subseteq E^{\delta}$, it follows that $E$ is order dense in $F$. It also follows that $F$ is order dense in $F^{\delta}$, and so by Theorem~\ref{pervasive}, $F$ is pervasive.
\end{proof}

Combining this with Corollary~\ref{odm}, we get the following refinement of the fact that $E^{\delta}$ is unique among the Dedekind complete vector lattices which contain $E$ as a supremum-dense subspace.

\begin{corollary}\label{dedeperv}
If $E$ is a pervasive Archimedean ordered vector space, then $E^{\delta}$ is the unique Dedekind complete vector lattice which contains $E$ as an order dense majorizing subspace.
\end{corollary}

The following result underscores the importance of pervasiveness.

\begin{proposition}\label{o=do}
If $E$ is a pervasive Archimedean ordered vector space, then order convergence and directed order convergence agree on $E$.
\end{proposition}
\begin{proof}
We only need to show that order convergence is stronger than directed order convergence. In light of Proposition~\ref{oca}, it is enough to show that if $\left(e_{\alpha}\right)_{\alpha\in A}\subseteq E$ is (directed) order null in $E^{\delta}$, then it is directed order null in $E$. Let $P\subseteq E^{\delta}$ be a directed downward set which witnesses the order convergence of the net. According to Proposition~\ref{pervdird}, there is $Q\subseteq E\cap \left(P+E_{+}\right)$ which is directed downward and such that $\bigwedge Q=0_{E^{\delta}}$. For every $q\in Q$ there is $p\in P$ with $q\ge p$, and so there is $\alpha_{0}$ such that $\pm e_{\alpha}\le p\le q$, as soon as $\alpha\ge\alpha_{0}$. Hence, $Q$ witnesses $e_{\alpha}\todo 0_{E}$, as desired.
\end{proof}

We will call $E$ \emph{weakly pervasive} if whenever $e,f\in E_{+}$ are such that $e\wedge_{E_{+}}f=0_{E}$, it follows that $e\wedge_{E}f=0_{E}$ (and so it is a weakened version of the condition (vi) in Theorem~\ref{pervasive}).

\begin{lemma}\label{wp}
For an ordered vector space $E$ the following conditions are equivalent:
\item[(i)] $E$ is a vector lattice;
\item[(ii)] For every $e\in E$ there are $f,g\in E_{+}$ such that $e=f-g$ and $f\wedge_{E}g=0_{E}$;
\item[(iii)] $E$ is weakly pervasive and for every $e\in E$ there are $f,g\in E_{+}$ such that $e=f-g$ and $f\wedge_{E_{+}}g=0_{E}$.
\end{lemma}
\begin{proof}
(i)$\Rightarrow$(iii): To prove that $E$ is weakly pervasive assume that $e,f\in E_{+}$ are such that $e\wedge_{E_{+}}f=0_{E}$. If $g\in E$ is such that $g\le e,f$, then $g^{+}\in E_{+}$ together with $g^{+}\le e,f$ yield $g^{+}=0_{E}$, hence $g\le 0_{E}$. Thus, $e\wedge_{E} f=0_{E}$.

For every $e\in E$ we have that $e^{+},e^{-}\in E_{+}$, $e=e^{+}-e^{-}$ and $e^{+}\wedge_{E_{+}}e^{-}=0_{E}$.

(iii)$\Rightarrow$(ii) follows from the definition of weak pervasiveness.

(ii)$\Rightarrow$(i): Let $e\in E$ and let $f,g\in E_{+}$ be such that $e=f-g$ and $f\wedge g=0_{E}$. Then, $e^{+}= \left(f-g\right)^{+}=f-f\wedge g=f$. Hence, every element of $E$ has a positive part, which is sufficient for $E$ to be a lattice.
\end{proof}

\begin{proposition}\label{pervcomplete}
For a directed ordered vector space $E$ the following conditions are equivalent:
\item[(i)] $E$ is a Dedekind complete vector lattice;
\item[(ii)] $E$ is pervasive, Archimedean and complete with respect to order convergence;
\item[(iii)] $E$ is monotone complete and pervasive;
\item[(iv)] $E$ is directed, Archimedean, weakly pervasive, and such that every order bounded totally ordered set in $E$ has a supremum.
\end{proposition}
\begin{proof}
(i)$\Rightarrow$(ii) is established in \cite[Proposition~9.10]{btw}. (ii)$\Rightarrow$(iii) follows from part~(ii) of Proposition~\ref{monc}. In (iii)$\Rightarrow$(iv), Archimedeanness follows from part~(i) of Proposition~\ref{monc}, while the rest is straightforward. In (iv)$\Rightarrow$(i), the fact that $E$ is a lattice is proved in \cite[Theorem~4.12]{gluck}, while Dedekind completeness follows from \cite[Corollary~1]{markowsky}.%(i)$\Rightarrow$(iii): We may assume that $E$ is an order dense subspace of $E^{\delta}$. Let $f\in E^{\delta}_{+}$. It follows from Theorem~\ref{pervasive} that there is a directed $Q\subseteq E$ such that $\bigvee Q=f$. Since $E$ is majorizing in $E^{\delta}$, it follows that $Q$ is an increasing order bounded net in $E$ (indexed by itself). By assumption, there is $e:=\bigvee_{E}Q$, but as the inclusion of $E$ into $E^{\delta}$ is supremum-preserving, it follows that $f=e\in E$. Due to arbitrariness of $f$ we conclude that $E^{\delta}=E$.%According to Proposition~\ref{oca}, $Q$ is order Cauchy in $E$, and so it order converges to some $e\in E$. Since order limit of an increasing net must be its supremum, we conclude that $e=\bigvee_{E}Q$. A
\end{proof}

It is easy to see that a linear subspace $H\subseteq E$ is full if and only if $0_{E}\le e\le h\in H$ implies $e\in H$. If $E$ is a vector lattice, an \emph{ideal} in $E$ is a full linear sublattice of $E$. If $A\subseteq E$, we denote the ideal generated by $A$ in $E$ by $I_{E}\left(A\right)$.

\begin{proposition}\label{pervid}
Let $F$ be a vector lattice, let $E\subseteq F$ be an order dense subspace, and let $H\subseteq E$ be a directed full subspace. Then $H$ is an order dense majorizing subspace of $I_{F}\left(H\right)$.
\end{proposition}
\begin{proof}
First, since $H$ is directed, it is easy to see that $I_{F}\left(H\right)=\bigcup_{h\in H_{+}}\left[-h,h\right]$. Let $0_{F}<f\in I_{F}\left(H\right)$, so that there is $h\in H_{+}$ with $f\le h$. As $E$ is order dense, there is $e\in\left(0_{F},f\right]\cap E$. In particular, $0_{F}\le e\le h$, and since $H$ is full we conclude that $e\in H$.
\end{proof}

In case several posets are present, we denote order convergence in a poset $P$ by $\mathrm{o}_{P}$. Part~(i) of the next result is close to \cite[Theorem~27]{km2}.

\begin{corollary}\label{pervideal}
Let $E$ be a pervasive Archimedean ordered vector space and let $H\subseteq E$ be a directed full subspace. Then:
\item[(i)] $H$ is pervasive and $I_{E^{\delta}}\left(H\right)=H^{\delta}$;
\item[(ii)] $\mathrm{o}_{E}$ and $\mathrm{o}_{H}$ agree on subsets of $H$ which are order bounded in $H$.
\end{corollary}
\begin{proof}
(i) follows from Corollary~\ref{dedeperv} and Proposition~\ref{pervid}.

(ii): Let $h\in H_{+}$. It follows from (i) that $R:=\left[-h,h\right]_{H^{\delta}}$ is an order interval in $E^{\delta}$. Hence, the restrictions of both $\mathrm{o}_{E^{\delta}}$ and $\mathrm{o}_{H^{\delta}}$ to $R$ agree with $\mathrm{o}_{R}$ (see \cite[Lemma~2.2]{abt}). By the last part of Proposition~\ref{oca}, we have $\mathrm{o}_{E}=\left.\mathrm{o}_{E^{\delta}}\right|_{E}$ and $\mathrm{o}_{H}=\left.\mathrm{o}_{H^{\delta}}\right|_{H}$. Since $\left[-h,h\right]\subseteq R$, we conclude that $\left.\mathrm{o}_{E}\right|_{\left[-h,h\right]}=\left.\mathrm{o}_{E^{\delta}}\right|_{\left[-h,h\right]}=\left.\mathrm{o}_{H^{\delta}}\right|_{\left[-h,h\right]}=\left.\mathrm{o}_{H}\right|_{\left[-h,h\right]}$.
\end{proof}

The following example shows that the fact that $I_{E^{\delta}}\left(H\right)=H^{\delta}$ for each full directed subspace $H$ of $E$ does not imply pervasiveness of $E$.

%%ST310: say what L_\infty is? I know they all isomorphic, but choose the model L_\infty[0,1].
\begin{example}\label{nonperv}
We will denote $L_{\8}:=L_{\8}\left[0,1\right]$ and let
\[E:=\left\{\left(e,f\right)\in L_{\8}\times L_{\8}: \textstyle\int e = \int f\right\}.\]
First, let us prove that $E^{\delta}=L_{\8}\times L_{\8}$. It is enough to show that $E$ is supremum-dense in $L_{\8}\times L_{\8}$. Let $g,h\in L_{\8}$. If $\int g = \int h$, then $\left(g,h\right)\in E$. Alternatively, without loss of generality we may assume that $\delta:=\int g - \int h>0$. Let $g_{1}:=g-2\delta\1_{\left[0,1/2\right]}$ and $g_{2}:=g-2\delta\1_{\left[1/2,1\right]}$. We have that $\int g_{i}=\int g-\delta=\int h$, for $i=1,2$. Hence, $\left(g_{1},h\right), \left(g_{2},h\right)\in E$ and $\left(g_{1},h\right)\vee \left(g_{2},h\right)=\left(g,h\right)$.

To prove that $E$ is not pervasive, we use the condition (ii) in Theorem~\ref{pervasive}: we show that $E$ is not order dense in $E^{\delta}=L_{\8}\times L_{\8}$. Indeed, if $\left(e,f\right)\in E_{+}$ is such that $\left(e,f\right)\le \left(\1,\0\right)$, then $f=\0$, hence $\int e=\int f=0$, and so $e=\0$.

Let $H\subseteq E$ be full and directed. Let us show that $I_{E^{\delta}}\left(H\right)=H^{\delta}$, for which it is again enough to prove that $H$ is supremum-dense in $I_{E^{\delta}}\left(H\right)$. Let $\left(g,h\right)\in I_{E^{\delta}}\left(H\right)$, so that there is $\left(e,f\right)\in H_{+}$ with $\left|g\right|\le e$ and $\left|h\right|\le f$. If $f=\0$, then $\int e=\int f=0$, hence $e=\0$, and so $g=h=\0$, implying $\left(g,h\right)\in H$. Similarly, $e=\0$ also yields $\left(g,h\right)\in H$. We thus assume that $e,f>\0$. Also, we may again assume that $\delta:=\int g - \int h>0$.

There are disjoint measurable $A,B\subseteq\left[0,1\right]$ and $\varepsilon>0$ such that $\lambda\left(A\right)=\lambda\left(B\right)=\varepsilon$ and $\varepsilon\1_{A}, \varepsilon\1_{B}\le e$. It follows that for
\[g_{1}:=g-\frac{\delta}{\varepsilon}\1_{A}\quad\text{and}\quad g_{2}:=g-\frac{\delta}{\varepsilon}\1_{B}\]
we have $\int g_{1}=\int g_{2}=\int h$, and so $\left(g_{1},h\right),\left(g_{2},h\right)\in E$. Since $f\ge\0$ and $\varepsilon\1_{A}\le e$, we get
\[\left(e,f\right)\ge \left(g_{1},h\right)\ge -\left(e,f\right) - \frac{\delta}{\varepsilon^{2}}\left(e,f\right)=-\frac{\delta+\varepsilon^{2}}{\varepsilon^{2}}\left(e,f\right),\]
implying $\left(g_{1},h\right)\in H$. Similarly, $\left(g_{2},h\right)\in H$, and moreover, disjointness of $A$ and $B$ yields $\left(g_{1},h\right)\vee \left(g_{2},h\right)=\left(g,h\right)$.
\qed\end{example}

Up until this point in this section we only considered Archimedean ordered vector spaces. Let us comment on the situation outside of this assumption. We call $E$ a \emph{pre-Riesz space} if it embeds as a supremum-dense subspace of a vector lattice. Then, there is an essentially unique vector lattice $E^{\rho}$ which contains $E$ as a supremum-dense subspace that generates $E^{\rho}$ as a vector lattice (see \cite[Chapter~2]{kv}). We call $E^{\rho}$ the \emph{Riesz completion} of $E$; it is the sublattice generated by $E$ in the lattice $E^{\partial}$ (no linear operations are involved). In case $E$ is Archimedean, $E^{\rho}$ is also the linear sublattice generated by $E$ in $E^{\delta}$. It is easy to see that a pre-Riesz space must be directed. Conversely, a directed Archimedean space is a pre-Riesz space.

\begin{remark}\label{prepervasive}
Theorem~\ref{pervasive} remains partially valid for pre-Riesz spaces. Namely, the conditions (i), (vi), (vii), (viii) and (ix) are equivalent to the version of the condition (ii) in which $E^{\delta}$ is replaced with $E^{\rho}$. We only need to supply a new proof for (i)$\Rightarrow$(vi). Assume that $F$ is a vector lattice, and $E$ is order dense in $F$. Let $Q\subseteq E_{+}$ be such that $\bigwedge_{E_{+}}Q=0_{E}$, and let $e\in E$ be a lower bound for $Q$. Then, $e^{+}\in F$ is also a lower bound for $Q$. If $e\not\le 0_{E}$, then $e^{+}>0_{F}$, and since $E$ is order dense in $F$, there is $0_{E}<g\in E$ such that $g\le e^{+}\le Q$, which contradicts $\bigwedge_{E_{+}}Q=0_{E}$. Thus, $\bigwedge_{E}Q=0_{E}$.

It follows that a pervasive pre-Riesz space is Archimedean whenever
\[\bigwedge\nolimits_{E_{+}}\left\{\tfrac{1}{n}f: n\in\N\right\}=0_{E},\quad\text{for every }f\in E_{+}.\] In particular, every pervasive almost Archimedean (see \cite[Section~1.2, Exercise~6]{at}) pre-Riesz space is Archimedean.
\qed\end{remark}

We call vectors in $E$ \emph{disjoint} if they are disjoint in $E^{\rho}$, and say that $E$ is \emph{disjointness-free} if $e\perp f$ implies that either $e=0_{E}$, or $f=0_{E}$.

\begin{proposition}\label{l:vector lattice totally ordered disjointness free}
   Let $E$ be a pre-Riesz space. The following conditions are equivalent:
    \item[(i)] $E$ is totally ordered;
    \item[(ii)] $E$ is a disjointness-free lattice;
    \item[(iii)] $E$ is disjointness-free and pervasive.\medskip

    If in this case $E$ is also Archimedean, then $\dim E\le 1$.
\end{proposition}
\begin{proof}
(i)$\Rightarrow$(ii): Clearly, $E$ is a lattice. Moreover, since the infimum of two elements is their minimum, it follows that $E$ is disjointness-free.

(ii)$\Rightarrow$(i): If $E$ is not totally ordered, then after translation there is an element $e$ incomparable with $0_{E}$, and so $e^{+}$ and $e^{-}$ are nonzero disjoint elements.

(ii)$\Rightarrow$(iii) is trivial.

(iii)$\Rightarrow$(i): First, we show that $E^{\rho}$ is disjointness-free. Indeed, if $0_{E^{\rho}}<g,h\in E^{\rho}_{+}$ are disjoint, according to Remark~\ref{prepervasive}, there are $e\in\left(0_{E^{\rho}},g\right]\cap E$ and $f\in\left(0_{E^{\rho}},h\right]\cap E$, which are then also disjoint, contradiction. Now we can apply the equivalence of (i) and (ii) to $E^{\rho}$ and conclude that it is totally ordered. Therefore every two elements in $E^{\rho}$, and hence in $E$, are comparable.

It is well-known (see \cite[Chapter~1, Exercise~14]{ab0}) that an Archimedean totally ordered vector space is $\{0\}$ or $\R$, showing the final statement.
\end{proof}

\subsection{Criteria for atomicity of a pervasive space}\label{sec_pervatom}

We call $e\in E_{+}$ an \emph{atom} if $\left[0_{E},e\right]\subseteq \R e$. Clearly, if $e$ is an atom, and $r>0$, then $re$ is an atom. If $E$ is Archimedean, then $e$ is an atom if and only if $\left[0_{E},e\right]$ is totally ordered (see \cite[Lemma~2.1]{lrw}). It is easy to see that atoms are preserved by linear order isomorphisms.

\begin{proposition}\label{pervatom}
Let $F$ be an Archimedean vector lattice, and let $E\subseteq F$ be an order dense subspace. Then, $f\in F_{+}$ is an atom in $F$ if and only if $f\in E$ and it is an atom in $E$.
\end{proposition}
\begin{proof}
Necessity: There is $e\in \left(0_{F},f\right]\cap E\subseteq\R f\cap E$, and so there is $r>0$ such that $rf\in E$. Thus, $f\in E$, and clearly it is an atom there.

Sufficiency: Assume that $f$ is not an atom in $F$, so that there are disjoint $g,h\in\left(0_{F},f\right]$. Find $u\in \left(0_{F},g\right]\cap E$ and $v\in \left(0_{F},h\right]\cap E$. It follows that $u,v\in \left[0_{F},f\right]_{E}$ are linearly independent, contradicting the fact that $f$ is an atom in $E$.
\end{proof}

We will call $E$ \emph{atomic} if for every $e>0_{E}$, $\left(0_{E},e\right]$ contains an atom.

\begin{corollary}\label{pervatomic}
Let $F$ be an Archimedean vector lattice, and let $E\subseteq F$ be an order dense subspace. Then $F$ is atomic if and only if $E$ is atomic. In particular, in this case every element of $E_{+}$ is the supremum of a collection of atoms.
\end{corollary}
\begin{proof}
Necessity: Every element of $F_{+}$ is a supremum of a collection of atoms in $F$. According to Proposition~\ref{pervatom}, atoms in $F$ are also atoms in $E$, and so every element of $E_{+}$ is a supremum in $F$ (hence in $E$) of a collection of atoms in $E$.

Sufficiency: Let $f>0_{F}$. There is $e\in\left(0_{F},f\right]\cap E$, and since $E$ is atomic, there is an atom $g\in\left(0_{F},e\right]\cap E$. Then, $0_{F}<g\le f$, and according to Proposition~\ref{pervatom}, $g$ is an atom in $F$. As $f$ was arbitrary, we conclude that $F$ is atomic.
\end{proof}

\begin{remark}
An extension of Corollary~\ref{pervatomic} and more information about atoms in pervasive spaces can be found in \cite{mz}.
\qed\end{remark}

We now proceed towards a characterization of atomicity in terms of order convergence extending the results from \cite{abt}. We start with some topological preliminaries. For a compact Hausdorff space $K$ we denote the vector lattice of all continuous functions on $K$ by $\Co\left(K\right)$. The indicator function of a set $A\subseteq K$ is denoted by $\1_{A}$; in particular, $\1:=\1_{K}$ and $\0:=\1_{\varnothing}$.

\begin{lemma}\label{pervurisohn}
Let $K$ be a compact Hausdorff space and let $E\subseteq \Co\left(K\right)$ be an order dense subspace and let $U_{1},\dots,U_{n}\subseteq K$ be nonempty open sets. Then, there is $f\in \left[\0,\1\right]_{E}$ which vanishes outside of $U_{1}\cup\dots\cup U_{n}$ and such that $f^{-1}\left(1\right)\cap U_{k}\ne\varnothing$, for every $k=1,\dots,n$.
\end{lemma}
\begin{proof}
First, assume that $U_{1},\dots,U_{n}$ are disjoint and fix $k=1,\dots,n$. By Urysohn's lemma there is $f_{k}\in \Co\left(K\right)\cap \left(\0,\1_{U_{k}}\right]$. Since $E$ is order dense in $\Co\left(K\right)$, we can find some $e_{k}\in\left(\0,f_{k}\right]\cap E$. In particular, $e_{k}$ vanishes outside of $U_{k}$, and so it attains its norm at some point $x_{k}\in U_{k}$. Then, $f:=\sum_{k=1}^{n}\frac{1}{\|e_{k}\|}e_{k}\in \left[\0,\1\right]_{E}$ vanishes outside of $U_{1}\cup\dots\cup U_{n}$, and moreover $f\left(x_{k}\right)=1$ for every $k=1,\dots,n$.

In the general case, we first find disjoint nonempty open sets $V_{1},\dots,V_{m}$ such that $V_{1}\cup\dots\cup V_{m}\subseteq U_{1}\cup\dots\cup U_{n}$ and for every $k\in\{1,\dots,n\}$ there is $l\in\{1,\dots,m\}$ so that $V_{l}\subseteq U_{k}$. Their existence can be proved by induction over $n$.

The case $n=1$ is trivial. If the claim is proved for $n$, and $U_{0},\dots,U_{n}$ are open and nonempty, by the induction hypothesis we can find $V_{1},\dots,V_{m}$ corresponding to $U_{1},\dots,U_{n}$. Let us show that the collection of all nonempty sets among $V_{i}\cap U_{0}$, $V_{i}\setminus \overline{U_{0}}$, $i\in\{1,\dots,m\}$, and $U_{0}\setminus \bigcup_{i=1}^{m}\overline{V_{i}}$ satisfies the requirements. For $j\in\{1,\dots,n\}$ find $i\in\{1,\dots,m\}$ such that $V_{i}\subseteq U_{j}$. Then, one of $V_{i}\cap U_{0}$, $V_{i}\setminus \overline{U_{0}}$ is nonempty and contained in $V_{i}$, hence in $U_{j}$. Also, either $V_{i}\cap U_{0}\ne\varnothing$, for some $i\in\{1,\dots,m\}$, or $U_{0}\setminus \bigcup_{i=1}^{m}\overline{V_{i}}$ is nonempty and contained in $U_{0}$. This verifies the conclusion of the induction step.

Now, according to the special case, there is $f\in \left[\0,\1\right]_{E}$ which vanishes outside of $V_{1}\cup\dots\cup V_{m}\subseteq U_{1}\cup\dots\cup U_{n}$ and such that $f^{-1}\left(1\right)\cap V_{l}\ne\varnothing$, for every $l\in\{1,\dots,m\}$. For every $k\in\{1,\dots,n\}$ find $l$ as above; it follows that $f^{-1}\left(1\right)\cap U_{k}\supseteq f^{-1}\left(1\right)\cap V_{l}\ne\varnothing$.
\end{proof}

\begin{proposition}\label{pervnoc}
Let $E\subseteq \Co\left(K\right)$ be an order dense subspace, where $K$ is a compact Hausdorff space without isolated points. Then there is a net in $\left[\0,\1\right]_{E}$ which has no order Cauchy (in $\Co\left(K\right)$) subnets.
\end{proposition}
\begin{proof}
There is a dense $Z\subseteq K$ such that $K\setminus Z$ is also dense (see \cite{hewitt}). Let $\mathcal{Z}$ and $\mathcal{Y}$ be the collections of all finite subsets of $Z$ and $K\setminus Z$, respectively, ordered by inclusion. For $z\in Z$ let $\mathcal{U}_{z}$ be the collection of all neighborhoods of $z$, ordered by reverse inclusion. Let $\Gamma:=\mathcal{Z}\times \mathcal{Y}\times \prod_{z\in Z}\mathcal{U}_{z}$, which is a directed set. Let $\gamma=\left(A,B,\left(U_{z}\right)_{z\in Z}\right)\in\Gamma$, where $A=\left\{z_{1},\dots,z_{n}\right\}\subseteq Z$. For every $k\in\{1,\dots,n\}$ we have that $V_{k}:=U_{z_{k}}\setminus B$ is a neighborhood of $z_{k}$. Applying Lemma~\ref{pervurisohn} to $V_{1},\dots,V_{n}$, there is $f_{\gamma}\in \left[\0,\1\right]_{E}$ which vanishes on $B$ and such that $f^{-1}_{\gamma}\left(1\right)\cap V_{k}\ne\varnothing$, for every $k\in\{1,\dots,n\}$.

Let $\left(f_{\gamma_{\delta}}\right)_{\delta\in\Delta}$ be a subnet of $\left(f_{\gamma}\right)_{\gamma\in\Gamma}$. We claim that for every $\delta_{0}\in\Delta$ we have $\bigvee_{\delta\ge\delta_{0}}f_{\gamma_{\delta}}=\1$ and $\bigwedge_{\delta\ge\delta_{0}}f_{\gamma_{\delta}}=\0$. Once this is established, it will follow immediately from Proposition~\ref{ocauchy} that $\left(f_{\gamma_{\delta}}\right)_{\delta\in\Delta}$ is not order Cauchy.

For the former claim, let $U\ne\varnothing$ be an open set. As $Z$ is dense, we can find $x\in Z\cap U$. There is $\delta\ge\delta_{0}$ such that $\gamma_{\delta}\ge\left(\left\{x\right\},\varnothing, \left(U_{z}\right)_{z\in Z}\right)$, where $U_{x}:=U$ and $U_{z}:=K$ if $z\ne x$. Then, $\gamma_{\delta}=\left(A,B,\left(V_{z}\right)_{z\in Z}\right)$, where $x\in A$ and $V_{x}\subseteq U_{x}=U$. We have that
\[f^{-1}_{\gamma_{\delta}}\left(1\right)\cap U\supseteq f^{-1}_{\gamma_{\delta}}\left(1\right)\cap V_{x}\ne\varnothing.\] Since $U$ was chosen arbitrarily, it follows from \cite[Lemma~3.1]{bt} that $\bigvee_{\delta\ge\delta_{0}}f_{\gamma_{\delta}}=\1$.

We now prove the latter claim. For every $y\in K\setminus Z$, there is $\delta\ge\delta_{0}$ such that $\gamma_{\delta}\ge\left(\varnothing,\left\{y\right\},\left(K\right)_{z\in Z}\right)$. Arguing as in the previous step, we see that $f_{\gamma_{\delta}}\left(y\right)=0$, and since $y$ was arbitrary, it is easy to conclude that $\bigwedge_{\delta\ge\delta_{0}}f_{\gamma_{\delta}}=\0$.
\end{proof}

Let $E$ be an ordered vector space, and let $A\subseteq E$. We call $A$ \emph{order compact} if every net in $A$ has a subnet that order converges to some $e\in A$. We say that $A$ is \emph{order pre-compact} if every net in $A$ has a subnet that is order Cauchy in $E$. Further, $A$ is called order \emph{locally (pre-)compact} if every net in $A$ which order converges in $E$ to an element of $A$ has a tail contained in an order (pre-)compact subset of $E$. It is easy to see that a subset of an order (locally) pre-compact set is order (locally) pre-compact. The following two results extend \cite[Theorems~4.5 and~4.9]{abt} and partially \cite[Theorem~9.13]{bctv}.

\begin{theorem}\label{ordprec}
For a pervasive Archimedean ordered vector space $E$ the following conditions are equivalent:
\item[(i)] $E$ is order locally pre-compact;
\item[(ii)] All order intervals are order pre-compact;
\item[(iii)] Order convergence is topological on the intervals in $E$;
\item[(iv)] Order convergence has the \emph{Choquet property}, i.e.\ if $\left(e_{\alpha}\right)_{\alpha\in A}$ is not order-null, it has a subnet which has no further order-null subnets;
\item[(v)] $E$ is atomic;
\item[(vi)] There is a set $X$ such that $E$ embeds as a subspace of $\R^{X}$ with $c_{00}\left(X\right)\subseteq E$.\medskip

Moreover, on order bounded subsets of $E$ order convergence agrees with pointwise convergence.
\end{theorem}
\begin{proof}
(i)$\Rightarrow$(ii): Fix $e\in E_{+}$, and let $\Gamma:=\N\times\left[0_{E},e\right]$, ordered by the first component. Define a net $\left(e_{\gamma}\right)_{\gamma\in\Gamma}$ by $e_{\left(n,f\right)}:=\frac{1}{n}f$. The tails of this net are of the form $\left[0_{E},\frac{1}{n}e\right]$, for some $n\in\N$. It follows that $e_{\gamma}\toru 0_{E}$, therefore $e_{\gamma}\too 0_{E}$, and so there is $n\in\N$ such that $\left[0_{E},\frac{1}{n}e\right]$ is order pre-compact. It is easy to deduce from this that every order interval in $E$ is order pre-compact.

(ii)$\Rightarrow$(i) follows from the fact that every order convergent net is eventually order bounded, hence contained in an interval, which is order pre-compact.

(iii)$\Rightarrow$(iv): Suppose that $\left(e_{\alpha}\right)_{\alpha\in A}$ is not order-null. If it is not eventually order bounded, it has a subnet with no further order bounded subnets (see \cite[Proposition~9.5]{bctv}; take into account the comment after \cite[Proposition~9.1]{bctv}). In particular, this subnet has no further order-null subnets.

If $\left(e_{\alpha}\right)_{\alpha\in A}$ is eventually order bounded, without loss of generality we may assume that $\left(e_{\alpha}\right)_{\alpha\in A}\subseteq \left[-e,e\right]$, for some $e\in E_{+}$. By assumption, there is a topology $\tau$ on $\left[-e,e\right]$ which agrees with order convergence. As the net is not order-null, there are a $\tau$-neighborhood $U$ of $0_{E}$ and a majorizing $A'\subseteq A$ such that $\left(e_{\alpha}\right)_{\alpha\in A'}\subseteq \left[-e,e\right]\setminus U$. It is clear that no further subnet of $\left(e_{\alpha}\right)_{\alpha\in A'}$ is order-null.

(ii),(iv)$\Rightarrow$(v): Assume that $e>0_{E}$ is such that $\left[0_{E},e\right]$ contains no atoms. The subspace $H:=\bigcup_{n\in\N}\left[-ne,ne\right]$ satisfies the conditions of Corollary~\ref{pervideal}. Hence, without loss of generality we may assume that $E=H$, and so $e$ is an order unit in $E$. Furthermore, it remains an order unit in $E^{\delta}$, which we therefore may assume to be isomorphic to $\Co\left(K\right)$, for some compact Hausdorff $K$ (see \cite[Theorem~45.4]{lz}). Thus, $E$ embeds as an order dense subspace of $\Co\left(K\right)$ and $e=\1$. By virtue of Proposition~\ref{oca}, order convergence in $E$ is the restriction of order convergence in $\Co\left(K\right)$. According to Proposition~\ref{pervatom}, $\Co\left(K\right)$ is atomless, and so $K$ has no isolated points (see \cite[Lemma~3.2]{abt}).

By Proposition~\ref{pervnoc} there is a net in $\left[\0,\1\right]_{E}$ which has no order Cauchy subnets. We conclude that $\left[0_{E},e\right]$ is not order pre-compact. This establishes (ii)$\Rightarrow$(v).

To prove (iv)$\Rightarrow$(v), we first adapt \cite[Lemma~9.12]{bctv} to order dense subspaces of $\Co\left(K\right)$. To that end, we construct $f_{\alpha}$ as in the proof of that lemma, except that instead of requiring $f_{\alpha}\left(t_{\alpha}\right)=1$ we only demand that $f_{\alpha}^{-1}\left(1\right)\cap U_{t_{\alpha}}^{\alpha}\ne\varnothing$ (which we can achieve by Lemma~\ref{pervurisohn}).

Now, according to this modified result, there is a net in $\left[\0,\1\right]_{E}$ which is not order-null but such that its every subnet contains a further order-null subnet. This shows that order convergence does not have the Choquet property.

(v)$\Rightarrow$(vi): It follows from Corollary~\ref{pervatomic} that $E^{\delta}$ is atomic, and so isomorphic to an order dense ideal in $\R^{X}$ for some set $X$ (see \cite[Theorems~1.40 and~1.78]{ab0}). As $E$ is order dense in $E^{\delta}$, it is order dense in $\R^{X}$, and so contains all atoms of the latter, in particular, $\1_{\left\{x\right\}}$ for every $x\in X$.

(vi)$\Rightarrow$(ii),(iii): First, let $F$ be the ideal generated by $E$ in $\R^{X}$. Then, by Proposition~\ref{pervid}, $E$ is order dense and majorizing, hence supremum-dense in $F$, and so by Proposition~\ref{ocp}, order convergence in $E$ is the restriction of order convergence in $F$. Now, by \cite[Proposition~2.5]{abt}, order convergence in order bounded sets in $F$ agrees with the pointwise convergence, which is topological, and these sets are order pre-compact.
\end{proof}

\begin{corollary}\label{automaticlat}
For a pervasive Archimedean ordered vector space $E$ the following conditions are equivalent:
\item[(i)] $E$ is order locally compact;
\item[(ii)] Every order interval in $E$ is order compact;
\item[(iii)] $E$ is an atomic Dedekind complete vector lattice;
\item[(iv)] There is a set $X$ such that $E$ embeds as an ideal of $\R^{X}$ with $c_{00}\left(X\right)\subseteq E$.
\end{corollary}
\begin{proof}
(i)$\Leftrightarrow$(ii) is proved analogously to the corresponding equivalence in Theorem~\ref{ordprec}. It is easy to see that order compactness of order intervals implies monotone completeness, and so in this case $E$ is an atomic Dedekind complete vector lattice, by virtue of Proposition~\ref{pervcomplete} and Theorem~\ref{ordprec}. Hence, we get (ii)$\Rightarrow$(iii), while (iii)$\Leftrightarrow$(iv)$\Rightarrow$(ii) were already proven in \cite{abt}.
\end{proof}

Recall that a metric space is compact if and only if every continuous real-valued function on it is bounded, and in this case every continuous real-valued function on it is uniformly continuous. One may wonder whether an analogous property can be used to characterize order compactness. The notion of uniform continuity with respect to a linear convergence structure is discussed in \cite[Section~2]{bctv}.

\begin{corollary}\label{ordco}
If $E$ is an atomic Dedekind complete vector lattice, then every order continuous function on order intervals in $E$ is bounded and order uniformly continuous.
\end{corollary}
\begin{proof}
By Corollary~\ref{automaticlat} and Theorem~\ref{ordprec}, the order intervals in $E$ are order compact, and order convergence in these intervals is induced by a linear topology. Hence, we may apply the properties of compact uniform spaces to order convergence in this case, and derive both claims.
\end{proof}

We now provide a partial converse to Corollary~\ref{ordco}, and an example showing that a full converse is false.

\begin{proposition}
If $E$ is an order continuous Banach lattice, then the following conditions are equivalent:
\item[(i)] $E$ is atomic;
\item[(ii)] For every $e\in E_{+}$, every order continuous real-valued function on $\left[0_{E},e\right]$ is bounded;
\item[(iii)] For every $e\in E_{+}$, every order continuous real-valued function on $\left[0_{E},e\right]$ is order uniformly continuous.
\end{proposition}
\begin{proof}
Recall that order continuous Banach lattices are Dedekind complete (see \cite[Theorem~3.24]{ab0}), and so (i)$\Rightarrow$(ii),(iii) follow from Corollary~\ref{ordco}.

Let us show that order (uniform) continuity is equivalent to norm (uniform) continuity. Clearly, since order convergence is stronger than norm convergence, norm (uniform) continuity implies order (uniform) continuity. To prove the converse, assume that $\varphi\colon\left[0_{E},e\right]\to \R$ is uniformly order continuous, but not uniformly norm continuous. Then, there are $\left(e_{n}\right)_{n\in\N}$ and $\left(f_{n}\right)_{n\in\N}$ in $\left[0_{E},e\right]$ and $\varepsilon>0$ such that $e_{n}-f_{n}\to 0_{E}$ in norm, but $\left|\varphi\left(e_{n}\right)-\varphi\left(f_{n}\right)\right|\ge\varepsilon$, for all $n\in\N$. There is a subsequence $\left(n_{m}\right)_{m\in\N}$ such that $e_{n_{m}}-f_{n_{m}}\too 0_{E}$ (see \cite[Lemma~1.1]{tt}), which contradicts order uniform continuity, since $\left|\varphi\left(e_{n_{m}}\right)-\varphi\left(f_{n_{m}}\right)\right|\ge\varepsilon$, for every $m\in\N$. The ``non-uniform'' case is done analogously.

Now the condition (ii) is equivalent to boundedness of every norm continuous function on $\left[0_{E},e\right]$, for every $e\in E_{+}$. This condition is equivalent to compactness of $\left[0_{E},e\right]$, for every $e\in E_{+}$, which is further equivalent to atomicity of $E$ (see \cite[Theorem~1]{lipecky}). Hence, we get (ii)$\Rightarrow$(i).

Finally, the condition (iii) is equivalent to uniform norm continuity of every norm continuous function on $\left[0_{E},e\right]$, for every $e\in E_{+}$. It is easy to see that every uniformly norm continuous function on a convex norm bounded set is bounded. Hence, we get (iii)$\Rightarrow$(ii).
\end{proof}

\begin{example}\label{ordconst}
Let us show that order continuous functions on order intervals in $\Co\left[0,1\right]$ are constant. Note that this is well known for linear functionals (see e.g.\ \cite[Exercise~1.25]{ab0}). Let $g\in \Co\left[0,1\right]_{+}$ and let $\varphi\colon\left[\0,g\right]\to\R$ be order continuous. Assume that $\0<h\le g$ is such that $\varphi\left(h\right)\ne \varphi\left(\0\right)$. Without loss of generality we may assume that $\varphi\left(\0\right)=0$ and $\varphi\left(h\right)=1$. Then, $U:=\varphi^{-1}\left(-\8,\frac{1}{2}\right)$ is an order open neighborhood of $\0$ in $[\0, g]$. Arguing as in the proof of \cite[Example~10.7]{btw} (but multiplying all the involved functions by $h$), we construct a dense $X\subseteq \left[0,1\right]$ and an increasing sequence $\left(f_{n}\right)_{n\in\N}\subseteq U$ such that $f_{n}\left(x\right)\to h\left(x\right)$, for every $x\in X$. It follows that $f_{n}\too h$, hence by order continuity we get $\frac{1}{2}>\varphi\left(f_{n}\right)\to \varphi\left(h\right)=1$, contradiction.
\qed\end{example}

We conclude the subsection with an example of an atomless supremum-dense subspace of an atomic vector lattice (see Example~\ref{stratomless} for the converse). In this space order intervals are directed order pre-compact.

\begin{example}\label{poly}
Let $K:=\left\{\frac{1}{m}\right\}_{m\in\N}\cup\left\{0\right\}$. Then, $\Co\left(K\right)\simeq c$ is an atomic vector lattice. Let $E$ be the space of the restrictions of polynomials to $K$. Since $E$ is norm dense in $\Co\left(K\right)$, for every $f\in \Co\left(K\right)$ and $\varepsilon>0$ there is $e\in E$ such that $\|f-e\|\le\varepsilon$, hence $f-2\varepsilon\1\le e - \varepsilon \1 \le f$. This clearly implies that $E$ is supremum-dense in $\Co\left(K\right)$.

Next, we prove that $E$ has no atoms. Let $\0\ne e\in E_{+}$ and let $f\colon K\to \R$ be defined by $f\left(t\right):=te\left(t\right)$. Then, $f\in \left[\0,e\right]$. Assume that $f=re$, for some $r\in \R$. Then, for every $t\in K$ with $t\ne r$ we have $te\left(t\right)=re\left(t\right)$, hence $e\left(t\right)=0$. Therefore, $e$ has infinitely many zeros, which is impossible for a nonzero polynomial.

Let us show that bounded pointwise convergence on $\left\{\frac{1}{m}\right\}_{m\in\N}$ implies directed order convergence in $E$. Note that since order intervals in $E$ are pointwise pre-compact (as subsets of order intervals in the atomic vector lattice), this proves that they are also directed order pre-compact. Let $\left(e_{\alpha}\right)_{\alpha\in A}\subseteq E\cap\left[\0,\1\right]$ be pointwise null on $\left\{\frac{1}{m}\right\}_{m\in\N}$. For each $n\in \N$ let $q_{n} = 2\left(1-\frac{t}{2}\right)^{n}$. We see that $\left(q_{n}\right)_{n\in \N}$ is a decreasing sequence which converges to $\0$ pointwise on the dense set $\left\{\frac{1}{n}\right\}_{n\in\N}$, hence $q_{n}\downarrow \0$. Note that $q_{n}\left(0\right)=2$ and $q_{n}\left(\frac{1}{m}\right) > 0$, for all $m,n\in \N$.

Fix $n\in \N$. The set $U := q_{n}^{-1}\left(1,+\8\right)$ is an open neighborhood of $0$, hence there is $k\in \N$ such that $\left\{\frac{1}{m}\right\}_{m\ge k}\subseteq U$.
For every $l\in \{1,\ldots,k-1\}$ we have $q_{n}\left(\frac{1}{l}\right) > 0$, and since $\left(e_{\alpha}\right)_{\alpha\in A}$ is pointwise null on $\left\{\frac{1}{m}\right\}_{m\in\N}$, there is $\alpha_{l}$ such that $e_{\alpha}\left(\frac{1}{l}\right)\le q_{n}\left(\frac{1}{l}\right)$, whenever $\alpha\ge\alpha_{l}$. Find $\alpha_{0}\ge\alpha_{1},\ldots ,\alpha_{k-1}$. Then, if $\alpha\ge\alpha_{0}$ and $l \in \{1,\ldots, k-1\}$, we have $e_{\alpha}\left(\frac{1}{l}\right)\le q_{n}\left(\frac{1}{l}\right)$; if $x\in \left\{\frac{1}{m}\right\}_{m\ge k}\cup\left\{0\right\}$, then $e_{\alpha}\left(x\right)\le 1<q_{n}\left(x\right)$. In summary, $e_{\alpha}\le q_{n}$ when $\alpha\ge\alpha_{0}$, and so $\left\{q_{n}\right\}_{n\in\N}$ witnesses directed order convergence of $\left(e_{\alpha}\right)_{\alpha\in A}$ to $\0$.
\qed\end{example}

\subsection{Relative uniform and $\sigma$-order compactness}\label{ss33}

Throughout this section, $E$ is a pervasive Archimedean ordered vector space. We say that $\left(e_{\alpha}\right)_{\alpha\in A}\subseteq E$ is \emph{(directed) $\sigma$-order convergent} to $e$ if there is a countable (downward directed) $Q\subseteq E$ with $\bigwedge Q=0_{E}$ which dominates the tails of $\left(e_{\alpha}\right)_{\alpha\in A}$. Clearly, directed $\sigma$-order convergence ($\mathrm{d\sigma o}$) is weaker than relative uniform convergence, but stronger than both $\sigma$-order ($\sigma\mathrm{o}$) and directed order convergences, while the latter two are stronger than order convergence. In other words, $\mathrm{o}\le \mathrm{do},\sigma\mathrm{o}\le \mathrm{d\sigma o}\le \mathrm{ru}$. Let us consider a class of spaces where, except for $\mathrm{ru}$, all these convergences coincide.

\begin{proposition}\label{csp}
The following conditions are equivalent:
\item[(i)] $E$ has the countable supremum property (CSP), i.e.\ if $Q\subseteq E$ is such that $\bigwedge Q=0_{E}$, there is a countable $Q'\subseteq Q$ such that $\bigwedge Q'=0_{E}$;
\item[(ii)] $E$ has the directed countable supremum property, i.e.\ if $Q\subseteq E$ is such that $Q\downarrow 0_{E}$, there is a countable $Q'\subseteq Q$ such that $\bigwedge Q'=0_{E}$ (in this case $Q'$ may be taken to be a decreasing sequence);
\item[(iii)] $E^{\delta}$ has the CSP;
\item[(iv)] Each order bounded disjoint set in $E$ is countable;
\item[(v)] Order convergence and $\sigma$-order convergence agree on $E$;
\item[(vi)] Order convergence and directed $\sigma$-order convergence agree on $E$.
\end{proposition}
\begin{proof}
(iii)$\Rightarrow$(i) follows from the fact that, for $Q\subseteq E$, we have $\bigwedge_{E} Q=0_{E}$ if and only if $\bigwedge_{E^{\delta}} Q=0_{E^{\delta}}$.

(i)$\Rightarrow$(ii): Existence of some countable $Q'\subseteq Q$ with $\bigwedge Q'=0_{E}$ follows immediately from (i). If $Q'=\left\{q_{n}\right\}_{n\in\N}$, we can construct a sequence $\left(e_{n}\right)_{n\in\N}\subseteq Q$ as follows. Let $e_{1}:=q_{1}$, and if $e_{1},\dots,e_{n}$ are already selected, then since $Q$ is directed downward, it contains some $e_{n+1}\le e_{n},q_{n+1}$. It is easy to see that $\bigwedge_{n\in\N}e_{n}=0_{E}$.

(ii)$\Rightarrow$(iii): It follows from Theorem~\ref{pervasive} that $E$ is an order dense subspace of $E^{\delta}$. Suppose that $P\subseteq E^{\delta}$ is such that $\bigwedge P=0_{E^{\delta}}$. According to Corollary~\ref{pervdird} applied to $P^{\wedge}$, there is $Q\subseteq E\cap \bigcup_{p\in P^{\wedge}}\left(p+E_{+}^{\delta}\right)$ such that $Q\downarrow 0_{E^{\delta}}$. By our assumption, there is a countable $Q'\subseteq Q$ such that $Q'\downarrow 0_{E}$. For every $q\in Q'$ there is a finite $P_{q}\subseteq P$ such that $\bigwedge P_{q}\le q$. It follows now that $P':=\bigcup_{q\in Q'}P_{q}$ is countable, and it is easy to see that $\bigwedge P'=0_{E^{\delta}}$.

(iii)$\Leftrightarrow$(iv): According to \cite[Theorem~29.3]{lz}, $E^{\delta}$ has the CSP if and only if each disjoint order bounded set there is countable. Clearly, this condition implies (iv). For the converse, note that if $D\subseteq E^{\delta}_{+}\setminus\left\{0_{E^{\delta}}\right\}$ is disjoint, uncountable and order bounded, then for every $d\in D$ there is $e_{d}\in E\cap \left(0_{E^{\delta}},d\right]$; then $\left\{e_{d}:d\in D\right\}$ is disjoint, uncountable and order bounded in $E$, which is a contradiction.

(ii)$\Rightarrow$(vi) follows from the definitions and the fact that order and directed order convergences agree on pervasive spaces (Proposition~\ref{o=do}). (vi)$\Rightarrow$(v) follows from the fact that $\sigma$-order convergence is stronger than order convergence, but weaker than directed $\sigma$-order convergence.

(v)$\Rightarrow$(ii): Assume that $Q\subseteq E$ is such that $Q\downarrow 0_{E}$. Then, $Q\too 0_{E}$, hence $Q\toso 0_{E}$, and so there is a countable $P\subseteq E$ such that $\bigwedge P=0_{E}$, and $P$ dominates the tails of $Q$. This means that for every $p\in P$ there is $q_{p}\in Q$ with $q_{p}\le p$. It now follows that $Q':=\left\{q_{p}:p\in P\right\}$ is a countable subset of $Q$ with $\bigwedge Q'=0_{E}$.
\end{proof}

The following example exhibits a pervasive Archimedean space with an order unit such that $\sigma\mathrm{o}<\mathrm{d\sigma o}$.

\begin{example}\label{sosdo}
Recall that if $\mathcal{F}$ is a filter on a set $X$, then $r\in \R$ is the limit of $f\colon X\to\R$ along $\mathcal{F}$ if for every $\varepsilon>0$ we have $f^{-1}\left(r-\varepsilon,r+\varepsilon\right)\in\mathcal{F}$. It is easy to see that the limit defined this way is unique, and if $\mathcal{F}$ is an ultrafilter and $f$ is bounded, then the limit along $\mathcal{F}$ always exists (see \cite[17.5(ii)]{schechter}).

Let $F$ be the space of all bounded Borel functions on $\left(0,1\right]$. Since $\left(\left(0,r\right]\right)_{0<r\le 1}$ is a decreasing net of subsets of $\left(0,1\right]$, there is an ultrafilter $\mathcal{U}$ on $\left(0,1\right]$ which contains each subset. Taking a limit with respect to $\mathcal{U}$ defines a vector lattice homomorphism from $\ell_{\8}\left(0,1\right]$ into $\R$. We denote its restriction to $F$ by $\varphi$. For every $t\in \left(0,1\right]$ we have that $\varphi\left(\1_{\left(0,t\right]}\right)=\varphi\left(\1_{\left(0,t\right)}\right)=1$, hence $\1_{\left\{t\right\}}\in\ker\varphi$.

Let us show that $E:=\left\{f\in F: \varphi\left(f\right)=\int f\right\}$ is order dense in $F$. Indeed, if $f>\0$, there is $t\in \left(0,1\right]$ such that $r:=f\left(t\right)>0$. Then, $f\ge r\1_{\left\{t\right\}}$, and since $\1_{\left\{t\right\}}\in\ker\varphi$ and $\int \1_{\left\{t\right\}}=0$, we conclude that $r\1_{\left\{t\right\}}\in E$. Also, clearly $\1\in E$ is an order unit.

We order $\mathcal{U}$ by reverse inclusion, and for $U\in\mathcal{U}$ let $x_{U}$ be an arbitrary point in $U$. Let us show that for every $f\in F$ we have $\lim_{U\in \mathcal{U}}f\left(x_{U}\right)=\lim_{\mathcal{U}}f=\varphi\left(f\right)$. Indeed, for every $\varepsilon>0$, by the definition of the limit with respect to a filter, we get
\[U_{\varepsilon}:=f^{-1}\left(\varphi\left(f\right)-\varepsilon,\varphi\left(f\right)+\varepsilon\right)\in\mathcal{U};\]
then, for every $U\subseteq U_{\varepsilon}$ we have $x_{U}\in U\subseteq U_{\varepsilon}$, and so $\left|f\left(x_{U}\right)-\varphi\left(f\right)\right|<\varepsilon$. Therefore, for every $f\in E$ we have $\int f=\varphi\left(f\right)=\lim_{U\in \mathcal{U}}f\left(x_{U}\right)$.

Let us show that $\left(\1_{\left\{x_{U}\right\}}\right)_{U\in\mathcal{U}}$ is not $\mathrm{d\sigma o}$-null. Assume that $\left\{h_{n}\right\}_{n\in\N}\subseteq E$ is decreasing with $\bigwedge_{n\in\N}h_{n}=\0$ and dominates the tails of our net. For every $n\in\N$ we have that $\1_{\left\{x_{U}\right\}}\le h_{n}$, hence $h_{n}\left(x_{U}\right)\ge 1$ for small enough $U$. It follows that $\int h_{n}=\lim_{U\in \mathcal{U}}h_{n}\left(x_{U}\right)\ge 1$, for every $n\in\N$. On the other hand, as $E$ is order dense in $F$, it follows from Corollary~\ref{odm} that $\bigwedge_{n\in\N}h_{n}=\0$ in $F$. Since the infimum of a countable collection of Borel functions is its pointwise infimum, we conclude that $h_{n}\left(x\right)\downarrow 0$, for every $x\in\left(0,1\right]$. As a result, we get $1\le \int h_{n}\downarrow 0$, contradiction.

Finally, we prove that $\1_{\left\{x_{U}\right\}}\toso\0$. For $n\in\N$ let $A_{n}:=\left(0,\frac{1}{2^{n}}\right]$, $B_{n}:=\left(\frac{1}{2}, \frac{1}{2}+\frac{1}{2^{n}}\right]$ and $g_{n}:=\1_{A_{n}}+\left(2^{n}-1\right)\1_{B_{n}}$. We have that $\varphi\left(g_{n}\right)=1=\int g_{n}$, hence $g_{n}\in E$, for every $n\in\N$. Moreover, $\bigwedge_{n\in\N}g_{n}=\0$ pointwise, hence in $E$. For $n\in\N$ take $A_{n}\in\mathcal{U}$ as an index. Then, for every $U\subseteq A_{n}$ we have $x_{U}\in U\subseteq A_{n}$, hence $\1_{\left\{x_{U}\right\}}\le \1_{A_{n}}\le g_{n}$. Thus, $\left\{g_{n}\right\}_{n\in\N}$ dominates the tails of our net.
\qed\end{example}

We denote the collection of all functions on a set $X$ whose support is countable by $\R^{X}_{csp}$. We define (local) [pre-]compactness with respect to $\sigma$-order or relative uniform convergences in the same way as with respect to order convergence. Clearly, relative uniform compactness implies $\sigma$-order compactness, and so on.

\begin{proposition}\label{socomp}
The following conditions are equivalent:
\item[(i)] $E$ is $\sigma$-order locally pre-compact;
\item[(ii)] Every order interval in $E$ is $\sigma$-order pre-compact;
\item[(iii)] $E$ is atomic and has the CSP;
\item[(iv)] There is a set $X$ such that $E$ embeds as a subspace of $\R^{X}_{csp}$ with $c_{00}\left(X\right)\subseteq E$.
\end{proposition}
\begin{proof}
(i)$\Leftrightarrow$(ii) is proved analogously to the corresponding equivalence in Theorem~\ref{ordprec}.

(ii)$\Rightarrow$(iii): Let us first show that $E$ satisfies the condition (ii) of Proposition~\ref{csp}. Assume that $Q\downarrow 0_{E}$. Without loss of generality we may assume that this net is contained in an order interval, and so contains a $\sigma$-order Cauchy subnet $\left(e_{\alpha}\right)_{\alpha\in A}$. There is a countable $P\subseteq E$ with $\bigwedge P=0_{E}$ which dominates the tails of $\left(e_{\alpha}\right)_{\alpha\in A}$. For every $p\in P$ there is $\alpha_{p}$ such that $e_{\alpha_{p}}-e_{\alpha}\le p$, for every $\alpha\ge\alpha_{p}$; in fact, since $e_{\alpha}\downarrow 0_{E}$, it follows that $e_{\alpha_{p}}\le p$. Then, $Q':=\left\{e_{\alpha_{p}}:p\in P\right\}$ is a countable subset of $Q$ with $\bigwedge Q'=0_{E}$.

It now follows from Proposition~\ref{csp} that order and $\sigma$-order convergences agree on $E$. In particular, every order interval in $E$ is order pre-compact, and so $E$ is atomic, according to Theorem~\ref{ordprec}. The converse is proved by reversing the last two arguments.

(iii)$\Rightarrow$(iv): According to Theorem~\ref{ordprec}, $E$ is isomorphic to a subspace of $\R^{X}$, for some set $X$, with $\1_{\left\{x\right\}}\in E$ for every $x\in X$. Since $E$ also has the CSP, in light of condition (iv) in Proposition~\ref{csp}, $E\subseteq\R^{X}_{csp}$. The converse is straightforward.
\end{proof}

\begin{corollary}\label{socomp2}
The following conditions are equivalent:
\item[(i)] $E$ is $\sigma$-order locally compact;
\item[(ii)] Every order interval in $E$ is $\sigma$-order compact;
\item[(iii)] $E$ is an atomic Dedekind complete vector lattice with the CSP;
\item[(iv)] There is a set $X$ such that $E$ embeds as an ideal of $\R^{X}_{csp}$ with $c_{00}\left(X\right)\subseteq E$.
\end{corollary}
\begin{proof}
(i)$\Leftrightarrow$(ii) is proved analogously to the corresponding equivalence in Theorem~\ref{ordprec}. (iii)$\Rightarrow$(iv) follows from combining Proposition~\ref{socomp} with \cite[Theorem~1.40]{ab0}, while the converse is straightforward.

(ii)$\Rightarrow$(iii): It follows from Proposition~\ref{socomp} that $E$ has the CSP, and so order and $\sigma$-order convergences agree on $E$. Hence, order intervals in $E$ are order compact, and so, by Corollary~\ref{automaticlat}, $E$ is an atomic Dedekind complete vector lattice. The converse is proved by reversing these arguments.
\end{proof}

\begin{remark}
The analogues of the conditions (i) and (ii) in Proposition~\ref{socomp} and Corollary~\ref{socomp2} but for directed $\sigma$-order convergence can be added to the lists.
\qed\end{remark}

For $f,g\in\R^{X}$ we say that $g$ \emph{majorates} $f$ if $\supp f\subseteq\supp g$, and for every $n\in\N$ the set $\left\{x\in X: g\left(x\right)<n\left|f\left(x\right)\right|\right\}$ is finite. Note that the union of these sets is the support of $f$, and so $f\in \R^{X}_{csp}$.

\begin{proposition}\label{rucomp}
The following conditions are equivalent:
\item[(i)] $E$ is relatively uniformly locally pre-compact;
\item[(ii)] Every order interval in $E$ is relatively uniformly pre-compact;
\item[(iii)] $E$ is atomic and such that order and relative uniform convergences agree on $E$;
\item[(iv)] There is a set $X$ such that $E$ embeds as a subspace of $\R^{X}$ with $c_{00}\left(X\right)\subseteq E$ and such that every $e\in E$ is majorated by an element of $E$.
\end{proposition}
\begin{proof}
(i)$\Leftrightarrow$(ii) is proved analogously to the corresponding equivalence in Theorem~\ref{ordprec}.

(ii)$\Rightarrow$(iii): Arguing as in the proof of the corresponding implication in Proposition~\ref{socomp}, one can show that any $Q\downarrow 0_{E}$ contains an $\mathrm{ru}$-null sequence. This fact can then be leveraged to show that $\mathrm{ru}=\mathrm{do}$ on $E$. By Proposition~\ref{o=do}, we conclude that $\mathrm{ru}=\mathrm{o}$ on $E$, and in particular, order intervals in $E$ are order pre-compact. Thus, $E$ is atomic, by virtue of Theorem~\ref{ordprec}. By reversing the latter two arguments we get (iii)$\Rightarrow$(ii).

(iii)$\Rightarrow$(iv): First, since $E$ is atomic, by Theorem~\ref{ordprec} we may assume that it is a subspace of $\R^{X}$ for some set $X$, with $\1_{\left\{x\right\}}\in E$ for every $x\in X$. Moreover, order convergence agrees with pointwise convergence in order bounded subsets of $E$. The last condition is trivially satisfied if $X$ is finite, and so we assume otherwise.

Let $e\in E_{+}$. Let $A$ be the collection of all finite subsets of $X$. For $\alpha\in A$ define $e_{\alpha}\colon X\to\R$ by $e_{\alpha}\left(x\right):=\frac{e\left(x\right)}{\left|\alpha\right|}$ if $x\in\alpha$, and $e_{\alpha}\left(x\right):=e\left(x\right)$, otherwise. Then, $\supp \left(e-e_{\alpha}\right)$ is finite, hence $e-e_{\alpha}\in E$, and so $e_{\alpha}\in E$.

Clearly, $e_{\alpha}\to\0$ pointwise, therefore in order, hence relatively uniformly. Thus, there is $g\in E_{+}$ such that for every $n\in\N$ there is $\alpha$ such that $e_{\alpha}\le \frac{1}{n}g$, and so, for all $x\notin\alpha$, we have $e\left(x\right)=e_{\alpha}\left(x\right)\le \frac{1}{n}g\left(x\right)$. Since $\alpha$ is a finite set, it follows that $g$ majorates $e$.

If $f\in E$, find $e\in E_{+}$ such that $\pm f\le e$, and $g\in E_{+}$ which majorates $e$. It is clear that $g$ also majorates $f$.

(iv)$\Rightarrow$(iii): Suppose that $e_{\alpha}\downarrow \0$. Without loss of generality we may assume that this net is bounded above by some $e\in E_{+}$. Let $g$ be a majorant of $e$. Fix $n\in\N$. There are $x_{1},\dots,x_{m}\in X$ such that $\frac{1}{n}g\left(x\right)\ge e\left(x\right)$, for all $x\notin\left\{x_{1},\dots,x_{m}\right\}$. Since $\left(e_{\alpha}\right)$ converges to $\0$ pointwise, there is $\alpha_{n}$ such that $e_{\alpha}\left(x_{k}\right)\le \frac{1}{n}g\left(x_{k}\right)$, whenever $\alpha\ge\alpha_{n}$ and $k\in\left\{1,\dots,m\right\}$. It now follows that $e_{\alpha}\le \frac{1}{n}g$ as soon as $\alpha\ge\alpha_{n}$. Since $n$ was arbitrary, we conclude that $e_{\alpha}\toru\0$.
\end{proof}

Analogously to Corollary~\ref{socomp2}, we can derive the following consequence from Proposition~\ref{rucomp}.

\begin{corollary}\label{rucomp2}
The following conditions are equivalent:
\item[(i)] $E$ is relatively uniformly locally compact;
\item[(ii)] Every order interval in $E$ is relatively uniformly compact;
\item[(iii)] $E$ is an atomic Dedekind complete vector lattice such that order and relative uniform convergences agree on $E$;
\item[(iv)] There is a set $X$ such that $E$ embeds as an ideal of $\R^{X}$ with $c_{00}\left(X\right)\subseteq E$ and such that every $e\in E$ is majorated by an element of $E$.
\end{corollary}

\begin{example}
Both $\R^{X}_{csp}$ and $c_{0}\left(X\right)$ satisfy the conditions of Corollary~\ref{rucomp2} for arbitrary $X$. Indeed, any $f\in c_{0}\left(X\right)$ is majorated by $\sqrt{\left|f\right|}$; any $f\in \R^{X}_{csp}$ is majorated by $\left|f\right|g$, where $\supp f=\left\{x_{n}\right\}_{n\in\N}$, and $g\in\R^{X}$ is such that $g\left(x_{n}\right)=n$. Another example is the space of all real sequences of polynomial growth. The atomic order continuous Banach lattices form another class of examples, since on such spaces order and relative uniform convergences agree (see \cite[Proposition~3]{bw}). An example of a space which satisfies the conditions of Corollary~\ref{socomp2} but not of Proposition~\ref{rucomp} is $\ell_{\8}$.
\qed\end{example}

\subsection{Pervasive JB-algebras are associative}\label{ss34}

In this section we consider a special class of ordered vector spaces, namely JB-algebras, and show that pervasiveness is a rather restrictive condition in this class. We first need to introduce a few general notions.

Let $E$ be an Archimedean ordered vector space. We say that $e\in E_{+}$ is an \emph{order unit} if for each $f \in E$ there is $r \in \R_{+}$ with $f \in \left[-r e, r e\right]$. In this case, $\|\cdot\|_{e}$, defined by $\|f\|_{e}:=\bigwedge\left\{r>0: f \in \left[-r e, r e\right]\right\}$, is a norm on $E$ called the \emph{order unit norm}. It is easy to see that $E_{+}$ is closed with respect to $\|\cdot\|_{e}$. Convergence in this norm agrees with relative uniform convergence in $E$, and so does not depend on the choice of the order unit. Hence, relative uniform convergence in $E$ is topological. One may show that $f\in E$ is an order unit if and only if it is an element of the interior of $E_{+}$ with respect to this topology.

An \emph{order unit space} is a pair $\left(E,e\right)$, where $E$ is an Archimedean ordered vector space and $e$ is an order unit. We always assume that $E$ is equipped with the norm $\|\cdot\|_{e}$. An operator from $\left(E,e\right)$ into $\left(F,f\right)$ is called \emph{unital} if it maps $e$ to $f$. The ``natural'' notion of a morphism between order unit spaces is that of a positive unital operator. We view $\R$ as an order unit space $\left(\R,1\right)$, and call positive unital functionals from $\left(E,e\right)$ to $\left(\R,1\right)$ \emph{states}.

The set $K_{E}$ of states on $(E, e)$ is a $\mathrm{w}^{*}$-closed convex subset of $\Ba_{E^{*}}$ and hence is a $\mathrm{w}^{*}$-compact convex set; its extreme points are called \emph{pure states}. It is known that $E^{*}=\spa K_{E}$ (see \cite[Definition~1.10 and Theorem~1.19]{as2}).

Let $F$ be a vector space and $A\subseteq F$ be convex. We denote the collection of extreme points of $A$ by $\ext A$. We call $f\colon A\to\R$ \emph{affine} if it preserves convex combinations. If $\tau$ is a topology on $A$, we denote the ordered vector space of all $\tau$-continuous affine functions on $A$ by $\Aff\left(A,\tau\right)$ (the order is pointwise). If $A$ is $\tau$-compact, we view $\Aff\left(A,\tau\right)$ as an order unit space with $\1$ playing the role of a unit. Then, the order unit norm is equal to the supremum norm $\|\cdot\|_{\8}$, and $\Aff\left(A,\tau\right)$ is a Banach space in this norm.

It is clear that if $\left(E,e\right)$ is an order unit space, then $f\mapsto\hat{f}$ is a positive operator from $E$ into $\Aff\left(K_{E},\mathrm{w}^{*}\right)$, where $\hat{f}\left(\varphi\right) := \varphi\left(f\right)$. This can be refined as follows (see \cite[Theorem~1.20]{as2} and \cite[Section~2.5]{kv}; part~(i) is known as Kadison's embedding theorem).

\begin{theorem}\label{T:rep of o.u.s.}\interlinepenalty=10000
Let $\left(E,e\right)$ be an order unit space. Then:
\item[(i)] $f\mapsto \hat{f}$ is an isometric order isomorphism from $E$ onto a norm-dense subspace of $\Aff\left(K_{E},\mathrm{w}^{*}\right)$.
\item[(ii)] $f\mapsto \hat{f}|_{\overline{\ext}^{\mathrm{w}^{*}}K_{E}}$ is an isometric order isomorphism from $E$ onto a supremum-dense subspace of $\Co\left(\overline{\ext}^{\mathrm{w}^{*}}K_{E},\mathrm{w}^{*}\right)$.
\end{theorem}

We proceed with a brief review of JB-algebras. More details can be found in e.g.\ \cite{as}, \cite{hs}, \cite{lrw}, \cite{rt}, \cite{vkr}.

A \emph{Jordan algebra} $(A, \circ)$ is a commutative, not necessarily associative algebra such that
\[a \circ \left(b \circ a^2\right) = \left(a \circ b\right) \circ a^2 \quad\text{for all }a,b \in A.\]

A \emph{JB-algebra} $A$ is a normed, complete real Jordan algebra satisfying
\[\|a \circ b\| \le \|a\|\|b\|,\qquad \|a^2\| =\|a\|^2,\qquad\|a^2\| \le\|a^2 + b^2\|,\]
for all $a,b \in A$. An important example of a JB-algebra is the set of self-adjoint elements of a $C^*$-algebra equipped with the Jordan product $a \circ b := \frac{1}{2}\left(ab + ba\right)$.

A linear operator between JB-algebras is a \emph{Jordan homomorphism} if it preserves the product and the unit (if it exists).

%By the Gelfand-Naimark theorem, this JB-algebra is a norm closed Jordan subalgebra of the Jordan algebra of self-adjoint bounded operators on a complex Hilbert space. Such a JB-algebra is called a \emph{JC-algebra}.

%The elements $a,b \in A$ are said to \emph{operator commute} if $a \circ \left(b \circ c\right) = b \circ \left(a \circ c\right)$ for all $c \in A$. %In a JC-algebra, two elements operator commute if and only if they commute for the $C^*$-multiplication, see \cite[Proposition~1.49]{as}. An element $a\in A$ is said to be \emph{central} if it operator commutes with all elements of $A$.

Assume that $A$ has a unit $1_{A}$. We say that $a\in A$ is \emph{invertible} if it is an invertible element of $\mathrm{JB}\left(a,1_{A}\right)$, the JB-algebra generated by $a$ and $1_{A}$ (note that the latter is associative by \cite[Corollary~1.4]{as}; a characterization of invertibility can be found in \cite[1.15--1.17]{as}). %We denote the set of invertible elements of $A$ by $\mathrm{Inv}\left(A\right)$.

The \emph{spectrum} of $a\in A$, which is denoted by $\sigma\left(a\right)$, is defined to be the set of $\lambda\in\R$ such that $\lambda 1_{A} - a\notin \mathrm{Inv}\left(A\right)$. There is a continuous functional calculus, that is, $\mathrm{JB}\left(a,1_{A}\right)$ is isometrically and algebraically isomorphic to $\Co\left(\sigma\left(a\right)\right)$, see \cite[Theorem~3.2.4]{hs}.

The cone of elements with non-negative spectrum is denoted by $A_{+}$, and equals the set of squares by the functional calculus, and its interior $\mathring{A}_{+}$ consists of all elements with strictly positive spectrum. The cone $A_{+}$ induces a partial ordering $\le$ on $A$ by writing $a \le b$ if $b - a \in A_{+}$. Then, $1_{A}$ is an order unit of $A$, and so $\left(A,1_{A}\right)$ is an order unit space. A non-unital JB-algebra can be endowed with the partial order inherited from its unitization, see \cite[Section~3.3]{hs}, which is a unital JB-algebra that contains $A$ as a JB-ideal of codimension $1$.

For every $a\in A$ the functional calculus yields $a^{+},a^{-}$ in the JB-subalgebra generated by $a$, with $a=a^{+}-a^{-}$.

\begin{lemma}\label{apm}
We have $a^{+}\wedge_{A_{+}}a^{-}=0_{A}$.
\end{lemma}
\begin{proof}
First, it is clear that $U_{a^{+}}a^{-}=0_{A}=U_{a^{-}}a^{+}$. Assume that $0_{A}\le b\le a^{+},a^{-}$. Then, $0_{A}\le U_{a^{+}}b\le U_{a^{+}}a^{-}=0_{A}$. Hence, $U_{a^{+}}b=0_{A}$, implying $U_{b}a^{+}=0_{A}$ (see \cite[Lemma~1.26]{as}). Therefore, $0_{A}\le b^{3}=U_{b}b\le U_{b}a^{+}=0_{A}$, thus $b^{3}=0_{A}$, and so $b=0_{A}$.
\end{proof}

We call a linear functional on $A$ \emph{normal} if for $\left(a_{\gamma}\right)_{\gamma\in\Gamma}\subseteq A$ with $a_{\gamma}\downarrow 0_{A}$ it follows that $\varphi\left(a_{\gamma}\right)\to 0$. Clearly, every directed order continuous functional is normal, and the converse is true for positive functionals, according to Proposition~\ref{docont}.

A \emph{JBW-algebra} is the Jordan analogue of a von Neumann algebra: it is a unital JB-algebra which is monotone complete (recall that this means that every order bounded increasing net has a supremum\footnote{In \cite{hs} the term ``monotone complete'' means that every norm-bounded increasing net has a supremum. Since we are assuming existence of an order unit, boundedness in order and in norm agree.}) and has a separating set of normal states, or equivalently, a JB-algebra that is a dual space \cite[Theorem~4.4.16]{hs}. In this case we have the following result.

\begin{proposition}\label{normal}
A functional on a JBW-algebra is normal if and only if it is directed order continuous.
\end{proposition}
\begin{proof}
Sufficiency was already addressed above. For necessity, note that by \cite[Proposition~2.52]{as} any normal functional is a difference of positive normal functionals. The latter two are directed order continuous by Proposition~\ref{docont}, and so is their difference.
\end{proof}

The set of normal states on $A$ is called the \emph{normal state space} of $A$. The topology on $A$ defined by the duality of $A$ and the normal state space of $A$ is called the \emph{$\sigma$-weak topology}. That is, we say a net $\left(a_{\gamma}\right)_{\gamma\in\Gamma}\subseteq A$ converges $\sigma$-weakly to $a$ if $\varphi\left(a_{\gamma}\right)\to\varphi\left(a\right)$ for all normal states $\varphi$ on $A$.

Since the linear span $N$ of the normal state space is the unique predual of $A$ by \cite[Theorem~4.4.16]{hs}, the $\sigma$-weak topology equals the $\mathrm{w}^{*}$-topology. By \cite[Corollary~2.17]{as}, the normal state space of a JBW-algebra determines the ordering on $A$, that is, for each $a \in A$, $\psi\left(a\right) \ge 0$ for all normal states $\psi$ if and only if $a \in A_{+}$. Therefore, $A$ is isomorphic to the dual of $N$ not only as a Banach space, but also as an ordered Banach space.

In the following result (ii)$\Leftrightarrow$(iii) is \cite[Theorem~15]{ilr} (\cite[Theorem~1]{sherman} for the case of C*-algebras), but we provide a simpler proof. A version of (ii)$\Leftrightarrow$(iv) for C*-algebras was established in \cite{fmt} (see also \cite{glueck} and \cite{gs}).

\begin{theorem}\label{jblat}
For a JB-algebra $A$ the following conditions are equivalent:
\item[(i)] $A$ is Jordan isomorphic to $\Co_{0}\left(X\right)$ for some locally compact Hausdorff space $X$;
\item[(ii)] $A$ is associative;
\item[(iii)] $A$ is a vector lattice;
\item[(iv)] $A$ satisfies the Riesz decomposition property, i.e.\ $\left[0_{A},a\right]+\left[0_{A},b\right]=\left[0_{A},a+b\right]$, for every $a,b\in A_{+}$;
\item[(v)] $A^{**}$ is a vector lattice;
\item[(vi)] $A$ is pervasive;
\item[(vii)] $A$ is weakly pervasive.
\end{theorem}
\begin{proof}
(i)$\Rightarrow$(ii), as well as (i)$\Rightarrow$(iii)$\Rightarrow$(iv) and (iii)$\Rightarrow$(vi) are trivial. (vi)$\Rightarrow$(vii) follows from Theorem~\ref{pervasive} or results in \cite{km}. (ii)$\Rightarrow$(i) is \cite[Theorem~3.2.2]{hs}. (vii)$\Rightarrow$(iii) follows from combining Lemmas~\ref{wp} and~\ref{apm}.

%For (i)$\Rightarrow$(vi) note that the unitization of $\Co_{0}\left(X\right)$ is $\Co\left(X_{\8}\right)$, where $X_{\8}$ is the one-point compactification of $A$.\medskip
%It is clear that if $A$ is a vector lattice, then it is pervasive. Conversely, if $A$ is a JBW-algebra then it is monotone complete, hence if additionally it is pervasive, by Proposition~\ref{pervcomplete} it is a Dedekind complete vector lattice.\medskip

(iv)$\Rightarrow$(v): First, note that the order on $A^{**}$ as a JB-algebra is determined by normal states, and in particular, is consistent with its order as a dual of an ordered vector space $A^{*}$. Since for every $a\in A$ we have that $A_{+}\ni \sqrt{a^{2}}\ge a$ (here we use \cite[Theorem~3.2.2]{hs} for $\mathrm{JB}\left(a\right)$), it follows that $A$ is a directed ordered vector space. Moreover, since $A$ satisfies the Riesz decomposition property, $A_{+}$ is norm-closed in $A$, and the norm is monotone on $A_{+}$, by \cite[Theorem~2.47 and Exercise~2.17]{at}, it follows that $A^{*}$ is a vector lattice endowed with a norm equivalent to a Banach lattice norm. We conclude that $A^{**}$ is also a lattice (see \cite[Theorem~4.1]{ab}).

(v)$\Rightarrow$(ii): Since $A$ is a subalgebra of $A^{**}$, it is enough to show that $A^{**}$ is associative. Moreover, recall that the latter is always unital. Let $K_{A^{**}}$ be the space of states on $A^{**}$, and let $K$ be the collection of the pure states, i.e.\ the extreme points of $K_{A^{**}}$. According to \cite[Theorem~2.5.6]{kv}, $K$ is $\mathrm{w}^{*}$-closed in $K_{A^{**}}$ and the representation $a\mapsto \hat{a}$ as in Theorem~\ref{T:rep of o.u.s.} is an isometric vector lattice homomorphism onto $\Co\left(K\right)$. Moreover, $1_{A^{**}}$ is mapped into $\1$. Since any unital isometry between JB-algebras is a Jordan isomorphism (see \cite[Theorem~4]{wy}), $A^{**}$ is isomorphic to $\Co\left(K\right)$, and in particular it is associative.
\end{proof}

\section{Ice cream spaces}\label{s4}

In this section, we will examine the concepts of (directed) order convergence, pervasiveness and atomicity in an explicit class of ordered vector spaces, the so-called \emph{ice cream spaces}. In Subsection~\ref{icebsc}, for a Banach space $F$, we construct the ice cream space $\icecream\left(F\right)$, identify it with the space $\Aff\left(\Ba_{F^{*}}, \mathrm{w}^{*}\right)$ of $\mathrm{w}^{*}$-continuous affine functions on $\Ba_{F^{*}}$, and show that directed order and norm convergence agree on $\icecream\left(F\right)$. In Subsection~\ref{iceoc}, we show that under mild hypotheses on $F$, order convergence and the order topology on $\Aff\left(\Ba_{F^{*}}, \mathrm{w}^{*}\right)$ coincide with continuous convergence and the compact-open topology, respectively. In Subsection~\ref{icelat}, we study to which extent $\icecream\left(F\right)$ has lattice structure. In Subsection~\ref{iceatm}, we give a necessary and sufficient condition on $F$ for $\icecream\left(F\right)$ to be (strongly) atomic.

\subsection{Preliminaries}\label{icepre}

It follows from part~(i) of Theorem~\ref{T:rep of o.u.s.} that any complete order unit space is, up to an isomorphism, the space of affine continuous functions on some compact convex subset of a locally convex space. In this section we consider ice cream spaces, which are precisely the order unit spaces whose state space is isomorphic to the dual unit ball of a Banach space equipped with the $\mathrm{w}^{*}$-topology (see Proposition~\ref{iceaff}). First, let us cover some general information about ordered vector spaces.

\begin{lemma}\label{unital}
Let $\left(E,e\right)$ and $\left(F,f\right)$ be order unit spaces, and let $T\colon E\to F$ be a linear operator. Then:
\item[(i)] If $T$ is unital, then it is positive if and only if $\|T\|=1$.
\item[(ii)] If $T$ is a bijection, then any two of being unital, being an order isomorphism and being an isometry imply the third one.
\item[(iii)] If $T$ is unital, then it is an isometry if and only if it is an order embedding.
\item[(iv)] If $F=\R$, then any two of being unital, positivity and having norm $1$ imply the third one.
\end{lemma}
\begin{proof}
(i): Necessity follows from $T\Ba_{E}=T\left[-e,e\right]\subseteq\left[-f,f\right]=\Ba_{F}$. For sufficiency, let $g\in E_{+}$. We may assume that $\|g\|\le 2$, so that $g\in \left[0_{E},2e\right]=e+\left[-e,e\right]$. Then, $Tg\in Te+T\left[-e,e\right]\subseteq f+\left[-f,f\right]\subseteq F_{+}$.

(ii): The facts that every unital order isomorphism is an isometry, and every unital isometry is an order isomorphism follow from applying (i) to $T$ and $T^{-1}$. If $T$ is an isometric order isomorphism, we have that $T\left[-e,e\right]= \left[-f,f\right]$, and since $e=\max \left[-e,e\right]$, it follows that $Te=\max \left[-f,f\right]=f$.

(iii) follows from (ii) applied to $E$ and $\left(TE,f\right)$.

(iv): In view of (i), it is enough to show that a positive functional of norm $1$ is unital, which is done using a similar argument as in (ii).
\end{proof}

Let $A$ be a convex subset of a vector space $E$. Recall that a \emph{face} of $A$ is a subset $B$ of $A$ with the property that whenever $e,f\in A$ and $t\in\left(0,1\right)$ are such that $te+\left(1-t\right)f\in B$, then $e,f\in B$. Extreme points correspond to faces of cardinality $1$. For some basic information on faces, see \cite[Section~1.6]{at}. Being a face is a transitive relation. The collection of faces of $A$ is closed with respect to arbitrary intersections, and so it is a complete lattice. If $A$ is a cone, so is every nonempty face $B$ of $A$.

\begin{proposition}\label{face}
Let $E$ be an ordered vector space. For a cone $C\subseteq E_{+}$ the following conditions are equivalent:
\item[(i)] $C$ is a face of $E_{+}$;
\item[(ii)] If $e,f\in E_{+}$ are such that $e+f\in C$, then $e,f\in C$;
\item[(iii)] $C$ is a lower set in $E_{+}$, i.e.\ $0_{E}\le e\le f\in C$ implies $e\in C$;
\item[(iv)] $C$ is a full set;
\item[(v)] There is a directed full subspace $H\subseteq E$ such that $C=H_{+}$.
\end{proposition}
\begin{proof}
(i)$\Rightarrow$(ii): Since $C$ is a cone, we have that $\frac{1}{2}e+\frac{1}{2}f\in C$, hence $e,f\in C$.

(ii)$\Rightarrow$(iii): We have that $e,f-e\in E_{+}$, and $e+\left(f-e\right)=f\in C$ implying $e\in C$.

(iii)$\Rightarrow$(i): Assume that $e,f\in E_{+}$ and $t\in\left(0,1\right)$ are such that $te+\left(1-t\right)f\in C$. Then, $0_{E}\le e\le \frac{1}{t}\left( te+\left(1-t\right)f\right)\in C$ implies $e\in C$. Similarly, $f\in C$, and so $C$ is a face of $E_{+}$.

(iii)$\Rightarrow$(v): It is clear that $H:=C-C$ is a directed subspace of $E$. Let us show that $H_{+}=C$. If $f,e\in C$ with $f-e\ge 0_{E}$, then $0_{E}\le f-e\le f\in C$ implies $f-e\in C$, so that $H_{+}\subseteq C$. The converse inclusion is obvious.

If $e,f\in H$ with $e\le f$, then $f-e\in H_{+}=C$, therefore $\left[0_{E},f-e\right]\subseteq C$, and so $\left[e,f\right]=e+\left[0_{E},f-e\right]\subseteq C-C$. We conclude that $H=C-C$ is a full subspace.

% It is left to recall that any directed full subspace of $E$ is an ideal (see \cite[Theorem 4.3.22]{kv}).

(v)$\Rightarrow$(iv)$\Rightarrow$(iii) is straightforward.
\end{proof}

\begin{corollary}
The collection of nonempty faces $C$ of $E_{+}$ is order isomorphic to the collection of directed full linear subspaces $H$ of $E$ via $C\mapsto C-C$ and $H\mapsto H_{+}$.
\end{corollary}

A ray in $E_{+}$ which constitutes a face of $E_+$ is called an \emph{extreme ray}.

\begin{corollary}\label{atmo}
Let $E$ be an ordered vector space and let $e>0_{E}$. Then:
\item[(i)] $e$ is an atom if and only if $\left[0,+\8\right)e$ is an extreme ray.
\item[(ii)] $E$ is atomic if and only if every nonempty and nonzero face of $E_{+}$ contains an extreme ray.
\end{corollary}
\begin{proof}
(i): Necessity: If $e$ is an atom, then $\left[0,1\right]e$ is a lower set in $E_{+}$, hence so is $\left[0,+\8\right)e$, and so the latter is a face of $E_{+}$ by Proposition~\ref{face}.

Sufficiency: If $\left[0,+\8\right)e$ is an extreme ray, it is a face of $E_{+}$. Hence, $\left[0_{E},e\right]\subseteq \left[0,+\8\right)e$, and so $\left[0_{E},e\right]\subseteq \left[0,1\right]e$.

(ii): Necessity: If $C\subseteq E_{+}$ is a face which contains $f>0_{E}$, then there is an atom $e\in\left(0_{E},f\right]\subseteq C$. It follows from (i) that $\left[0,+\8\right)e$ is an extreme ray of $E_{+}$ contained in $C$.

Sufficiency: If $f>0_{E}$, then $\bigcup_{n\in\N}\left[0_{E},nf\right]$ satisfies the conditions of Proposition~\ref{face}, and so is a face. By our assumption and (i), this face contains an atom, say $e$. Let $n\in\N$ be such that $e\in\left[0_{E},nf\right]$. Then, $\frac{1}{n}e$ is an atom contained in $\left[0_{E},f\right]$.
\end{proof}

\begin{proposition}\label{minext}
Assume that $E$ is Archimedean. Then, $C\subseteq E_{+}$ is an extreme ray if and only if it is a minimal nonzero face.
\end{proposition}
\begin{proof}
Since an extreme ray is one-dimensional, and any face of a cone is a cone, necessity is clear. For sufficiency, assume that $C$ is a minimal nonzero face of $E_{+}$.

Suppose that $C$ is not a ray, so that there are two linearly independent vectors $u,v\in C$. Let $F:=\spa\left\{u,v\right\}$, which is an Archimedean ordered vector space. Then, $F_{+}:=E_{+}\cap F$ is closed with respect to the unique Hausdorff linear topology on the two-dimensional space $F$, and it is not a ray since it contains $u,v$. We clearly have $C\cap F\subseteq E_{+}\cap F=F_{+}$. For the converse inclusion, assume that $t,s\in\R$ are such that $tu+sv=:e\in E_{+}$. Then, either $t,s\ge 0$, and then clearly $e\in C$, or say $t>0>s$, so that $C\ni u=\frac{1}{t}e+\frac{-s}{t}v$ implies $e\in C$.

Since $F_{+}$ is a proper cone but not a ray, there exists $0_E < f \in \partial{F_{+}}\subseteq C$. It follows from minimality that $\bigcup_{n\in\N}\left[0_{E},nf\right]=C\supseteq F_{+}$. Hence, $f$ is an order unit of $F$ implying that $f$ belongs to the relative interior of $F_+$ in $F$ (see \cite[Proposition~1.5.11(i)]{kv}). This contradicts $f \in \partial F_+$.
\end{proof}

\begin{remark}
The assumption that $E$ is Archimedean in Proposition~\ref{minext} is not superfluous: consider $E=\R^{2}$ ordered by the cone $C:=\R\times\left(0,+\8\right)\cup\left\{\left(0,0\right)\right\}$, which has only two nonempty faces: $C$ and $\left\{\left(0,0\right)\right\}$.
\qed\end{remark}

\subsection{Basic properties of the ice cream spaces}\label{icebsc}

Throughout the subsection, $F$ is a normed space. Endow $\icecream\left(F\right):=F\oplus_{1} \R$ with the \emph{ice cream order}, determined by the convex cone $\left\{\left(f,r\right): \|f\|\le r\right\}$. Then, $\left(f,r\right)$ is an order unit if and only if $\|f\|< r$. In particular, we will always assume that $1_{\smallicecream\left(F\right)}:=\left(0_{F},1\right)$ is the designated order unit of $\icecream\left(F\right)$. We will show in Proposition~\ref{vdd} that the norm on $\icecream\left(F\right)$ agrees with the order unit norm. It follows that the interior of $\icecream\left(F\right)_{+}$ consists of the elements $\left(f,r\right)$ with $\|f\|< r$, and the boundary $\partial \icecream\left(F\right)_{+}$ of $\icecream\left(F\right)_{+}$ consists of the elements of the form $\left(f,\|f\|\right)$ with $f\in F$.

\begin{remark}
A slightly different (but essentially equivalent) notion of an ice cream cone exists in the literature, see \cite[Section~2.6]{at}. If $F$ is a Hilbert space, then $\icecream\left(F\right)$ is called a \emph{spin factor}. A term sometimes used for $\icecream\left(F\right)$ in the general case is \emph{generalized spin factor}.
\qed\end{remark}

\begin{remark}
Note that $F\times \left[0,+\8\right)\subseteq \icecream\left(F\right)$ can be identified with the collection of all closed balls in $F$, endowed with the structure of a metric ordered cone: the addition and positive multiplication are defined in terms of the algebraic operations with sets in a vector space, the ordering is by inclusion and the distance is the Hausdorff metric. It was observed in \cite{schmidt} that the correspondence $\left(f,r\right)\mapsto\Ba\left(f,r\right)$ is an order isomorphism, an isometry, and an algebraic isomorphism.
\qed\end{remark}

\begin{proposition}[cf.\ \cite{vddb}]\label{vdd}
$\icecream\left(F\right)$ is the free order unit space over $F$, i.e.\ for every order unit space $\left(E,e\right)$ and linear map $T\colon F\to E$ such that $T\Ba_{F}\subseteq \Ba_{E}=\left[-e,e\right]$, there is a unique positive linear map $T_{\smallicecream}\colon\icecream\left(F\right)\to E$ such that $T_{\smallicecream}\left(0_{F},1\right)=e$ and $T_{\smallicecream}\iota=T$, where $\iota$ is the natural inclusion from $F$ into $\icecream\left(F\right)=F\oplus\R$.
\end{proposition}
\begin{proof}
First, let us verify that the norm on $\icecream\left(F\right)$ is the order unit norm. Indeed, the unit ball consists of all $\left(f,r\right)$ such that $\|f\|+\left|r\right|\le 1$, which is equivalent to
\[\left(0_{F},-1\right)\le \left(f,r\right)\le\left(0_{F},1\right).\]

Let $\left(E,e\right)$ be an order unit space, and let $T\colon F\to E$ be linear with $T\Ba_{F}\subseteq \left[-e,e\right]$. Clearly, $T_{\smallicecream}\colon\icecream\left(F\right)\to E$, defined by $T_{\smallicecream}\left(f,r\right):=Tf+re$, is the unique linear map which satisfies $T_{\smallicecream}\left(0_{F},1\right)=e$ and $T_{\smallicecream}\iota=T$. It is left to show that $T_{\smallicecream}$ is positive. Indeed, if $\left(f,r\right)\in \icecream\left(F\right)_{+}$, then $\|f\|\le r$, hence $Tf\in\left[-re,re\right]$, thus $T_{\smallicecream}\left(f,r\right)=Tf+re\ge 0_{E}$.
\end{proof}

\begin{proposition}
Isometric order isomorphisms between ice cream spaces are precisely the maps induced by isometric isomorphisms between the underlying normed spaces.
\end{proposition}
\begin{proof}
Let $G$ be a Banach space. It is clear that any isometric isomorphism from $F$ onto $G$ induces an isometric order isomorphism from $\icecream\left(F\right)$ onto $\icecream\left(G\right)$. For the converse, suppose that $T\colon\icecream\left(F\right)\to \icecream\left(G\right)$ is an isometric order isomorphism. Then, by part~(ii) of Lemma~\ref{unital}, $T$ is unital, i.e.\ $T\left(0_{F},1\right)=\left(0_{G},1\right)$.

Fix $f\in F$. We have that $\left(f,\|f\|\right)\in\partial \icecream\left(F\right)_{+}$, hence $T\left(f,\|f\|\right)\in\partial \icecream\left(G\right)_{+}$. Therefore, $T\left(f,\|f\|\right)=\left(g,\|g\|\right)$, for some $g\in G$ with $\|g\|=\|f\|$. Then,
\[T\left(f,0\right)=T\left(f,\|f\|\right)-T\left(0_{F},\|f\|\right)=\left(g,\|g\|\right)-\left(0_{G},\|f\|\right)=\left(g,0\right).\]
Since $f\in F$ was arbitrary, we conclude that $T$ isometrically maps $F$ into $G$. Applying the same argument to $T^{-1}$ shows that $S:=\left.T\right|_{F}$ is an isometric isomorphism onto $G$, and $T=S\oplus \Id_{\R}$.
\end{proof}

We now study states on $\icecream\left(F\right)$. Recall that $F$ is \emph{strictly convex} (or \emph{rotund}) if every point in the unit sphere $\So_{F}$ is extreme. We will say that two ordered normed spaces are \emph{order isometric} if there is an isometric order isomorphism between them.

\begin{proposition}\label{icecream} Let $F$ be a normed space.
\item[(i)] $\icecream\left(F\right)^{*}=F^{*}\oplus_{\8}\R$ order isometrically with respect to the ice cream order on the latter.
\item[(ii)] The states on $\icecream\left(F\right)$ are the functionals of the form $\left(\nu,1\right)$, where $\nu\in \Ba_{F^{*}}$. The pure states are the functionals of the form $\left(\nu,1\right)$, where $\nu\in \ext\Ba_{F^{*}}$.
\item[(iii)] $\icecream\left(F\right)$ is order isometric to a supremum-dense subspace of $\Co\left(\overline{\ext}^{\mathrm{w}^{*}}\Ba_{F^{*}},\mathrm{w}^{*}\right)$.
\item[(iv)] If $F$ is norm complete, $\dim F>0$ and $F^{*}$ is strictly convex, then $F$ is non-reflexive if and only if there exists a strictly positive pure state on $\icecream\left(F\right)$.
\end{proposition}
\begin{proof}
(i): The only non-straightforward part of the statement is showing that the two spaces are order isomorphic. We have that $\left<\left(f^{*},s\right),\left(f,r\right)\right>=\left<f^{*},f\right>+sr$. If $\|f^{*}\|\le s$ and $\|f\|\le r$, then $sr\ge -\left<f^{*},f\right>$, and so $\left<\left(f^{*},s\right),\left(f,r\right)\right>\ge 0$. If $\|f^{*}\|> s$, there is $f\in F$ such that $\|f\|=1$ and $\left<f^{*},f\right><-s$, and then $\left(f,1\right)\ge 0$ but $\left<\left(f^{*},s\right),\left(f,1\right)\right>=\left<f^{*},f\right>+s<0$.

(ii): Let $\left(\nu,r\right)\in \icecream\left(F\right)^{*}$. This functional is a state if and only if $\left(\nu,r\right)\in \icecream\left(F\right)^{*}_{+}$ and $\left(\nu,r\right)\left(0_{F},1\right)=1$, i.e.\ $r=1$; by (i), this is equivalent to $\|\nu\|\le r=1$. The second claim follows from the first one.

(iii) follows from (ii) and part~(ii) of Theorem~\ref{T:rep of o.u.s.}.

(iv): First, since $F^{*}$ is strictly convex, we have $\ext\Ba_{F^{*}}=\So_{F^{*}}$, and so the pure states on $\icecream\left(F\right)$ are precisely the functionals $\left(\nu,1\right)$, where $\nu\in \So_{F^{*}}$.

Let us show that $\left(\nu,1\right)$, where $\nu\in \So_{F^{*}}$, is strictly positive if and only if $\nu$ does not attain its norm. Strict positivity means that $\left(\nu,1\right)\left(f,r\right)=\nu\left(f\right)+r>0$, whenever $f\in F$ and $r>0$ are such that $\|f\|\le r$. Attainment of the norm for $\nu$ means existence of $f\in\So_{F}$ such that $\nu\left(f\right)=-1$, equivalently $\left(\nu,1\right)\left(f,1\right)=0$.

It follows that a strictly positive pure state on $\icecream\left(F\right)$ exists if and only if some $\nu\in \So_{F^{*}}$ does not attain its norm. By James' Theorem (see \cite[Corollary~3.131]{fhhmz}), this last condition is equivalent to $F$ not being reflexive.
\end{proof}

The class of spaces in part~(iv) of Proposition~\ref{icecream} is quite broad.

\begin{example}\label{lur}
If $F$ is a Banach space with a separable dual, by \cite[Theorem~8.7]{fhhmz}, it has an equivalent norm such that $F^{*}$ is \emph{locally uniformly rotund (LUR)}, i.e.\ whenever $\left(f_{n}\right)_{n\in\N}\subseteq F$ and $f\in F$ are such that $2\|f_{n}\|^{2}+2\|f\|^{2}-\|f_{n}+f\|^{2}\to 0$, it follows that $f_{n}\to f$. Note that if $F^{*}$ is LUR, then it is strictly convex (this can be deduced from \cite[Fact~7.7]{fhhmz}).
\qed\end{example}

\begin{remark}
In part~(iv) of Proposition~\ref{icecream} it was established that pure states on ice cream spaces can be strictly positive. Note that this situation cannot occur even for limits of pure states on pervasive order unit spaces. Indeed, assume that $\left(E,e\right)$ is such a space, with $\dim E>1$, and $\nu\in \overline{\ext}^{\mathrm{w}^{*}}K_{E}$. By part~(ii) of Theorem~\ref{T:rep of o.u.s.}, we have that $E$ embeds as a supremum-dense subspace of $\Co\bigl(\overline{\ext}^{\mathrm{w}^{*}}K_{E}\bigr)$. According to Proposition~\ref{arch1}, $E$ is in fact order dense in $\Co\bigl(\overline{\ext}^{\mathrm{w}^{*}}K_{E}\bigr)$. Take $\0<f\in \Co\bigl(\overline{\ext}^{\mathrm{w}^{*}}K_{E}\bigr)$ which vanishes at $\nu$. Then, there is $0_{E}<e\in E$ such that $e\le f$, and so $\nu\left(e\right)=0$. Thus, $\nu$ is not strictly positive.
\qed\end{remark}

The following representation of ice cream spaces as spaces of affine functions will be very useful for studying order convergence in them.

\begin{proposition}\label{iceaff}
Let $F$ be a normed space.
\item[(i)] $\icecream\left(F\right)$ is order isometric to a norm-dense subspace of $\Aff\left(\Ba_{F^{*}},\mathrm{w}^{*}\right)$ endowed with the pointwise order and supremum norm.
\item[(ii)] $\icecream\left(F^{*}\right)$ is order isometric to $\Aff\left(\Ba_{F},\mathrm{w}\right)=\Aff\left(\Ba_{F},\|\cdot\|\right)$, which is also equal to the space of all bounded affine functions on $\Ba_{F}$.
\item[(iii)] The states on $\Aff\left(\Ba_{F^{*}},\mathrm{w}^{*}\right)$ and $\Aff\left(\Ba_{F},\|\cdot\|\right)$ are the point evaluations at points of $\Ba_{F^{*}}$ and $\Ba_{F^{**}}$, respectively.
\end{proposition}
\begin{proof}
(i) follows from part~(ii) of Proposition~\ref{icecream}, part~(i) of Theorem~\ref{T:rep of o.u.s.}, and the fact that the $\mathrm{w}^{*}$-topology on $\Ba_{F^{*}}$ as a subset of $F^{*}$ agrees with the $\mathrm{w}^{*}$-topology on $\Ba_{F^{*}}$ as the state space of $\icecream\left(F\right)$.

(ii): By (i) we have that $\icecream\left(F^{*}\right)\simeq \Aff\left(\Ba_{F^{**}},\mathrm{w}^{*}\right)$. Since by Goldstine's theorem (see \cite[Theorem~3.96]{fhhmz}) we have that $\Ba_{F}$ is dense in $\left(\Ba_{F^{**}},\mathrm{w}^{*}\right)$, we can identify $\Aff\left(\Ba_{F^{**}},\mathrm{w}^{*}\right)$ with the set of all $e\in \Aff\left(\Ba_{F},\mathrm{w}\right)$ that can be continuously extended to $\Ba_{F^{**}}$. It is also clear that $\Aff\left(\Ba_{F},\mathrm{w}\right)\subseteq \Aff\left(\Ba_{F},\|\cdot\|\right)$, which in turn is contained in the space of all bounded affine functions on $\Ba_{F}$. Assume that $u$ is a bounded affine function on $\Ba_{F}$ and let $v:=u-u\left(0_{F}\right)\1$, which is also a bounded affine function on $\Ba_{F}$. This function satisfies $v\left(0_{F}\right)=0$, and for every $f\in \Ba_{F}$ and $t\in\left[0,1\right]$ we have that
\[v\left(tf\right)=tv\left(f\right)+\left(1-t\right)v\left(0_{F}\right)=tv\left(f\right).\] Moreover, for $f,g\in \Ba_{F}$ we have $v\left(\frac{1}{2}f+\frac{1}{2}g\right)=\frac{1}{2}v\left(f\right)+\frac{1}{2}v\left(g\right)$. It is easy to see that extending $v$ by positive homogeneity yields a linear functional on $F$. This functional is bounded on $\Ba_{F}$, so is an element of $F^{*}$. Hence, we may view $v$ as a $\mathrm{w}^{*}$-continuous linear functional on $F^{**}$, and so $u=v+u\left(0_{F}\right)\1\in\Aff\left(\Ba_{F^{**}},\mathrm{w}^{*}\right)$.

\begin{comment}
Then, arguing as before, $v:=u-u\left(0_{F}\right)\1$ can be extended to a norm-continuous linear functional on $F$. Such a functional is weakly continuous, thus justifying $u=v+u\left(0_{F}\right)\1\in \Aff\left(\Ba_{F},\mathrm{w}\right)$.

The second claim is proven similarly; let us only comment on the equality $\Aff\left(\Ba_{F},\mathrm{w}\right)=\Aff\left(\Ba_{F},\|\cdot\|\right)$. The direct inclusion is obvious. For the converse, assume that $u\in\Aff\left(\Ba_{F},\|\cdot\|\right)$. Then, arguing as before, $v:=u-u\left(0_{F}\right)\1$ can be extended to a norm-continuous linear functional on $F$. Such a functional is weakly continuous, thus justifying $u=v+u\left(0_{F}\right)\1\in \Aff\left(\Ba_{F},\mathrm{w}\right)$.\medskip
\end{comment}

(iii) follows from (i), (ii) and part~(ii) of Proposition~\ref{icecream}.
\end{proof}

Note that the specific isomorphism from $\icecream\left(F^{*}\right)$ onto $\Aff\left(\Ba_{F},\|\cdot\|\right)$ is $\left(\nu,r\right)\mapsto \nu+r\1$ (as functions on $\Ba_{F}$).

We now turn to studying convergences on ice cream spaces, starting with the behavior of monotone nets.

\begin{proposition}\label{docauch}
An increasing net in $\icecream\left(F\right)$ which is bounded from above is norm Cauchy.
\end{proposition}
\begin{proof}
Let $\left(\left(f_{\alpha}, r_{\alpha}\right)\right)_{\alpha\in A}$ be an increasing bounded from above net in $\icecream\left(F\right)$. By shifting, normalizing and passing to a tail, we may assume that $\left(\left(f_{\alpha}, r_{\alpha}\right)\right)_{\alpha\in A}\subseteq \left[0_{\smallicecream\left(F\right)},1_{\smallicecream\left(F\right)}\right]$. In particular, $\left(0_{F},0\right)\le \left(f_{\alpha}, r_{\alpha}\right)\le\left(0_{F},1\right)$ implies $\|f_{\alpha}\|\le r_{\alpha}$, as well as $\|f_{\alpha}\|\le 1-r_{\alpha}$, hence $r_{\alpha}\in\left[0,1\right]$, for every $\alpha$. If $\alpha\le\beta$, we have that $\|f_{\beta}-f_{\alpha}\|\le r_{\beta}-r_{\alpha}$, and so $r_{\alpha}\le r_{\beta}$. It follows that $\left(r_{\alpha}\right)_{\alpha\in A}\subseteq\left[0,1\right]$ is an increasing net, therefore it converges in $\left[0,1\right]$. The fact that $\|f_{\beta}-f_{\alpha}\|\le r_{\beta}-r_{\alpha}$, for $\alpha\le\beta$, implies that $\left(f_{\alpha}\right)_{\alpha\in A}$ is a Cauchy net.
\end{proof}

In the following result part~(i) is a generalization of \cite[Example~6.2]{hauser}.

\begin{proposition}\label{icedo}
Let $F$ be a normed space.
\item[(i)] The directed order, relative uniform and norm convergence structures on $\icecream\left(F\right)$ coincide.
\item[(ii)] Monotone completeness, directed order completeness, relative uniform completeness and norm completeness of $\icecream\left(F\right)$ are each equivalent to norm completeness of $F$.
\end{proposition}
\begin{proof}
(i): Since $\icecream\left(F\right)$ is an order unit space, norm convergence agrees with relative uniform convergence in $\icecream\left(F\right)$. The latter convergence implies directed order convergence, and so it is left to show that directed order convergence implies norm convergence.

Assume that $e_{\alpha}\todo 0_{\smallicecream\left(F\right)}$. Let $Q\downarrow 0_{\smallicecream\left(F\right)}$ be as in part~(i) of Proposition~\ref{todovs}. According to Proposition~\ref{docauch}, the set $Q$, viewed as a net indexed by $-Q$, is norm-Cauchy. Fix $\varepsilon>0$ and let $q_{0}\in Q$ be such that $\|q_{0}-q\|\le\varepsilon$, for every $Q\ni q\le q_{0}$. It follows that $q\ge q_{0}-\varepsilon 1_{\smallicecream\left(F\right)}$, and since $0_{\smallicecream\left(F\right)}=\bigwedge\left\{q\in Q:q\le q_{0}\right\}$, we obtain $0_{\smallicecream\left(F\right)}\ge q_{0}-\varepsilon 1_{\smallicecream\left(F\right)}$, hence $q_{0}\le \varepsilon 1_{\smallicecream\left(F\right)}$. There is $\alpha_{0}\in A$ such that $-\varepsilon 1_{\smallicecream\left(F\right)}\le -q_{0}\le e_{\alpha}\le q_{0}\le \varepsilon 1_{\smallicecream\left(F\right)}$, for every $\alpha\ge\alpha_{0}$. We conclude that $\|e_{\alpha}\|\to 0$.

(ii): It follows immediately from (i) that completenesses with respect to directed order, relative uniform and norm convergences are equivalent. Monotone completeness implies relative uniform completeness by part~(i) of Corollary~\ref{monc}. Since $\icecream\left(F\right)\simeq F\oplus_{1}\R$ as normed spaces, it follows that $\icecream\left(F\right)$ is norm complete if and only if $F$ is. If $\icecream\left(F\right)$ is norm complete, it is monotone complete. Indeed, by Proposition~\ref{docauch}, every increasing net which is bounded from above is norm Cauchy, hence norm convergent, thus directed order convergent, and so the limit is the supremum of the net.%Let us show that if $\icecream\left(F\right)$ is monotone complete, then $F$ is norm complete. Let $\left(f_{n}\right)_{n\in\N}\subseteq F$ be such that $\sum\limits_{n\in\N}\|f_{n}\|=1$. For $n\in\N$, let $e_{n}:=\left(f_{1}+\dots+f_{n},\|f_{1}\|+\dots+\|f_{n}\|\right)$. It follows that $e_{n}\in \left[0_{\smallicecream\left(F\right)}, 2\cdot 1_{\smallicecream\left(F\right)}\right]$, and also $e_{n}\le e_{n+1}$ for every $n\in\N$. Since $\icecream\left(F\right)$ is monotone complete, $\left(e_{n}\right)_{n\in\N}$ increases to some element $\left(f,r\right)$, which by (i) is also the norm limit. But then $\sum\limits_{n\in\N} f_n=f$, and so every absolutely convergent series in $F$ is convergent, justifying its completeness.
\end{proof}

\subsection{Order convergence in the ice cream spaces}\label{iceoc}

In this section, we seek to describe order convergence and the order topology on an ice cream space $\icecream\left(F\right)$. For a reflexive separable $F$, we will show that order convergence in $\icecream\left(F\right)$ coincides with so-called continuous convergence in $\icecream\left(F\right) \subseteq \Co(\Ba_{F^*}, \mathrm{w}^{*})$.

Let $X$ be a topological space. We say that a net $\left(f_{\alpha}\right)_{\alpha\in A}\subseteq\Co\left(X\right)$ \emph{continuously converges} to $f\in\Co\left(X\right)$ if for every $x\in X$ and net $\left(x_{\beta}\right)_{\beta\in B}\subseteq X$ such that $x_{\beta}\to x$, we have $f_{\alpha}\left(x_{\beta}\right)\xrightarrow[\alpha\in A,~\beta\in B]{}f\left(x\right)$.

Throughout this section, $F$ is a Banach space. By \cite[Proposition~8.4]{btw}, the restriction to $F^{*}$ of continuous convergence in $\Co\left(F\right)$ agrees with bounded $\mathrm{w}^{*}$-convergence, defined as follows. A net in $F^{*}$ \emph{bounded $\mathrm{w}^{*}$}-converges to $f^{*}\in F^{*}$ if it converges to $f^{*}$ in the $\mathrm{w}^{*}$-topology and is eventually norm bounded. The \emph{bounded $\mathrm{w}^{*}$}-topology is the strongest topology on $F^{*}$ which agrees with the $\mathrm{w}^{*}$-topology on norm-bounded subsets of $F^{*}$. This topology, denoted by $\mathrm{bw}^{*}$, is the topological modification of the bounded $\mathrm{w}^{*}$-convergence, i.e.\ the closed sets of the $\mathrm{bw}^{*}$-topology are precisely the closed sets in bounded $\mathrm{w}^{*}$-convergence (see \cite[Example~8.17]{bctv}). We will need the following properties of this topology.

\begin{proposition}\label{bw}
\item[(i)] Viewing $F^{*}\subseteq \Co\left(F\right)$ or $F^{*}\subseteq \Co\left(\Ba_{F}\right)$, the restriction of the compact-open topology agrees with the $\mathrm{bw}^{*}$-topology.
\item[(ii)] $\left(F^{*},\mathrm{bw}^{*}\right)^{*}=F$.
\end{proposition}
\begin{proof}
For (ii), as well as the case of $\Co\left(F\right)$ in (i), see \cite[Lemmas~8.45 and~8.46]{at}. Also, it is easy to see that the restriction map $R\colon\Co\left(F\right)\to\Co\left(\Ba_{F}\right)$ defined by $\varphi\mapsto \varphi|_{\Ba_{F}}$ is a continuous map with respect to compact-open topologies. Clearly, $R|_{F^{*}}$ is injective. To show that $R|_{F^{*}}$ is a homeomorphism onto its image, let $\left(f^{*}_{\alpha}\right)_{\alpha\in A}$ be a net in $F^{*}$ such that $\bigl(f^{*}_{\alpha}|_{\Ba_{F}}\bigr)_{\alpha\in A}$ is $\mathrm{bw}^{*}$-null and let $K\subseteq F$ be compact. There is $r>0$ such that $rK\subseteq \Ba_{F}$, and so $f^{*}_{\alpha}|_{\Ba_{F}}\to\0$ uniformly on $rK$. Due to linearity, we have that $f^{*}_{\alpha}\to\0$ uniformly on $K$, and since $K$ was arbitrary, we conclude that $f^{*}_{\alpha}\to\0$.
\end{proof}

To study order convergence in ice cream spaces, we will use the representations from Proposition~\ref{iceaff}. In particular, by e.g.\ continuous convergence in $\Aff\left(\Ba_{F},\|\cdot\|\right)$ we mean the restriction of continuous convergence in $\Co\left(\Ba_{F}\right)$ to $\Aff\left(\Ba_{F},\|\cdot\|\right)$. On the other hand, we will also use the fact that $\Aff\left(\Ba_{F},\|\cdot\|\right)$ is (non-isometrically) isomorphic to $\icecream\left(F\right)^{*}$, and so we will be referring to the $\mathrm{w}^{*}$- and bounded $\mathrm{w}^{*}$-topologies on $\Aff\left(\Ba_{F},\|\cdot\|\right)$ with respect to $\icecream\left(F\right)$.

We can define bounded weak convergence, bounded pointwise convergences, etc. in a way similar to how bounded $\mathrm{w}^{*}$-convergence was defined before Proposition~\ref{bw}.

\begin{proposition}\label{iceo}
\item[(i)] The weak and pointwise topologies coincide on $\Aff\left(\Ba_{F^{*}},\mathrm{w}^{*}\right)$. The $\mathrm{w}^{*}$- and pointwise topologies coincide on $\Aff\left(\Ba_{F},\|\cdot\|\right)$.
\item[(ii)] Bounded pointwise convergence, bounded $\mathrm{w}^{*}$-convergence, bounded compact-open convergence and continuous convergence agree on $\Aff\left(\Ba_{F},\|\cdot\|\right)$.
\item[(iii)] The bounded $\mathrm{w}^{*}$-topology agrees with the compact-open topology on $\Aff\left(\Ba_{F},\|\cdot\|\right)$.
\end{proposition}
\begin{proof}
The first claim in (i) follows from part~(ii) of Proposition~\ref{icecream}, and so in this proof we will be only working with $E:=\Aff\left(\Ba_{F},\|\cdot\|\right)$.

(i): Recall that $\left(\nu,r\right)\left(f,s\right)=\nu\left(f\right)+rs$, where $\left(\nu,r\right)\in \icecream\left(F^{*}\right)$ and $\left(f,s\right)\in \icecream\left(F\right)$. Hence, if we identify $\left(\nu,r\right)$ with $\nu+r\1$, then $\left(f,1\right)$ acts on it as the evaluation at $f$, for $f\in\Ba_{F}$. This shows that the $\mathrm{w}^{*}$-topology is stronger than the pointwise topology on $E$. On the other hand, since $\icecream\left(F\right)=\spa\left\{\left(f,1\right): f\in\Ba_{F}\right\}$, we get the converse relation.

%Clearly, $\delta_{f^{*}}\in E^{*}$, for every $f^{*}\in F^{*}$, and so the weak topology on $E$ is stronger than the pointwise topology. To show the converse, assume that $\left(e_{\alpha}\right)_{\alpha\in A}\subseteq E$ is pointwise null. Note that $e_{\alpha}=f_{\alpha}+r_{\alpha}\1$, where $f\in F$ and $r_{\alpha}\in\R$, for every $\alpha$. We have that $r_{\alpha}=e_{\alpha}\left(0_{F}\right)\to 0$. Let $\nu\in E^{*}$. Since $E\simeq \icecream\left(F\right)$, it follows that $E^{*}\simeq \icecream\left(F^{*}\right)$, and so $\nu=\left(f^{*},s\right)$, where $f^{*}\in F^{*}$ and $s\in\R$. We have $\nu\left(e_{\alpha}\right)=\left<f^{*},f_{\alpha}\right>+sr_{\alpha}=e_{\alpha}\left(f^{*}\right)-r_{\alpha}+sr_{\alpha}\to 0$. As $\nu$ was arbitrary, we conclude that $e_{\alpha}\tow\0$.\medskip

(ii): It follows from (i) that bounded $\mathrm{w}^{*}$- and bounded pointwise convergences coincide on $E$. Continuous convergence is always stronger than compact-open convergence (see \cite[Proposition~8.2]{btw}), which in turn is stronger than pointwise convergence.

Let $\left(e_{\alpha}\right)_{\alpha\in A}\subseteq E$ be a net, and for every $\alpha\in A$ write $e_{\alpha}=f^{*}_{\alpha}+e_{\alpha}\left(0_{F}\right)\1$, where $f^{*}_{\alpha}\in F^{*}$. Assume that $\left(e_{\alpha}\right)_{\alpha\in A}$ converges to $\0$ with respect to continuous convergence. Then, $e_{\alpha}\left(0_{F}\right)\to 0$, and so $\left(f^{*}_{\alpha}\right)_{\alpha\in A}$ also converges to $\0$ with respect to continuous convergence. It follows from the discussion at the start of Subsection~\ref{iceoc} that $\left(f^{*}_{\alpha}\right)_{\alpha\in A}$ is eventually norm-bounded. We conclude that $\left(e_{\alpha}\right)_{\alpha\in A}$ is eventually norm-bounded, hence bounded compact-open-null. Similarly, if $\left(e_{\alpha}\right)_{\alpha\in A}$ is bounded pointwise-null, then $\left(f^{*}_{\alpha}\right)_{\alpha\in A}$ is bounded $\mathrm{w}^{*}$-null, hence null with respect to continuous convergence, which also implies that $\left(e_{\alpha}\right)_{\alpha\in A}$ converges to $\0$ with respect to continuous convergence.

(iii): Recall that the $\mathrm{w}^{*}$-topology on $E$ is with respect to $\icecream\left(F\right)$, and so, by renorming the latter, we may represent $E\simeq \left(F\oplus_{\8} \R\right)^{*}\subseteq\Co\left(\Ba_{F}\times \left[-1,1\right]\right)$. Recall that by part~(i) of Proposition~\ref{bw}, the bounded $\mathrm{w}^{*}$-topology on $E$ coincides with the compact-open topology viewing $E\subseteq\Co\left(\Ba_{F}\times \left[-1,1\right]\right)$. Hence, we need to show that a net $\left(e_{\alpha}\right)_{\alpha\in A}\subseteq E$ is compact-open-null on $\Ba_{F}$ if and only if it is compact-open-null on $\Ba_{F}\times \left[-1,1\right]$. In fact, arguing as in (ii) we have that $e_{\alpha}=f^{*}_{\alpha}+e_{\alpha}\left(0_{F}\right)\1$, where $e_{\alpha}\left(0_{F}\right)\to 0$, and $\left(f^{*}_{\alpha}\right)_{\alpha\in A}$ is compact-open-null (in either sense). Note that the $f_{\alpha}^{*}$'s when represented as functions on $\Ba_{F}\times \left[-1,1\right]$ only depend on the first variable. Hence, if $\left(f^{*}_{\alpha}\right)_{\alpha\in A}$ is compact-open-null on $\Ba_{F}\times \left[-1,1\right]$, then it is clearly compact-open-null on $\Ba_{F}$. Conversely, suppose $\left(f^{*}_{\alpha}\right)_{\alpha\in A}$ is compact-open-null on $\Ba_{F}$. If $K\subseteq \Ba_{F}\times \left[-1,1\right]$ is compact, there is a compact $L\subseteq F$ such that $K\subseteq L\times\left[-1,1\right]$. Using this, one can show that $\left(f^{*}_{\alpha}\right)_{\alpha\in A}$ is compact-open-null on $\Ba_{F}\times \left[-1,1\right]$.
\end{proof}

We now move on to the actual description of order convergence in $\icecream\left(F\right)$. In fact, some information is available for general order unit spaces. We need to recall that if $K$ is a compact Hausdorff space, then every order-null net in $\Co\left(K\right)$ converges pointwise to $\0$ on a co-meager set in $K$, while the converse holds for order bounded sequences (see \cite[Theorem~3.7 and the discussion after Corollary~3.3]{bt}). Note that in order unit spaces, the notions of order boundedness and norm boundedness coincide. To simplify the notation, until the end of this section all closures will be taken in the $\mathrm{w}^{*}$-topology, unless indicated otherwise.

\begin{proposition}\label{ouoc}
Let $\left(E,e\right)$ be an order unit space, embedded into $\Co\left(\mathop{\overline{\ext}}K_{E},\mathrm{w}^{*}\right)$ via $f\mapsto \hat{f}$ as described prior to Theorem~\ref{T:rep of o.u.s.}. Then:
\item[(i)] Order convergence in $E$ is the restriction of order convergence in $\Co\left(\mathop{\overline{\ext}}K_{E},\mathrm{w}^{*}\right)$.
\item[(ii)] Weak convergence implies order convergence for sequences in $E$.
\item[(iii)] If additionally $E^{*}$ is separable, then bounded weak convergence is stronger than order convergence in $E$.
\end{proposition}
\begin{proof}
Denote $K:=\left(\mathop{\overline{\ext}}K_{E},\mathrm{w}^{*}\right)$.

(i) follows from the fact that $E$ is supremum-dense in $\Co\left(K\right)$ and Proposition~\ref{ocp}.

(ii): Assume that $e_{n}\tow 0_{E}$. First, it follows from the Banach--Steinhaus theorem (see \cite[Theorem~3.85]{fhhmz}) that $\left\{e_{n}\right\}_{n\in\N}$ is norm bounded, and therefore also order bounded. Next, by definition we see that $\left(e_{n}\right)_{n\in\N}$ converges to $\0$ pointwise in $\Co\left(K\right)$. According to the comment before the proposition, this implies $e_{n}\too \0$ in $\Co\left(K\right)$, and so by (i) we conclude that $e_{n}\too 0_{E}$ in $E$.

(iii): First, $E$ is separable, and so $\left(\Ba_{E^{*}},\mathrm{w}^{*}\right)$ is metrizable; it also follows that $\left(\Ba_{E},\mathrm{w}\right)$ is metrizable (for these results see \cite[Propositions~2.8, 3.103 and~3.105]{fhhmz}). In particular, $K$ is metrizable and compact, and so order convergence in $\Co\left(K\right)$ is strongly first countable (see \cite[discussion after Proposition~4.21]{bbv}). Furthermore, $\left(\Ba_{E},\mathrm{w}\right)$ is metrizable, thus also strongly first countable. By (ii), the embedding of $\left(\Ba_{E},\mathrm{w}\right)$ into $\left(\Co\left(K\right),\mathrm{o}\right)$ is sequentially continuous, and so it is in fact continuous by \cite[Theorem~5.2]{bbv}. Thus, using (i) again we conclude that bounded weak convergence is stronger than order convergence in $E$.
\end{proof}

Before specializing to ice cream spaces, let us present an auxiliary result whose proof is inspired by \cite[Section~2]{gmm}.

\begin{lemma}\label{densenet}
Assume that the weak and $\mathrm{w}^{*}$-topologies agree on $\So_{F^{*}}$. Furthermore, let $E:= \Aff\left(\Ba_{F^{*}},\mathrm{w}^{*}\right)$.
If a net $\left(e_{\alpha}\right)_{\alpha\in A}\subseteq \Ba_{E}$ converges to $\0$ pointwise on a $\mathrm{w}^{*}$-dense subset of $\mathop{\overline{\ext}}\Ba_{F^{*}}$, then it converges weakly in $E$.
\end{lemma}
\begin{proof}
Let $C:=\left\{\nu\in \Ba_{F^{*}}: e_{\alpha}\left(\nu\right)\to 0\right\}$. By part~(i) of Proposition~\ref{iceo}, it is enough to show that $C=\Ba_{F^{*}}$. Since every $e_{\alpha}$ is affine, $C$ is convex. Let us show that $C$ is norm-closed. Assume that $\mu$ is in the norm closure of $C$, and let $\varepsilon>0$. There is $\nu\in C$ such that $\|\mu-\nu\|\le\varepsilon$. Find $\alpha_{0}$ such that $\left|e_{\alpha}\left(\nu\right)\right|\le\varepsilon$, for all $\alpha\ge \alpha_{0}$. For $\alpha\ge \alpha_{0}$ we have that $e_{\alpha}=f_{\alpha}+r_{\alpha}\1$, where $r_{\alpha}\in\R$ and $f_{\alpha}\in \Ba_{F}$, hence
\[\left|e_{\alpha}\left(\mu\right)\right|=\left|f_{\alpha}\left(\mu\right)+r_{\alpha}\right|\le \left|f_{\alpha}\left(\nu\right)+r_{\alpha}\right|+\left|f_{\alpha}\left(\mu-\nu\right)\right|=\left|e_{\alpha}\left(\nu\right)\right|+\left|f_{\alpha}\left(\mu-\nu\right)\right|
\le2\varepsilon.\]
Thus, $\mu\in C$. Being a norm-closed convex set, $C$ is weakly closed. By assumption, $C\cap \mathop{\overline{\ext}}\Ba_{F^{*}}$ is $\mathrm{w}^{*}$-dense in $\mathop{\overline{\ext}}\Ba_{F^{*}}$. Let us show that $C$ is $\mathrm{w}^{*}$-dense in $\Ba_{F^{*}}$. Since $C$ is convex, so is $D:=\overline{C}$. We have that $D\supseteq \overline{C\cap \mathop{\overline{\ext}}\Ba_{F^{*}}}\supseteq \ext\Ba_{F^{*}}$, hence by the Krein--Milman theorem, we conclude that $D\supseteq \overline{\conv}\left( \ext\Ba_{F^{*}}\right)=\Ba_{F^{*}}$.

Let $\mu\in\So_{F^{*}}$. As $C$ is $\mathrm{w}^{*}$-dense, there is $\left(\nu_{\beta}\right)_{\beta\in B}\subseteq C$ such that $\nu_{\beta}\toww\mu$. Then, $1=\|\mu\|\le\liminf_{\beta\in B}\|\nu_{\beta}\|\le\limsup_{\beta\in B}\|\nu_{\beta}\|\le 1$ implies $\|\nu_{\beta}\|\to 1$, and so without loss of generality we assume $\nu_{\beta}\ne 0_{F^{*}}$, for every $\beta$. It follows that $\So_{F^{*}}\ni\mu_{\beta}:=\|\nu_{\beta}\|^{-1}\nu_{\beta}\to\mu$ in the $\mathrm{w}^{*}$-topology, and so by our assumption we get $\mu_{\beta}\tow\mu$. Therefore, $\nu_{\beta}=\|\nu_{\beta}\|\mu_{\beta}\tow\mu$, and since $\nu_{\beta}\in C$ and $C$ is weakly closed, we conclude that $\mu\in C$. Thus, $\So_{F^{*}}\subseteq C$, hence $\Ba_{F^{*}}=\conv\So_{F^{*}}\subseteq C$, as required.
\end{proof}

\begin{proposition}\label{iceor}
\item[(i)] Weak convergence implies order convergence for sequences in $\icecream\left(F\right)$.
\item[(ii)] If, additionally, $F^{*}$ is separable, then bounded weak convergence is stronger than order convergence in $\icecream\left(F\right)$.
\item[(iii)] If, additionally, weak and $\mathrm{w}^{*}$-convergences agree on $\So_{F^{*}}$, then order convergence is stronger than bounded weak convergence in $\icecream\left(F\right)$.
\end{proposition}
\begin{proof}
(i) follows immediately from part~(ii) of Proposition~\ref{ouoc}. (ii) follows from part~(iii) of Proposition~\ref{ouoc} and the fact that $\icecream\left(F\right)^{*}\simeq F^{*}\oplus_{\8} \R$.

(iii): Assume that $\left(e_{\alpha}\right)_{\alpha\in A}\subseteq E:=\icecream\left(F\right)$ is order-null. By part~(i) of Proposition~\ref{ouoc} and the discussion preceding it, there are $\alpha_{0}$ and $r>0$ such that the net $\left(re_{\alpha}\right)_{\alpha\ge\alpha_{0}}$ is contained in $\Ba_{E}$ and converges pointwise on a co-meager set in $\left(\mathop{\overline{\ext}}\Ba_{F^{*}},\mathrm{w}^{*}\right)$. The latter is a compact Hausdorff space, and so any co-meager subset of this space is $\mathrm{w}^{*}$-dense. By Lemma~\ref{densenet}, the net $\left(re_{\alpha}\right)_{\alpha\ge\alpha_{0}}$ is weakly null. Thus, $\left(e_{\alpha}\right)_{\alpha\in A}$ is bounded weakly null.
\end{proof}

We finally get to the culmination of this section.

\begin{theorem}\label{iceord}
Assume that $F^{*}$ is separable and that the weak and $\mathrm{w}^{*}$-topologies agree on $\So_{F^{*}}$. Let $E:=\Aff\left(\Ba_{F^{*}},\mathrm{w}^{*}\right)\simeq \icecream\left(F\right)$. Then:
\item[(i)] Order, bounded weak and bounded pointwise convergences agree on $E$.
\item[(ii)] Order intervals in $E$ are order metrizable and order pre-compact.
\item[(iii)] $\dim F=\8$ if and only if directed order convergence is strictly stronger than order convergence in $E$.
\item[(iv)] Order intervals in $E$ are order compact if and only if $E$ is order complete if and only if $F$ is reflexive.
\end{theorem}
\begin{proof}
(i) follows from combining parts~(ii) and (iii) of Proposition~\ref{iceor} and part~(i) of Proposition~\ref{iceo}. (ii) follows from the facts that balls in $E$ are metrizable and pre-compact in the weak topology, that order and weak convergences agree on balls, and that every order interval in $E$ is norm bounded. (iii) follows from (i) and part~(i) of Proposition~\ref{icedo}.

(iv): If $F$ is reflexive, then so is $E$, and so order intervals in $E$ are weakly compact, which by (i) is equivalent to order compact.

If order intervals are order compact, then they are order complete. Since every order Cauchy net has a tail contained in an order interval, it follows that $E$ is order complete.

We now assume that $F$ is not reflexive and show that $E$ is not order complete. As $E$ is also not reflexive, we have that $\Ba_{E^{**}}\setminus\Ba_{E}\ne\varnothing$. By Goldstine's theorem, there is a net in $\Ba_{E}$ which $\mathrm{w}^{*}$-converges to $g\in \Ba_{E^{**}}\setminus\Ba_{E}$, and so it is a weakly Cauchy net, which does not have a weak limit. By (i), this net is also order Cauchy but not order convergent, thus showing that $E$ is not order complete.
\end{proof}

The \emph{order topology} is the topological modification of order convergence, i.e.\ the closed sets are precisely the sets which are closed with respect to order convergence.

\begin{corollary}\label{iceord2}
Let $F$ be a separable reflexive Banach space. Then, on $\Aff\left(\Ba_{F^{*}},\|\cdot\|\right)$ order convergence coincides with continuous convergence, while the order topology coincides with the compact-open topology.
\end{corollary}
\begin{proof}
First, note that the weak and $\mathrm{w}^{*}$-topologies agree on $F^{*}$, and so
\[E:=\Aff\left(\Ba_{F^{*}},\|\cdot\|\right)=\Aff\left(\Ba_{F^{*}},\mathrm{w}\right)=\Aff\left(\Ba_{F^{*}},\mathrm{w}^{*}\right).\] Since $F$ is separable and reflexive, it follows that $F^{*}$ is separable. %, while smoothness of $F$ implies that $F^{*}$ is strictly convex (see \cite[Corollary 7.23(ii)]{fhhmz}).
Thus, $F$ satisfies the conditions of Theorem~\ref{iceord}, and so order convergence, bounded weak convergence and bounded pointwise convergence all agree on $E$. By part~(ii) of Proposition~\ref{iceo}, bounded pointwise convergence also agrees with continuous convergence in $E$. As $E\simeq F\oplus_{1}\R$ is reflexive, bounded $\mathrm{w}^{*}$-convergence coincides with bounded weak convergence, and hence with order convergence. The compact-open topology is equal to the bounded $\mathrm{w}^{*}$-topology by part~(iii) of Proposition~\ref{iceo}, and so it is equal to the order topology.
\end{proof}

The following result gives more information on the spaces considered in Example~\ref{lur}.

\begin{corollary}\label{iceord1}
Let $F$ be a non-reflexive Banach space such that $F^{*}$ is separable and locally uniformly rotund. Then, $\icecream\left(F\right)$ is directed order complete, but not order complete.
\end{corollary}
\begin{proof}
First, $\icecream\left(F\right)$ is directed order complete by part~(ii) of Proposition~\ref{icedo}. As $F^{*}$ is LUR, it is strictly convex and has the $\mathrm{w}^{*}$-Kadec--Klee property, i.e.\ every $\mathrm{w}^{*}$-convergent sequence in $\So_{F^{*}}$ is norm-convergent (see \cite[Exercise~8.45]{fhhmz}). Since $F$ is separable, the $\mathrm{w}^{*}$-topology on $\So_{F^{*}}$ is metrizable. We have two metrizable topologies on $\So_{F^{*}}$, which agree on sequences, hence these topologies coincide. As the weak topology is stronger than the $\mathrm{w}^{*}$-topology but weaker than the norm topology, it follows that all three topologies coincide on $\So_{F^{*}}$. Hence, $F$ satisfies the conditions of Theorem~\ref{iceord}, and since it is not reflexive, by part~(iv) of that theorem, $\icecream\left(F\right)$ is not order complete.
\end{proof}

\subsection{Lattice operations in the ice cream spaces}\label{icelat}

In this subsection we explore the near lack of lattice operations in the ice cream spaces. Throughout this subsection, $F$ is again a normed space.

\begin{lemma}\label{iceable}
Let $0_{F}\ne f\in F$. There is an ordered vector space structure on $F$ such that $f$ is an order unit with respect to this order and the order unit norm on $F$ agrees with the original norm if and only if $\nu\left(f\right)=\pm 1$, for every $\nu\in \ext\Ba_{F^{*}}$. In this case, $\|f\|=1$.
\end{lemma}
\begin{proof}
Necessity: According to \cite[Definition~1.10 and Theorem~1.19]{as2}, we have that $\Ba_{F^{*}}=\conv\left(K_{F}\cup-K_{F}\right)$, where $K_{F}$ is the set of states on $\left(F,f\right)$. Then, any extreme point of $\Ba_{F^{*}}$ has to be contained in either $K_{F}$ or $-K_{F}$ and be extreme there. Thus, for every $\nu\in\ext\Ba_{F^{*}}$ either $\nu$ or $-\nu$ is a pure state on $\left(F,f\right)$, and so in particular, $\nu\left(f\right)=\pm 1$.

Sufficiency: Let $K:=\left\{\nu\in \Ba_{F^{*}}: \nu\left(f\right)=1\right\}$, which is $\mathrm{w}^{*}$-compact. By assumption, $\ext\Ba_{F^{*}}\subseteq K\cup -K$, and so by the Krein--Milman theorem, we have that
\[\|h\|=\bigvee_{\nu\in \ext\Ba_{F^{*}}}\left|\nu\left(h\right)\right|=\bigvee_{\nu\in K}\left|\nu\left(h\right)\right|,\quad\text{for every }h\in F.\]
Let $J\colon F\to\Co\left(K\right)$ be defined by $\left[Jh\right]\left(\nu\right):=\nu\left(h\right)$. Then, $Jf=\1$, and by the preceding step, $J$ is an isometry, and so we may view $F$ as a subspace of $\Co\left(K\right)$ with $\1=f\in F$. Hence, the norm on $F$ is the order unit norm of $\left(F,f\right)$.%We introduce a partial order on $F$ by declaring $F_{+}:=\left\{g\in F,~ \forall \nu\in K,~\nu\left(g\right)\ge 0\right\}$. It is easy to see that $F_{+}$ is a closed (and hence Archimedean) convex wedge. For every $g\in \Ba_{F}$ and $\nu\in K$ we have that $\nu\left(g-f\right)=\nu\left(g\right)-1\le 0$, hence $g\le f$, and similarly $g\ge -f$. This shows that $\Ba_{F}\subseteq\left[-f,f\right]$. Conversely, if $g\in \left[-f,f\right]$, it follows that $\left|\nu\left(g\right)\right|\le 1$, for every $\nu\in \ext\Ba_{F^{*}}$. It can be deduced from Krein-Milman theorem that $\left|\nu\left(g\right)\right|\le 1$, for every $\nu\in \Ba_{F^{*}}$, and so $\|g\|\le 1$. Thus, $\left[-f,f\right]\subseteq\Ba_{F}$, and so the order unit norm agrees with $\|\cdot\|$. Note that it also follows that $F_{+}$ is a proper cone.
\end{proof}

Recall from Proposition~\ref{ocp} that a set in a directed Archimedean ordered vector space $E$ has an infimum if and only if it has the same infimum in $H$, where $H$ is a vector lattice which contains $E$ as a supremum-dense subspace. According to part~(iii) of Proposition~\ref{icecream}, in the case $E=\icecream\left(F\right)$ we can take $H:=\Co\left(\mathop{\overline{\ext}}\Ba_{F^{*}},\mathrm{w}^{*}\right)$.

\begin{proposition}\label{antilattice}
For $f,g\in F$ and $r,s\in\R$ we have that $\left(f,r\right)\wedge \left(g,s\right)=0_{\smallicecream\left(F\right)}$ if and only if one of the following conditions is satisfied:
\begin{itemize}
\item $f=0_{F}$, $r=0$ and $\|g\|\le s$;
\item $g=0_{F}$, $s=0$ and $\|f\|\le r$;
\item $r,s>0$, $sf+rg=0_{F}$, and $\frac{1}{r}f=-\frac{1}{s}g$ satisfies the conditions of Lemma~\ref{iceable} (in particular, $\|f\|=r$ and $\|g\|=s$).
\end{itemize}
\end{proposition}
\begin{proof}
Necessity: First, it follows that $\left(f,r\right), \left(g,s\right)\in \icecream\left(F\right)_{+}$, and so $\|f\|\le r$ and $\|g\|\le s$. In particular, if $r=0$, then $f=0_{F}$, and if $s=0$, then $g=0_{F}$. Hence, we may assume from now on that $r,s>0$.

It follows from the discussion before the proposition that $\left(f,r\right)\wedge \left(g,s\right)=0_{\smallicecream\left(F\right)}$ if and only if $\left<\left(f,r\right),\left(\nu,1\right)\right>\left<\left(g,s\right),\left(\nu,1\right)\right>=0$, for every $\nu\in\ext\Ba_{F^{*}}$. This is equivalent to either $\nu\left(f\right)=-r$, or $\nu\left(g\right)=-s$, for every $\nu\in \ext\Ba_{F^{*}}$. Note that if for some $\nu\in \ext\Ba_{F^{*}}$ we have $\nu\left(f\right)=-r$, then $-\nu\in \ext\Ba_{F^{*}}$, and $-\nu\left(f\right)=r\ne -r$, and so $-\nu\left(g\right)=-s$, or $\nu\left(g\right)=s$. Conversely, if $\nu\left(g\right)=-s$, then $\nu\left(f\right)=r$, and so $\nu\left(sf+rg\right)=0$ for every $\nu\in \ext\Ba_{F^{*}}$. By the Krein--Milman theorem, it follows that $sf+rg=0_{F}$. It is also clear that for $h:=\frac{1}{r}f$ we have $\nu\left(h\right)=\pm 1$, for every $\nu\in \ext\Ba_{F^{*}}$.

Sufficiency of the first two conditions is clear. The third case can be obtained by reversing the arguments in the previous paragraph.
\end{proof}

We can now derive a result which was originally established in \cite{gs}.

\begin{corollary}
If $\icecream\left(F\right)$ is weakly pervasive, then $\dim F\le 1$, in which case it is a vector lattice.
\end{corollary}
\begin{proof}
It is clear that $\icecream\left(F\right)$ is a vector lattice if $\dim F=0$. If $\dim F= 1$, one can show that the clockwise rotation by $\frac{\pi}{4}$ is an order isomorphism between $\icecream\left(F\right)$ and $\R^{2}$, which is a vector lattice.

Assume that $\dim F> 1$, so that there are linearly independent $\nu,\mu\in \ext \Ba_{F^{*}}$ (this follows from the Krein--Milman theorem). Find $f\in\So_{F}$ such that $\mu\left(f\right)=0\ne \nu\left(f\right)$. Then, $f$ does not satisfy the conditions of Lemma~\ref{iceable}, and so according to Proposition~\ref{antilattice}, the infimum of $\left(f,1\right)$ and $\left(-f,1\right)$ does not exist. On the other hand, let us show that $\left(f,1\right)\wedge_{\smallicecream\left(F\right)_{+}}\left(-f,1\right)=\left(0_{F},0\right)$, which will prove that $\icecream\left(F\right)$ is not weakly pervasive. Indeed, if $0_{\smallicecream\left(F\right)}\le \left(g,r\right)\le \left(\pm f,1\right)$, then $\|f-g\|,\|f+g\|\le 1-r$, hence
\[2=\|2f\|\le \|f-g\|+\|f+g\|\le 2-2r,\]
which yields $r=0$ and $g=0_{F}$.
\end{proof}

We now characterize disjointness in ice cream spaces. Recall that vectors in a directed Archimedean ordered vector space $E$ are disjoint if they are disjoint in the Riesz completion $E^{\rho}$ of $E$. By part~(iii) of Proposition~\ref{icecream}, for $E=\icecream\left(F\right)$ we can replace $E^{\rho}$ with $\Co\left(\mathop{\overline{\ext}}\Ba_{F^{*}},\mathrm{w}^{*}\right)$.

\begin{proposition}
For $f,g\in F$ and $r,s\in\R$ we have that $\left(f,r\right)\perp \left(g,s\right)$ if and only if one of the following conditions is satisfied:
\begin{itemize}
\item $f=0_{F}$, $r=0$;
\item $g=0_{F}$, $s=0$;
\item $r,s\ne 0$, $sf+rg=0_{F}$, and $\frac{1}{r}f=-\frac{1}{s}g$ satisfies the conditions of Lemma~\ref{iceable} (in particular, $\|f\|=\left|r\right|$ and $\|g\|=\left|s\right|$);
\item $r=s=0$ and $\ext\Ba_{F^{*}}\subseteq f^{\top}\cup g^{\top}$.\footnote{In order to distinguish from the disjoint complement, we denote the annihilator of $A\subseteq F$ by $A^{\top}$.}
\end{itemize}
\end{proposition}
\begin{proof}
Arguing as in the proof of Proposition~\ref{antilattice}, one can show that $\left(f,r\right)\perp \left(g,s\right)$ if and only if either $\nu\left(f\right)=-r$, or $\nu\left(g\right)=-s$, for every $\nu\in \ext\Ba_{F^{*}}$.

Moreover, if $r,s\ne 0$, we are in a similar situation as in that proposition.

Assume that $r\ne 0$ and $s=0$. If $\nu\in \ext\Ba_{F^{*}}$ satisfies $\nu\left(f\right)=-r$, then $-\nu\in \ext\Ba_{F^{*}}$ and $-\nu\left(f\right)=r\ne -r$, and so $-\nu\left(g\right)=0$, or $\nu\left(g\right)=0$. If $\nu\in \ext\Ba_{F^{*}}$ is such that $\nu\left(f\right)\ne-r$, then also $\nu\left(g\right)=0$. Hence, $g=0_{F}$.

Similarly, if $r= 0$ and $s\ne 0$, we obtain that $f=0_{F}$.

If $r=s=0$, it follows that for every $\nu\in \ext\Ba_{F^{*}}$ either $\nu\left(f\right)=0$ or $\nu\left(g\right)=0$, so that $\nu\in f^{\top}\cup g^{\top}$.

Sufficiency of the third condition is proved using the arguments from the proof of Proposition~\ref{antilattice}; sufficiency of the other conditions is clear.
\end{proof}

\begin{corollary}\label{antilattice2}
If $\dim F>1$ and $F^{*}$ is strictly convex, then $\icecream\left(F\right)$ is disjointness-free. Moreover, whenever $P\subseteq \icecream\left(F\right)$ is a finite set such that $\bigvee P$ exists, then this supremum is in fact the maximum, i.e.\ there is $p\in P$ such that $\bigvee P=p$.
\end{corollary}
\begin{proof}
For the first claim, note that if $\dim F>1$ and $F^{*}$ is strictly convex, then no elements of $F$ satisfy the conditions of Lemma~\ref{iceable}, and the extreme points of $\Ba_{F^{*}}$ cannot be covered by two hyperplanes.

In order to prove the second claim, assume that $\left(f_{1},r_{1}\right),\ldots,\left(f_{n},r_{n}\right)\in \icecream\left(F\right)$ are such that $\bigwedge_{k=1}^{n}\left(f_{k},r_{k}\right)=0_{\smallicecream\left(F\right)}$. By the comment before Proposition~\ref{antilattice}, it follows that for every $\nu\in \mathop{\overline{\ext}}\Ba_{F^{*}}$ we have
\[\bigwedge_{k=1}^{n}\left<\left(f_{k},r_{k}\right),\left(\nu,1\right)\right>=\bigwedge_{k=1}^{n}\left(\nu\left(f_{k}\right)+r_{k}\right)=0.\] Since $F^{*}$ is strictly convex, for every $f\in F$ there is a unique $\nu_{f}\in \So_{F^{*}}$ with $\nu_{f}\left(f\right)=-\|f\|$ (combine \cite[Corollaries~7.23(i) and~7.22(ii)]{fhhmz}); in particular, if $\nu_{f}\ne \nu\in\So_{F^{*}}$, then $\|f\|>-\nu\left(f\right)$. Since $\dim F>1$ and $F^{*}$ is strictly convex, we have that $\mathop{\overline{\ext}}\Ba_{F^{*}}\supseteq \So_{F^{*}}$ is infinite. Hence, we can select $\nu\in \So_{F^{*}}$ such that $\nu\ne\nu_{f_{k}}$, for $k=1,\ldots,n$. Then, for every $k=1,\ldots,n$ we have that $\left(f_{k},r_{k}\right)\ge 0_{\smallicecream\left(F\right)}$, thus $r_{k}\ge \|f_{k}\|>-\nu\left(f_{k}\right)$. This implies $\bigwedge_{k=1}^{n}\left(\nu\left(f_{k}\right)+r_{k}\right)>0$, contradiction.
\end{proof}

\begin{remark}
In the case when $\dim F=\8$, the proof of the last claim in Corollary~\ref{antilattice2} is simplified: for every $\nu\in \mathop{\overline{\ext}}\Ba_{F^{*}}=\Ba_{F^{*}}$ we have $\bigwedge_{\left(f,r\right)\in P}\left<\left(f,r\right),\left(\nu,1\right)\right>=0$, and so taking $\nu=0_{F^{*}}$ yields $\bigwedge_{\left(f,r\right)\in P}r=0$. Hence, there is $\left(f,0\right)\in P$. Since $P\subseteq \icecream\left(F\right)_{+}$, it follows that $f=0_{F}$.
\qed\end{remark}

\subsection{Atomic vs strongly atomic spaces}\label{iceatm}

In this subsection, we discuss the possibility of extending Theorem~\ref{ordprec} beyond pervasive ordered vector spaces. Namely, we study the relation between order pre-compactness of order intervals and atomicity. Recall from Theorem~\ref{ordprec} that atomic pervasive spaces have order intervals which are pre-compact topological spaces with respect to order convergence. The next example shows that this is, in general, false in non-pervasive atomic spaces.

\begin{example}\label{watom}
Let $F:=\Co\left[0,1\right]\oplus c$, and let $E:=\left\{f\oplus g\in F: \int f=\lim_{\8} g\right\}$. Then, $E$ is Archimedean, and it is also directed, since it has an order unit. Let us show that $E$ is atomic. Let $\left\{e_{n}\right\}_{n\in\N}$ be the standard basis of $c_{00}\subseteq c$. Clearly, $\0\oplus e_{n}$ is an atom in $E$, for every $n\in\N$. Suppose $f\oplus g>\0\oplus\0$. Then, either $g>\0$, or $f>\0$, but in the latter case, $\lim_{\8} g=\int f>0$, and so $g>\0$. Hence, there is $n\in\N$ such that $g\left(n\right)>0$, thus $g\ge g\left(n\right)\left(\0\oplus e_{n}\right)$.

Let us show that $\1\oplus\1$ is not a supremum of a collection of atoms. To this end, note that the atoms in $\left[\0\oplus\0,\1\oplus\1\right]$ are of the form $\0\oplus re_{n}$, $r\in\left(0,1\right]$, $n\in\N$. Let $P$ be the set of these atoms. Then, $P^{\uparrow}=\left\{f\oplus g\in E_{+}: g\ge\1\right\}$. It is easy to see that $P^{\uparrow}$ does not have a least element, and so $P$ does not have a supremum.

Let us show that the interval $\left[\0\oplus\0,\1\oplus\1\right]$ is not order pre-compact. We construct a net $\left(f_{\gamma}\right)_{\gamma\in\Gamma}$ in $\Co\left[0,1\right]$ as in the proof of Proposition~\ref{pervnoc}, taking $Z:=\Q$. For every $\gamma\in\Gamma$ let $g_{\gamma}:=\left(\int f_{\gamma}\right)\1$. Using the same arguments one can show that if $f\oplus g$ is an eventual upper (resp. lower) bound of any subnet of the net $\left(f_{\gamma}\oplus g_{\gamma}\right)_{\gamma\in\Gamma}$, then $f\ge\1$ (resp. $f\le\0$). Consequently, no subnet of this net is order Cauchy. Starting with the net constructed in the proof of (iv)$\Rightarrow$(v) in Theorem~\ref{ordprec}, and selecting a subnet such that $\int f_{\gamma}\to r$, for some $r\in\left[0,1\right]$, shows that order convergence in $\left[\0\oplus\0,\1\oplus\1\right]$ does not have the Choquet property, hence is not topological.
\qed\end{example}

In an attempt to rescue the situation, we say that $e\in E_{+}$ is \emph{strongly atomic} if it is a supremum of a collection of atoms in $E_{+}$ (we also agree to count $0_{E}$ as strongly atomic), and call $E$ \emph{strongly atomic} if every element of $E_{+}$ is strongly atomic. Clearly, strong atomicity implies atomicity, and the converse is true if $E$ is pervasive and Archimedean, by Corollary~\ref{pervatomic} and Theorem~\ref{pervasive}. If $E$ is strongly atomic, and $F\subseteq E$ contains all atoms in $E$, then $F$ is strongly atomic. The space constructed in Example~\ref{watom} is atomic but not strongly atomic. It is easy to see that the supremum of a collection of strongly atomic elements of $E$ is strongly atomic. Strongly atomic elements are preserved by linear order isomorphisms, in particular positive scalar multiplication. We will see in Example~\ref{findimnot} that atomicity together with order compactness of order intervals does not imply strong atomicity. Example~\ref{smoothre} and the discussion before it refute that order compactness of order intervals implies atomicity. By Example~\ref{poly}, directed order pre-compactness of order intervals does not even imply existence of an atom, while in Example~\ref{megaice} we construct an order unit space with directed order compact intervals, which fails to be atomic. Example~\ref{noocomp} exhibits a strongly atomic space whose order intervals fail to be even order pre-compact.

This leaves us with the following problem.

\begin{question}
Does directed order compactness of order intervals imply existence of an atom?
\qed\end{question}

We now discuss atomicity of the ice cream spaces. Until the end of the section, we work with $\icecream\left(F\right)$, where $F$ is a normed space. We start with some observations.

\begin{lemma}\label{corresp}
There is an order isomorphism between the collection of faces of $\Ba_{F}$ and the collection of nonempty faces of $\icecream\left(F\right)_{+}$. It is given by
\[A\mapsto C_{A}:=\left\{r\left(f,1\right): f\in A,~r>0\right\}\cup\left\{0_{\smallicecream\left(F\right)}\right\},\]
and its inverse is given by $C\mapsto A_{C}:=\left\{f\in\Ba_{F}:\left(f,1\right)\in C\right\}$.
\end{lemma}
\begin{proof}
It is easy to see that if $C$ is a nonempty face of $\icecream\left(F\right)_{+}$, then $A_{C}$ is a face of $\Ba_{F}$.

Suppose now that $A\subseteq\Ba_{F}$ is a face. We will deduce from Proposition~\ref{face} that $C_{A}$ is a nonempty face of $\icecream\left(F\right)_{+}$. Indeed, if $g,h\in F$, $s\ge\|g\|$, and $t\ge\|h\|$ are such that $\left(g,s\right)+\left(h,t\right)\in C_{A}$, then $A\ni\frac{1}{s+t}\left(g+h\right)=\frac{s}{s+t}\frac{g}{s}+\frac{t}{s+t}\frac{h}{t}$ implies that $\frac{g}{s}\in A$, and thus $\left(g,s\right)=s\left(\frac{g}{s},1\right)\in C_{A}$ (we leave the case when $s$ or $t$ is $0$ to the reader).

It is also straightforward to check that the introduced correspondences are mutually inverse and order preserving.
\end{proof}

\begin{lemma}\label{ineq}
Let $f,g\in F$ and $t\in\left[0,1\right]$. Then, $\left(f,1\right)\ge t\left(g,1\right)$ if and only if there is $h\in\Ba_{F}$ such that $f=tg+\left(1-t\right)h$.
\end{lemma}
\begin{proof}
The case $t=1$ is trivial, and so we assume that $t<1$. If $f=tg+\left(1-t\right)h$, for some $h\in\Ba_{F}$, then $\|f-tg\|=\left(1-t\right)\|h\|\le 1-t$, and so $\left(f-tg,1-t\right)\ge 0_{\smallicecream\left(F\right)}$. Conversely, if $\left(f,1\right)\ge t\left(g,1\right)$, then $\|f-tg\|\le 1-t$, or $h:=\frac{f-tg}{1-t}\in\Ba_{F}$; we then have $f=tg+\left(1-t\right)h$.
\end{proof}

\begin{corollary}\label{atom}
Atoms in $\icecream\left(F\right)$ are of the form $r\left(g,1\right)$, where $g\in\ext\Ba_{F}$, and $r>0$.
\end{corollary}

\begin{proposition}\label{atm}The following conditions are equivalent:
\item[(i)] $\icecream\left(F\right)$ is atomic;
\item[(ii)] Every nonempty face of $\Ba_{F}$ contains an extreme point;
\item[(iii)] Every $f\in \Ba_{F}$ can be written as $f=tg+\left(1-t\right)h$, where $t\in\left(0,1\right]$, $g\in\ext\Ba_{F}$ and $h\in \Ba_{F}$;
\item[(iv)] Every $f\in \So_{F}$ can be written as $f=tg+\left(1-t\right)h$, where $t\in\left(0,1\right]$, $g\in\ext\Ba_{F}$ and $h\in \Ba_{F}$.\medskip

In particular, if $\dim F<\8$, then $\icecream\left(F\right)$ is atomic.
\end{proposition}
\begin{proof}
(i)$\Leftrightarrow$(ii): The order isomorphism between the collection of faces of $\Ba_{F}$ and the collection of nonempty faces of $\icecream\left(F\right)_{+}$ considered in Lemma~\ref{corresp} sets extreme points of $\Ba_{F}$ in correspondence with extreme rays of $\icecream\left(F\right)_{+}$. Hence, the equivalence follows from part~(ii) of Corollary~\ref{atmo}.

(i)$\Leftrightarrow$(iii): According to Corollary~\ref{atom}, $\icecream\left(F\right)$ is atomic if and only if for every $f\in\Ba_{F}$ there are $g\in\ext\Ba_{F}$ and $t>0$ such that $\left(f,1\right)\ge t\left(g,1\right)$, which by Lemma~\ref{ineq} is equivalent to $f=tg+\left(1-t\right)h$, for some $h\in\Ba_{F}$.

(iii)$\Rightarrow$(iv) is trivial. To get the converse, we may assume that $r:=\|f\|\in\left(0,1\right)$, and find $t\in\left(0,1\right]$, $h\in \Ba_{F}$ and $g\in\ext\Ba_{F}$ such that $\frac{1}{r}f=tg+\left(1-t\right)h$. Then, $f=rtg+r\left(1-t\right)h=rtg+\left(1-rt\right)e$, where $e=\frac{r\left(1-t\right)}{1-rt}h$. As $r\left(1-t\right)\le 1-rt$, we conclude that $e\in\Ba_{F}$.

The last claim follows from (ii), the Krein--Milman theorem and the fact that every face of a convex set in $\R^{n}$ is closed (see \cite[Corollary~18.1.1]{rockafellar}).
\end{proof}

It follows from part~(i) of Proposition~\ref{icedo} that directed order pre-compactness of order intervals in $\icecream\left(F\right)$ implies that $\dim F<\8$, which yields atomicity of $\icecream\left(F\right)$ by Proposition~\ref{atm}.

\begin{remark}
It was proved in \cite[Theorem~1]{als} that if $F$ is a Banach space, then the conditions of Proposition~\ref{atm} are also equivalent to the fact that for every $f\in \Ba_{F}$ there are $\left(t_{n}\right)_{n\in\N}\subseteq\left[0,1\right]$ and $\left(g_{n}\right)_{n\in\N}\subseteq\ext\Ba_{F}$ such that $\sum_{n\in\N}t_{n}=1$ and $\sum_{n\in\N}t_{n}g_{n}=f$. Note that this property is strictly weaker than the Krein--Milman property (see \cite[Remark~3.4]{als}).
\qed\end{remark}

According to part~(iv) of Theorem~\ref{iceord} and Proposition~\ref{atm}, if $F$ is a reflexive separable Banach space whose unit ball has a face without extreme points, then $\icecream\left(F\right)$ is an example of an order unit space which is not atomic, but whose order intervals are compact metrizable spaces with respect to order convergence. We now present a construction of such an $F$. We even make it so that $F^{*}$ is LUR.

\begin{example}\label{smoothre}
For $n\in\N$ and $p\ge 1$ let $\Ba_{p}^{n}$ be the unit ball in $\R^{n}$ with respect to the $\ell_{p}$-norm. Let $u_{n}:=\left(\frac{1}{n},\dots,\frac{1}{n}\right)\in\Ba_{1}^{n}\cap \frac{1}{\sqrt{n}}\Ba_{2}^{n}$.

Fix $n\in\N$. Let $F_{n}$ be $\R^{n}$ endowed with the norm whose unit ball is $\Ba_{F_{n}}:=\Ba_{1}^{n}+\frac{1}{\sqrt{n}}\Ba_{2}^{n}$, which is a compact convex body. If $f\in\partial\Ba_{F_{n}}$, then it must be on the $\ell_{2}$-sphere centered at some $e\in \partial\Ba_{1}^{n}$ of radius $\frac{1}{\sqrt{n}}$. Any supporting hyperplane to $\Ba_{F_{n}}$ through $f$ cannot intersect this sphere in any other point, which means that it is the tangent hyperplane to the sphere, ensuring its uniqueness. Hence, the support hyperplane to $\Ba_{F_{n}}$ is unique at each boundary point, and so $F_{n}$ is a smooth normed space. Then, $F_{n}^{*}$ is strictly convex, hence LUR (combine reflexivity of $F_{n}$, \cite[Corollary~7.23(ii), Exercise~9.1 and Fact~9.5]{fhhmz}).

Let $\nu_{n}\colon\R^{n}\to\R$ be the linear functional defined by $\nu_{n}\left(r_{1},\dots,r_{n}\right):=r_{1}+\dots+r_{n}$. By the Cauchy--Schwarz inequality, the maximal value of $\nu_{n}$ on $\Ba_{F_{n}}$ is $2$, and so
\[A_{n}:=\nu_{n}^{-1}\left(2\right)\cap \Ba_{F_{n}}=u_{n}+\left\{\left(r_{1},\dots,r_{n}\right): r_{1},\dots,r_{n}\ge 0,~ r_{1}+\dots+r_{n}=1\right\}\]
is a face of $\Ba_{F_{n}}$. The extreme points of $A_{n}$ are of the form $u_{n}+e_{k}^{n}$, $k=1,\dots,n$, where $e_{1}^{n},\dots,e_{n}^{n}$ is the standard basis for $\R^{n}$. It follows that if $f\in\ext A_{n}$, then $\|f\|_{\ell_{\8}}=1+\frac{1}{n}$. Also, note that $\|u_{n}\|_{\ell_{\8}}=\frac{1}{n}$ and $2u_{n}\in A_{n}$.

Let $F$ be the $\ell_{2}$-sum of the $F_{n}$'s, over $n\in\N$. It is easy to see that $F$ is reflexive and separable, and moreover $F^{*}$ is LUR (combine \cite[Exercises~2.32 and~8.4]{fhhmz}). Fix any $\left(a_{n}\right)_{n\in\N} \in \So_{\ell_{2}}$ with $a_{n} > 0$ for all $n$. Define
$A:= \left\{\left(a_{n}f_{n}\right)_{n\in\N} \in F: \forall n, f_{n}\in A_{n}\right\}$ and $A':=\left\{\left(a_{n}f_{n}\right)_{n\in\N} \in A: \|f_{n}\|_{\ell_{\8}}\to 0\right\}$.

It is clear that $A$ is a convex subset of $\So_{F}$. Let us show that it is a face of $\Ba_{F}$. Assume that $\left(e_{n}\right)_{n\in\N}, \left(h_{n}\right)_{n\in\N}\in\So_{F}$ and $t\in\left(0,1\right)$ are such that $\left(te_{n}+\left(1-t\right)h_{n}\right)_{n\in\N}\in A$. Then, for every $n\in\N$ there is $f_{n}\in A_{n}$ such that $te_{n}+\left(1-t\right)h_{n}=a_{n}f_{n}$. We have that
\begin{align*}
1&=\sum\limits_{n\in\N}\left\|te_{n}+\left(1-t\right)h_{n}\right\|^{2}_{F_{n}}\le \sum\limits_{n\in\N}\left(t\|e_{n}\|_{F_{n}}+\left(1-t\right)\|h_{n}\|_{F_{n}}\right)^{2} \\
&\le t\sum\limits_{n\in\N}\|e_{n}\|^{2}_{F_{n}}+\left(1-t\right)\sum\limits_{n\in\N}\|h_{n}\|^{2}_{F_{n}}=t+1-t=1.
\end{align*}
Hence, both inequalities in the above formula are in fact equalities. Due to strict convexity of the $\ell_{2}$-norm, the second of these equalities yields $\left(\|e_{n}\|\right)_{n\in\N}=\left(\|h_{n}\|\right)_{n\in\N}$. Therefore, the first equality now becomes $a_{n}=\|a_{n}f_{n}\|=\left\|te_{n}+\left(1-t\right)h_{n}\right\|=\|e_{n}\|=\|h_{n}\|$, for every $n\in\N$. It follows that $\frac{1}{a_{n}}e_{n},\frac{1}{a_{n}}h_{n}\in\So_{F_{n}}$ satisfy $\frac{t}{a_{n}}e_{n}+\frac{1-t}{a_{n}}h_{n}\in A_{n}$, which implies $\frac{1}{a_{n}}e_{n},\frac{1}{a_{n}}h_{n}\in A_{n}$, for every $n\in\N$. We conclude that $\left(e_{n}\right)_{n\in\N}, \left(h_{n}\right)_{n\in\N}\in A$.

It is also clear that $A'$ is convex, and it is nonempty as $\left(2a_{n}u_{n}\right)_{n\in\N}\in A'$. Let us show that it is a face of $A$, and hence a face of $\Ba_{F}$. Suppose that $\left(e_{n}\right)_{n\in\N}\in A\setminus A'$, $\left(h_{n}\right)_{n\in\N}\in A$ and $t\in\left(0,1\right)$. There are $\varepsilon>0$ and an increasing sequence $\left(n_{k}\right)_{k\in\N}$ with $\|e_{n_{k}}\|_{\ell_{\8}}\ge\varepsilon$, for every $k\in\N$. Then, $te_{n_{k}}+\left(1-t\right)h_{n_{k}}\ge te_{n_{k}}$, hence
\[\left\|te_{n_{k}}+\left(1-t\right)h_{n_{k}}\right\|_{\ell_{\8}}\ge \|te_{n_{k}}\|_{\ell_{\8}}\ge t\varepsilon,\quad\text{for every }k\in\N.\] Thus, $\lim_{k \to \infty}\left\|te_{n_{k}}+\left(1-t\right)h_{n_{k}}\right\|_{\ell_{\8}}\ne 0$, and therefore $\left(te_{n}+\left(1-t\right)h_{n}\right)_{n\in\N}\notin A'$.

To show that $A'$ has no extreme points, it is enough to prove that all extreme points of $A$ lie outside of $A'$. We have $\ext A= \left\{\left(a_{n}f_{n}\right)_{n\in\N} \in F: \forall n, f_{n}\in \ext A_{n}\right\}$. As $f_{n}\in \ext A_{n}$ implies $\|f_{n}\|=1+\frac{1}{n}\not\to 0$, we conclude that $\ext A\cap A'=\varnothing$.
\qed\end{example}

Before we characterize strong atomicity of an ice cream space in a similar way to Proposition~\ref{atm}, we present another order-theoretic property of such spaces.

\begin{proposition}\label{iceprop}
If $\dim F>0$, then $\partial \icecream\left(F\right)_{+}$ is supremum-dense in $\icecream\left(F\right)_{+}$. That is, every element in $\icecream\left(F\right)_{+}$ is the supremum in $\icecream\left(F\right)$ of a subset of $\partial \icecream\left(F\right)_{+}$.
\end{proposition}
\begin{proof}
Let $f\in \Ba_{F}$ and $P:=\left[0_{\smallicecream\left(F\right)}, (f,1)\right]\cap \partial \icecream\left(F\right)_{+}$. It is enough to show that every upper bound $\left(e,r\right)$ of $P$ satisfies $\left(e,r\right) \ge \left(f,1\right)$.

First, assume that $e\ne f$. For the continuous function $v\colon\R\to\R$ defined by
\[v\left(s\right):=s\|f-e\|+\left\|f+s\left(f-e\right)\right\|,\]
we have $v\left(0\right)=\|f\|\le 1$, and $\lim_{s\to+\8}v\left(s\right)=+\8$, hence there is $s_{0}\ge 0$ such that $v\left(s_{0}\right)=1$.

Let $g:=f+s_{0}\left(f-e\right)$; we have $\|g\|=1-s_{0}\|f-e\|$ and $f-g=s_{0}\left(e-f\right)$, which yields $\left(g,\|g\|\right)\le\left(f,1\right)$, so that $\left(g,\|g\|\right)\in P$. It follows that $\left(g,\|g\|\right)\le \left(e,r\right)$, hence
$\left(1+s_{0}\right)\left\|e-f\right\|=\|e-g\|\le r-\|g\|=r-1+s_{0}\|f-e\|$, and so $\left\|e-f\right\|\le r-1$, i.e.\ $\left(e,r\right)\ge\left(f,1\right)$.

If $e=f\ne 0_{F}$, for every $\varepsilon>0$ we have that $\left(1+\varepsilon\right)\left(e,r\right)\ge\left(e,r\right)\ge P$, and so by the preceding case we have $\left(1+\varepsilon\right)\left(e,r\right)\ge\left(f,1\right)$. Since $\varepsilon$ was arbitrary, we conclude that $\left(e,r\right)\ge\left(f,1\right)$. Finally, if $e=f=0_{F}$, take $h\in\So_{F}$. Then, $\bigl(\frac{1}{2}h,\frac{1}{2}\bigr)\in P$ is dominated by $\left(0_{F},r\right)$, which implies $r-\frac{1}{2}\ge \left\|\frac{1}{2}h\right\|=\frac{1}{2}$, and so $r\ge 1$. Thus, $\left(e,r\right)\ge\left(0_{F},1\right)=\left(f,1\right)$.
\end{proof}

\begin{theorem}\label{stratm}
\item[(i)] If $f\in \So_{F}$, then $\left(f,1\right)$ is strongly atomic in $\icecream\left(F\right)$ if and only if for every $t\in\left[0,1\right)$ there are $g,h\in \Ba_{F}$ such that $g\in\ext\Ba_{F}$ and $f=tg+\left(1-t\right)h$.
\item[(ii)] $\icecream\left(F\right)$ is strongly atomic if and only if for every $f\in \So_{F}$ and $t\in\left[0,1\right)$ there are $g,h\in \Ba_{F}$ such that $g\in\ext\Ba_{F}$ and $f=tg+\left(1-t\right)h$.
\end{theorem}
\begin{proof}
(i): Let $P$ be the collection of all atoms beneath $\left(f,1\right)$.

Necessity: Let $t\in\left[0,1\right)$. We have $\left(0_{F},2t\right)\not\ge\left(f,1\right)$, since $\|0_{F}-f\|=1>2t-1$. As $\left(f,1\right)=\bigvee P$, it follows that $\left(0_{F},2t\right)$ is not an upper bound for $P$. Hence, there is an atom $e\in P$ such that $\left(0_{F},2t\right)\not\ge e$. By Corollary~\ref{atom}, there are $g\in\ext\Ba_{F}$ and $r>0$ such that $e=r\left(g,1\right)$. The condition $\left(0_{F},2t\right)\not\ge r\left(g,1\right)$ implies $r=\|0_{F}-rg\|>2t-r$, therefore $r>t$. It follows that $\left(f,1\right)\ge r\left(g,1\right)\ge t \left(g,1\right)$, which by Lemma~\ref{ineq} yields $f=tg+\left(1-t\right)h$, for some $h\in\Ba_{F}$.

Sufficiency: Let $\left(e,r\right)$ be an upper bound for $P$. Fix $\varepsilon\in\left(0,1\right)$. There are $g\in\ext\Ba_{F}$ and $h\in \Ba_{F}$ such that $f=\left(1-\varepsilon\right)g+\varepsilon h$. For $u:=\left(1-\varepsilon\right)g$, we have $\|f-u\|=\varepsilon\|h\|\le\varepsilon$, and so $\left(f,1\right)\ge \left(u,1-\varepsilon\right)$. By Corollary~\ref{atom}, the latter vector is an atom, therefore contained in $P$, and so by our assumption, $\left(e,r\right)\ge\left(u,1-\varepsilon\right)$, i.e.\ $\|e-u\|\le r-1+\varepsilon$. Hence,
\[\|e-f\|\le \|f-u\|+\|e-u\|\le \varepsilon+r-1+\varepsilon=r-1+2\varepsilon.\] As $\varepsilon$ was arbitrary, we conclude that $\|e-f\|\le r-1$, thus $\left(e,r\right)\ge\left(f,1\right)$. This justifies that $\left(f,1\right)=\bigvee P$.

(ii): Necessity follows immediately from (i). For sufficiency, first note that (i) guarantees that every vector in $\partial \icecream\left(F\right)_{+}$ is strongly atomic, and then use Proposition~\ref{iceprop} and the fact that the collection of strongly atomic vectors is closed under taking suprema.
\begin{comment}

it is enough to show that if $f\in \Ba_{F}$ and $\left(e,r\right)$ is an upper bound for the collection $P$ of all atoms below $\left(f,1\right)$, then $\left(e,r\right)\ge\left(f,1\right)$.\medskip

First, assume that $e\ne f$. For the continuous function $\icecream\left(s\right):=s\|f-e\|+\left\|f+s\left(f-e\right)\right\|$ we have $\icecream\left(s\right)=\|f\|\le 1$, and $\lim\limits_{s\to+\8}\icecream\left(s\right)=+\8$, hence there is $s\ge 0$ such that $\icecream\left(s\right)=1$.

Let $g:=f+s\left(f-e\right)$; we have $\|g\|=1-s\|f-e\|$ and $f-g=s\left(e-f\right)$, which yields $\left(g,\|g\|\right)\le\left(f,1\right)$. Let $Q$ be the collection of atoms below $\left(g,\|g\|\right)$. We have $Q\subseteq P\le \left(e,r\right)$. It follows from (i) that $\left(g,\|g\|\right)$ is strongly atomic, therefore $\left(g,\|g\|\right)=\bigvee Q\le \left(e,r\right)$. The latter yields
$\left(1+s\right)\left\|e-f\right\|=\|e-g\|\le r-\|g\|=r-1+s\|f-e\|$, and so $\left\|e-f\right\|\le r-1$, i.e.\ $\left(e,r\right)\ge\left(f,1\right)$.

If $e=f$, for every $\varepsilon>0$ we have that $\left(1+\varepsilon\right)\left(e,r\right)\ge\left(e,r\right)\ge P$, and so by the preceding case we have $\left(1+\varepsilon\right)\left(e,r\right)\ge\left(f,1\right)$. Since $\varepsilon$ was arbitrary, we conclude $\left(e,r\right)\ge\left(f,1\right)$.
\end{comment}
\end{proof}

\begin{comment}
\begin{question}
What is a characterization of strong atomicity of $\left(f,1\right)$, when $\|f\|<1$?
\end{question}
\end{comment}

\begin{proposition}\label{strictstrong}
\item[(i)] If $F$ is strictly convex, then $\icecream\left(F\right)$ is strongly atomic. The converse holds if $\dim F<\8$.
\item[(ii)] If $\dim F<\8$ and $f\in \So_{F}$, then $\left(f,1\right)$ is strongly atomic if and only if it is an atom if and only if $f\in \ext\Ba_{F}$.
\end{proposition}
\begin{proof}
The first claim in (i) follows immediately from Theorem~\ref{stratm}, while the second claim follows from (ii).

(ii): The second equivalence was established in Corollary~\ref{atom}. It is also clear that every atom is strongly atomic. Suppose that $\left(f,1\right)$ is strongly atomic. According to \cite[Theorem~2.10]{al}, if $f\in\mathop{\overline{\ext}}\Ba_{F}\setminus \ext\Ba_{F}$, then the criterion in part~(i) of Theorem~\ref{stratm} is not satisfied, contradiction. It is also clear that if $f\notin\mathop{\overline{\ext}}\Ba_{F}$, the criterion in part~(i) of Theorem~\ref{stratm} is not satisfied either. We conclude that $f$ is extreme.
\end{proof}

The converse to the first claim in (i) is false in general. Namely, a non-strictly convex Banach space $F$ which satisfies the conditions of Theorem~\ref{stratm} was constructed in \cite[Section~2]{als}.

\begin{example}\label{findimnot}
If $F$ is finite-dimensional but not strictly convex, then $\icecream\left(F\right)$ is atomic by Proposition~\ref{atm}, order intervals in $\icecream\left(F\right)$ are order compact, but $\icecream\left(F\right)$ is not strongly atomic by part~(i) of Proposition~\ref{strictstrong}.
\qed\end{example}

\begin{example}\label{stratomless}
If $F$ and $F^{*}$ are strictly convex with $\dim F=\8$ (e.g., if $F=\ell_{2}$), then $\icecream\left(F\right)$ is strongly atomic by part~(i) of Proposition~\ref{strictstrong}, and it is supremum-dense in an atomless vector lattice $\Co\left(\Ba_{F^{*}},\mathrm{w}^{*}\right)$ by part~(iii) of Proposition~\ref{icecream}.
\qed\end{example}

It follows from Theorem~\ref{stratm} that if $\icecream\left(F\right)$ is strongly atomic, then $\ext\Ba_{F}$ is dense in $\So_{F}$. Let us show that the converse is false.

\begin{example}\label{ellipse}
  Let $f\colon\R^{2}\to\R$ be defined by
\[f\left(x,y\right):=\sqrt{(x - 1)^2 + y^2} + \sqrt{(x + 1)^2 +y^2}.\]
Clearly, $f$ is a symmetric convex function, therefore $B:=\left\{\left(x,y,z\right): \left|z\right|+f\left(x,y\right) \le 4\right\}$ is a symmetric convex body; see the figure below.
%\begin{comment}%
\begin{figure}[H]
    \centering
    \includegraphics[scale=0.4 ]{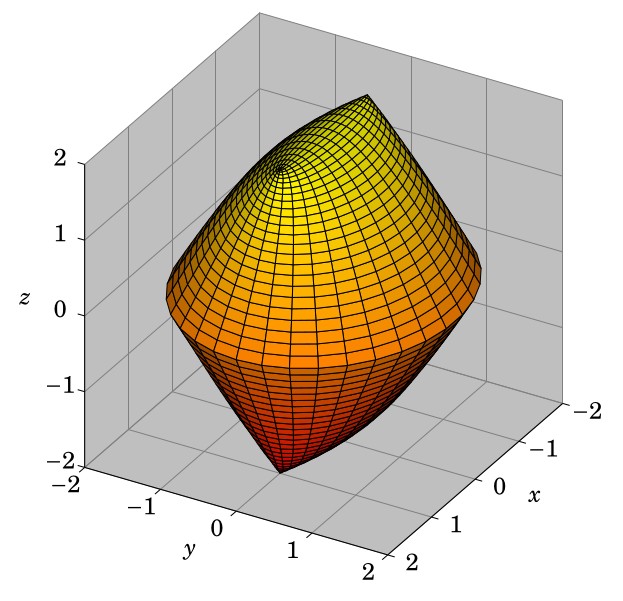}
    \label{fig:placeholder}
\end{figure}
%\end{comment}
Moreover, when $y\ne 0$, we have that
\[
\frac{\partial^{2}f}{\partial x^{2}}\left(x,y\right)=\frac{y^2}{((x - 1)^2 + y^2)^{\frac{3}{2}}}+\frac{y^2}{((x + 1)^2 + y^2)^{\frac{3}{2}}}>0
\]
and the determinant of the Hessian matrix of $f$ is
\[
\frac{4 y^2}{\left(\left((x -1)^2 + y^2\right)\left((x+1)^2 + y^2\right)\right)^{\frac{3}{2}}}>0,
\]
implying that $f$ is strictly convex when $y\ne0$. It follows that all points on the boundary $\partial B= \left\{\left(x,y,z\right): \left|z\right|=4-f\left(x,y\right)\right\}$ with $y\ne 0$ are extreme points of $B$. Note that $B$ is not strictly convex, as $\partial B$ contains $\left\{\left(x,y,z\right): y=0,~ \left|z\right|+\left|x+1\right|+\left|x-1\right|=4\right\}$, which is a broken line. Now let $F$ be $\R^{3}$ with the norm given by the Minkowski functional of $B$. Then, $\Ba_{F}=B$, and so $\ext\Ba_{F}$ is dense in $\So_{F}$. However, $\Ba_{F}=B$ is not strictly convex, and so by part~(i) of Proposition~\ref{strictstrong}, $\icecream\left(F\right)$ is not strongly atomic.
\qed\end{example}

Atoms in ice cream spaces may behave in a very different way from atoms in vector lattices.

\begin{example}
The sum of two atoms may fail to be a strongly atomic element. Indeed, let $F:=\ell_{\8}^{2}$, and let $f:=\left(1,1\right)$ and $g:=\left(1,-1\right)$. Note that $\frac{1}{2}\left(f+g\right)\in \So_{F}$ is not an extreme point of $\Ba_{F}$. Then, $e:=\left(f,1\right)$ and $h:=\left(g,1\right)$ are atoms in $\icecream\left(F\right)$, but $\frac{1}{2}\left(e+h\right)=\left(\frac{1}{2}\left(f+g\right),1\right)$ is not strongly atomic, according to part~(ii) of Proposition~\ref{strictstrong}.

It is also possible that the sum of two atoms is strongly atomic, but the supremum of the two summands does not exist. Indeed, let $H:=\ell_{2}^{2}$, so that $\icecream\left(H\right)$ is strongly atomic by part~(i) of Proposition~\ref{strictstrong}. Let $f:=\left(1,0\right)$ and $g:=\left(0,1\right)$, so that $\left(f,1\right)$ and $\left(g,1\right)$ are atoms in $\icecream\left(H\right)$, their sum is strongly atomic (because every positive vector in $\icecream\left(H\right)$ is), but $\left(f,1\right)$ and $\left(g,1\right)$ do not have a supremum, as $\icecream\left(H\right)$ is an antilattice (Corollary~\ref{antilattice2}), and $\left(f,1\right)$ and $\left(g,1\right)$ are incomparable.
\qed\end{example}

We conclude the section with two examples. The first one shows that strong atomicity does not imply order (and hence directed order) pre-compactness of the order intervals.

\begin{example}\label{noocomp}
Let $F:=\left(\ell_{1},\|\cdot\|_{*}\right)$, where $\|\cdot\|_{*}$ is the renorming of $\ell_{1}$ defined by $\|f\|_{*}:=\sqrt{\|f\|_{1}^{2}+\|f\|_{2}^{2}}$. One can show that this norm is strictly convex, which implies that $\icecream\left(F\right)$ is strongly atomic according to part~(i) of Proposition~\ref{strictstrong}. Also, since $\|\cdot\|_{2}\le\|\cdot\|_{1}$ on $\ell_{1}$, it follows that $\|\cdot\|_{*}\le\sqrt{2}\|\cdot\|_{1}$.

Let $\left(f_{n}\right)_{n\in\N}$ be the standard basis of $\ell_{1}$, and for $n\in\N$ let $e_{n}:=\left(f_{n},0\right)$. Then, $\left(e_{n}\right)_{n\in\N}$ is norm bounded in $\icecream\left(F\right)$, hence order bounded. Let us show that $\left(e_{n}\right)_{n\in\N}$ does not have an order Cauchy subnet.

Assume that $\left(e_{n_{\alpha}}\right)_{\alpha\in A}$ is an order Cauchy subnet, witnessed by $P\subseteq \icecream\left(F\right)$. That is, for every $p\in P$ there is $\alpha_{p}\in A$ such that for every $\beta,\gamma\ge\alpha_{p}$ we have $e_{n_{\beta}}-e_{n_{\gamma}}\le p$. Let us show that $P\ge \left(0_{F},1\right)$, thus immediately reaching a contradiction.

Fix $p\in P$ and let $p=\left(h,r\right)$. Fix $k\in \N$. There is $\beta_{k}\ge\alpha_{p}$ such that $n_{\beta_{k}}\ge k$, and there is $\gamma_{k}\ge\alpha_{p}$ such that $n_{\gamma_{k}}> n_{\beta_{k}}$. Let $g_{k}:=f_{n_{\beta_{k}}}-f_{n_{\gamma_{k}}}$; then $\left(g_{k},0\right)\le p$ yields $\|h- g_{k}\|_{*}\le r$. Since the coordinates of $h$ converge to $0$, it is not hard to show that $\|h- g_{k}\|_{1}\to \|h\|_{1}+2$ and $\|h- g_{k}\|_{2}^{2}\to \|h\|_{2}^{2}+2$. Hence,
\[r^{2}\ge \|h\|_{2}^{2}+2+ \|h\|_{1}^{2}+4+4\|h\|_{1}> \|h\|_{*}^{2}+2\|h\|_{*}+1= \left(\|h\|_{*}+1\right)^{2},\]
implying $r\ge \|h\|_{*}+1$, and so $\left(h,r\right)\ge \left(0_{F},1\right)$, as required.
\qed\end{example}

The last example demonstrates that directed order compactness of order intervals does not imply atomicity.

\begin{example}\label{megaice}
Let $B$ be the Boolean algebra of clopen sets of the Cantor space $\left\{0,1\right\}^{\N}$. Any choice of the first $n$ coordinates gives rise to an element of $B$; we denote the collection of all such elements by $B_{n}$. Clearly, $B_{n}$ consists of $2^{n}$ disjoint nonzero elements which add up to $1_{B}$. Moreover, for every $0_{B}\ne b\in B$ there is $n\in\N$ such that $b$ is a disjoint union of a finite number of elements of $B_{n}$. 

Let $M:=M\left(B\right)$ be the space of finitely additive real-valued measures on $B$. Note that $M$ is an AL-space with respect to the total variation norm. For every $n\in\N$, let $\rho_{n}$ be the seminorm on $M$ defined by $\rho_{n}\left(\mu\right):=\sum_{b\in B_{n}}\left|\mu\left(b\right)\right|$. Note that $\rho_{n}\le \rho_{n+1}$, for every $n\in\N$, and if $\mu\ge 0_{M}$, then $\rho_{n}\left(\mu\right)=\|\mu\|$, for every $n\in\N$.

On $\ell_{\8}\oplus M$ consider the following version of the ice cream cone
\[C:=\left\{\left(f,\mu\right): f\left(n\right)\ge \rho_{n}\left(\mu\right)\text{ for all }n\in\N\right\}.\]
Since $C$ is closed with respect to pointwise topologies on $\ell_{\8}$ and $M$, the cone $C$ is Archimedean. Let $E$ be $\ell_{\8}\oplus M$ ordered by $C$.

For every $\mu\in M$ we have $\left(\|\mu\|\1,\mu\right)\in C$. It follows that $h:=\left(\1,0_{M}\right)$ is an order unit of $E$. Indeed, for any $f\in\ell_{\8}$ and $\mu\in M$ we have $\pm\left(f,\mu\right)\le \left(\|f\|+\|\mu\|\right)h$.

Let us show that order intervals are directed order compact. In fact, since $E$ is an order unit space, it is enough to show this for $\left[0_{E},h\right]$. Let $\left(f_{\alpha},\mu_{\alpha}\right)_{\alpha\in A}\subseteq \left[0_{E},h\right]$. For every $n\in\N$ and $\alpha\in A$ we have $1-f_{\alpha}\left(n\right),f_{\alpha}\left(n\right)\ge \rho_{n}\left(\mu_{\alpha}\right)$; in particular, $0\le f_{\alpha}\left(n\right)\le 1$, and for every $0_{B}\ne c\in B$ there is $n\in\N$ such that $c$ is the disjoint union of a subset of $B_{n}$, hence $\left|\mu_{\alpha}\left(c\right)\right|\le \rho_{n}\left(\mu_{\alpha}\right)\le 1$.

By passing to a subnet, we may therefore assume that there are functions $f\colon \N\to\left[0,1\right]$ and $\mu \colon B\to\left[-1,1\right]$ such that $f_{\alpha}\left(n\right)\to f\left(n\right)$ and $\mu_{\alpha}\left(b\right)\to \mu\left(b\right)$, for every $n\in\N$ and $b\in B$.

By continuity, $\mu$ is a measure, and $1-f\left(n\right),f\left(n\right)\ge \rho_{n}\left(\mu\right)$, for every $n\in\N$, hence $e:=\left(f,\mu\right)\in \left[0_{E},h\right]$. It is left to show that $e_{\alpha}\todo e$. First, note that $e-e_{\alpha}\in \left[-h,h\right]$, for every $\alpha$. Hence, for every $n\in\N$ and $\alpha$ we have $1-\left|f_{\alpha}\left(n\right)-f\left(n\right)\right|\ge \rho_{n}\left(\mu_{\alpha}-\mu\right)$.

For $n\in\N$, let $g_{n}:=\frac{2}{n}\1_{\left[1,n\right]}+\1_{\left[n+1,\8\right)}$ and let $P:=\left\{\left(g_{n},0_{M}\right)\right\}_{n\in\N}$. It is easy to see that $P\downarrow 0_{E}$. For every $n\in\N$ there is $\alpha_{n}$ such that $\left|f_{\alpha}\left(k\right)-f\left(k\right)\right|\le \frac{1}{n}$, and $\left|\mu_{\alpha}\left(b\right)-\mu\left(b\right)\right|\le \frac{1}{2^{k}n}$, for every $\alpha\ge\alpha_{n}$, $k=1,\ldots,n$, and $b\in B_{k}$. Let us show that $\left(g_{n},0_{M}\right)\ge \pm\left(e_{\alpha}-e\right)$, for $\alpha\ge\alpha_{n}$. For $k=1,\ldots, n$, we have
\[g_{n}\left(k\right)-\left|f_{\alpha}\left(k\right)-f\left(k\right)\right|\ge \frac{1}{n}=2^{k}\frac{1}{2^{k}n}\ge \rho_{k}\left(\mu_{\alpha}-\mu\right).\]
For $k>n$, we have
\[g_{n}\left(k\right)-\left|f_{\alpha}\left(k\right)-f\left(k\right)\right|=1-\left|f_{\alpha}\left(k\right)-f\left(k\right)\right|\ge \rho_{k}\left(\mu_{\alpha}-\mu\right).\]
We conclude that $P$ witnesses $e_{\alpha}\todo e$.

We now show that $E$ is not atomic. First, if $\mu,\nu\in M_{+}$ are such that $\mu\le\nu$, then $\left(\|\mu\|\1,\mu\right)\le \left(\|\nu\|\1,\nu\right)$. Indeed, for $\lambda:=\nu-\mu\in M_{+}$ we have $\|\lambda\|=\|\nu\|-\|\mu\|$, and so $\left(\|\nu\|\1,\nu\right)-\left(\|\mu\|\1,\mu\right)=\left(\|\lambda\|\1,\lambda\right)\in C$.

Conversely, assume that $\mu\in M$, $\nu\in M_{+}$ and $f\in\ell_{\8}$ satisfy $0_{E}\le \left(f,\mu\right)\le \left(\|\nu\|\1,\nu\right)$. We show that $0_{M}\le \mu\le \nu$ and $f=\|\mu\|\1$. Without loss of generality, we may assume that $\|\nu\|=1$. Fix $n\in\N$. We have $f\left(n\right)\ge \rho_{n}\left(\mu\right)$ and $1-f\left(n\right)\ge \rho_{n}\left(\nu-\mu\right)$. Hence,
\[1\ge \rho_{n}\left(\mu\right)+\rho_{n}\left(\nu-\mu\right)= \sum\limits_{b\in B_{n}}\left(\left|\mu\left(b\right)\right|+\left|\nu\left(b\right)-\mu\left(b\right)\right|\right)\ge \sum\limits_{b\in B_{n}}\nu\left(b\right)=\nu\left(1_{B}\right)=\|\nu\|=1,\]
and so all present inequalities are in fact equalities. In particular, for every $b\in B_{n}$ we have $\left|\mu\left(b\right)\right|+\left|\nu\left(b\right)-\mu\left(b\right)\right|=\nu\left(b\right)$, hence $0\le \mu\left(b\right)\le \nu\left(b\right)$. We see that $\mu$ and $\nu-\mu$ are positive. Moreover, $f\left(n\right)\ge \rho_{n}\left(\mu\right)=\|\mu\|$, and $1-f\left(n\right)\ge \rho_{n}\left(\nu-\mu\right)=\|\nu-\mu\|=1-\|\mu\|$. We conclude that $f\left(n\right)=\|\mu\|$, hence $f = \|\mu\|\1$.

For every measure $\mu\in M_{+}$ which is not atomic there are $\lambda,\nu\in M_{+}$ which are not multiples of $\mu$ and such that $\mu=\lambda+\nu$. It follows that
\[e:=\left(\|\mu\|\1,\mu\right)=\left(\|\lambda\|\1,\lambda\right)+\left(\|\nu\|\1,\nu\right)\]
is a positive decomposition of $e$ into non-proportional positive summands, and so $e$ is not an atom. Assume further that $\mu$ is disjoint with all atomic measures, for example $\mu$ is the Bernoulli measure, which assigns $\mu\left(b\right):=2^{-n}$, for every $b\in B_{n}$. Suppose that $0_{E}\le g\le e$. Then, by the previous step, we have that $g=\left(\|\lambda\|\1,\lambda\right)$, where $0_{M}\le \lambda\le \mu$. In particular, $\lambda$ is atomless, and so $g$ is not an atom. Thus, there are no atoms in $\left[0_{E},e\right]$.
\qed\end{example}

\subsection*{Acknowledgements}

The authors would like to thank Jochen Gl\"uck for contributing ideas for Proposition~\ref{do=ru} and Example~\ref{jochen}, Kevin Abela for contributing Example~\ref{pat}, and Janko Stennder for bringing the authors' attention to \cite[Theorem~4.12]{gluck}. Furthermore, we are indebted to David Gao for contributing the main idea of Example~\ref{smoothre} through the service \href{https://mathoverflow.com/}{MathOverflow}. Additionally, we credit David Mu\~noz-Lahoz for setting up the interface that enabled interactions with models GPT-5.5 and GPT-5.6 Sol, which generated the ideas for Examples~\ref{monodo}, \ref{nonperv}, \ref{ordconst}, \ref{poly}, \ref{sosdo}, \ref{noocomp} and \ref{megaice}, as well as for Corollary~\ref{antilattice2} and for the proof of the implication (vii)$\Rightarrow$(iii) in Theorem~\ref{jblat}.

The project was initiated during the second and fifth authors' visit to the University of Pretoria. Parts of the work on the project were done during the first, fourth and sixth authors' visit to the Mathematical Institute of Leiden University, and the second and third authors' visit to the University of Alberta. All authors would like to express their gratitude to the Erasmus$+$ ICM Grant for facilitating these productive visits between Leiden University and the University of Pretoria, and between Leiden University and the University of Alberta. They are furthermore grateful to Marcel de Jeu and Vladimir Troitsky for their hospitality. Part of the work was additionally conducted during the first author's visit to the University of Pretoria, which was financially supported by the DSTI-NRF Centre of Excellence in Mathematical and Statistical Sciences (CoE-MaSS), South Africa. Opinions expressed and conclusions arrived at are those of the authors and are not necessarily to be attributed to the CoE-MaSS. The first author was also supported by the Pacific Institute for the Mathematical Sciences.

\end{document}